\documentclass[9pt]{amsart}
\usepackage{amsmath}
\usepackage{amsxtra}
\usepackage{amscd}
\usepackage{amsthm}
\usepackage{amsfonts}
\usepackage{amssymb}
\usepackage{eucal}
\usepackage[all]{xy}
\usepackage{graphicx}
\usepackage[usenames]{color}

\newtheorem{cor}[subsubsection]{Corollary}

\newtheorem{lem}[subsubsection]{Lemma}
\newtheorem{prop}[subsubsection]{Proposition}

\newtheorem{thm}[subsubsection]{Theorem}
\newtheorem{mainthm}[subsubsection]{Main Theorem}

\theoremstyle{remark}
\newtheorem{rem}[subsubsection]{Remark}
\newtheorem{example}[subsubsection]{Example}

\theoremstyle{definition}

\theoremstyle{remark}

\newcommand{\thmref}[1]{Theorem~\ref{#1}}

\newcommand{\secref}[1]{Sect.~\ref{#1}}
\newcommand{\lemref}[1]{Lemma~\ref{#1}}
\newcommand{\propref}[1]{Proposition~\ref{#1}}
\newcommand{\corref}[1]{Corollary~\ref{#1}}

\newcommand{\exref}[1]{Example~\ref{#1}}

\numberwithin{equation}{section}

\newcommand{\nc}{\newcommand}
\nc{\renc}{\renewcommand}
\nc{\ssec}{\subsection}
\nc{\sssec}{\subsubsection}
\nc{\on}{\operatorname}

\nc{\ips}{{\iota_P^{(S)}}}
\nc{\ipms}{{\iota_{P^-}^{(S)}}}
\nc{\sfpps}{{\sfp_P^{(S)}}}
\nc{\sfppms}{{\sfp_{P^-}^{(S)}}}

\nc\ol{\overline}
\nc\ul{\underline}
\nc\wt{\widetilde}
\nc\tboxtimes{\wt{\boxtimes}}
\nc\tstar{\wt{\star}}
\nc{\alp}{\alpha}

\nc{\ZZ}{{\mathbb Z}}
\nc{\NN}{{\mathbb N}}
\nc{\OO}{{\mathbb O}}
\renc{\SS}{{\mathbb S}}
\nc{\DD}{{\mathbb D}}
\nc{\GG}{{\mathbb G}}

\nc{\Fq}{{\mathbb F}_q}
\nc{\Fqb}{\ol{\mathbb F}_q}
\nc{\Ql}{{\mathbb Q}_\ell}
\nc{\Qlb}{{\ol{\mathbb Q}_\ell}}
\nc{\id}{\text{id}}
\nc\X{\mathcal X}

\nc{\red}{\on{red}}
\nc{\Ho}{\on{Ho}}
\nc{\Hom}{\on{Hom}}
\nc{\coHom}{\ul{\on{coHom}}}
\nc{\coMaps}{{\bf{coMaps}}}
\nc{\coef}{\on{coef}}
\nc{\Lie}{\on{Lie}}
\nc{\Loc}{\on{Loc}}
\nc{\coLoc}{\on{coLoc}}
\nc{\Pic}{\on{Pic}}
\nc{\Bun}{\on{Bun}}
\nc{\IC}{\on{IC}}
\nc{\Aut}{\on{Aut}}
\nc{\rk}{\on{rk}}
\nc{\Sh}{\on{Sh}}
\nc{\Perv}{\on{Perv}}
\nc{\pos}{{\on{pos}}}
\nc{\Conv}{\on{Conv}}
\nc{\Sph}{\on{Sph}}
\nc{\Sym}{\on{Sym}}
\nc{\BunBb}{\overline{\Bun}_B}
\nc{\BunNb}{\overline{\Bun}_N}
\nc{\BunTb}{\overline{\Bun}_T}
\nc{\BunBbm}{\overline{\Bun}_{B^-}}
\nc{\BunBbel}{\overline{\Bun}_{B,el}}
\nc{\BunBbmel}{\overline{\Bun}_{B^-,el}}
\nc{\Buno}{\overset{o}{\Bun}}
\nc{\BunPb}{{\overline{\Bun}_P}}
\nc{\BunBM}{\Bun_{B(M)}}
\nc{\BunBMb}{\overline{\Bun}_{B(M)}}
\nc{\BunPbw}{{\widetilde{\Bun}_P}}
\nc{\BunBP}{\widetilde{\Bun}_{B,P}}
\nc{\GUb}{\overline{G/U}}
\nc{\GUPb}{\overline{G/U(P)}}

\nc{\Hhom}{\underline{\on{Hom}}}
\nc\syminfty{\on{Sym}^{\infty}}
\nc\lal{\ol{\lambda}}
\nc\xl{\ol{x}}
\nc\thl{\ol{\theta}}
\nc\nul{\ol{\nu}}
\nc\mul{\ol{\mu}}
\nc\Sum\Sigma
\nc{\oX}{\overset{o}{X}{}}
\nc{\hl}{\overset{\leftarrow}h{}}
\nc{\hr}{\overset{\rightarrow}h{}}
\nc{\M}{{\mathcal M}}
\nc{\N}{{\mathcal N}}
\nc{\F}{{\mathcal F}}
\nc{\D}{{\mathcal D}}
\nc{\Q}{{\mathcal Q}}
\nc{\Y}{{\mathcal Y}}
\nc{\G}{{\mathcal G}}
\nc{\E}{{\mathcal E}}
\nc{\CalC}{{\mathcal C}}
\nc\Dh{\widehat{\D}}

\nc{\C}{{\mathcal C}}
\nc{\K}{{\mathcal K}}
\renewcommand{\H}{{\mathcal H}}

\nc{\T}{{\mathcal T}}
\nc{\V}{{\mathcal V}}
\renc{\P}{{\mathcal P}}
\nc{\A}{{\mathcal A}}
\nc{\B}{{\mathcal B}}
\nc{\U}{{\mathcal U}}

\nc{\Gr}{{\on{Gr}}}

\nc{\frn}{{\check{\mathfrak u}(P)}}

\nc{\fC}{\mathfrak C}
\nc{\fT}{\mathfrak T}
\nc{\p}{\mathfrak p}
\nc{\q}{\mathfrak q}
\nc\f{{\mathfrak f}}

\nc{\qo}{{\mathfrak q}}
\nc{\po}{{\mathfrak p}}
\nc{\s}{{\mathfrak s}}
\nc\w{\text{w}}

\nc\mathi\iota
\nc\Spec{\on{Spec}}
\nc\Proj{\on{Proj}}
\nc\Mod{\on{Mod}}
\nc{\tw}{\widetilde{\mathfrak t}}
\nc{\pw}{\widetilde{\mathfrak p}}
\nc{\qw}{\widetilde{\mathfrak q}}
\nc{\jw}{\widetilde j}

\nc{\grb}{\overline{\Gr}}
\nc{\I}{\mathcal I}

\nc{\lambdach}{{\check\lambda}}
\nc{\Lambdach}{{\check\Lambda}{}}
\nc{\much}{{\check\mu}}
\nc{\omegach}{{\check\omega}}
\nc{\nuch}{{\check\nu}}
\nc{\etach}{{\check\eta}}
\nc{\alphach}{{\check\alpha}}
\nc{\oblvtach}{{\check\oblvta}}
\nc{\rhoch}{{\check\rho}}
\nc{\ch}{{\check h}}

\nc{\Hb}{\overline{\H}}

\nc{\BA}{{\mathbb{A}}}
\nc{\BC}{{\mathbb{C}}}
\nc{\BE}{{\mathbb{E}}}
\nc{\BF}{{\mathbb{F}}}
\nc{\BG}{{\mathbb{G}}}
\nc{\BL}{{\mathbb{L}}}
\nc{\BM}{{\mathbb{M}}}
\nc{\BO}{{\mathbb{O}}}
\nc{\BD}{{\mathbb{D}}}
\nc{\BN}{{\mathbb{N}}}
\nc{\BP}{{\mathbb{P}}}
\nc{\BQ}{{\mathbb{Q}}}
\nc{\BR}{{\mathbb{R}}}
\nc{\BZ}{{\mathbb{Z}}}
\nc{\BS}{{\mathbb{S}}}
\nc{\Deep}{{\bf{deep}}}
\nc{\deep}{deep}

\nc{\CA}{{\mathcal{A}}}
\nc{\CB}{{\mathcal{B}}}

\nc{\CE}{{\mathcal{E}}}
\nc{\CF}{{\mathcal{F}}}
\nc{\CH}{{\mathcal{H}}}

\nc{\CL}{{\mathcal{L}}}
\nc{\CC}{{\mathcal{C}}}
\nc{\CG}{{\mathcal{G}}}
\nc{\CalD}{{\mathcal{D}}}
\nc{\CM}{{\mathcal{M}}}
\nc{\CN}{{\mathcal{N}}}
\nc{\CK}{{\mathcal{K}}}
\nc{\CO}{{\mathcal{O}}}
\nc{\CP}{{\mathcal{P}}}
\nc{\CQ}{{\mathcal{Q}}}
\nc{\CR}{{\mathcal{R}}}
\nc{\CS}{{\mathcal{S}}}
\nc{\CT}{{\mathcal{T}}}
\nc{\CU}{{\mathcal{U}}}
\nc{\CV}{{\mathcal{V}}}
\nc{\CW}{{\mathcal{W}}}
\nc{\CX}{{\mathcal{X}}}
\nc{\CY}{{\mathcal{Y}}}
\nc{\CZ}{{\mathcal{Z}}}
\nc{\CI}{{\mathcal{I}}}

\nc{\csM}{{\check{\mathcal A}}{}}
\nc{\oM}{{\overset{\circ}{\mathcal M}}{}}
\nc{\obM}{{\overset{\circ}{\mathbf M}}{}}
\nc{\oCA}{{\overset{\circ}{\mathcal A}}{}}
\nc{\obA}{{\overset{\circ}{\mathbf A}}{}}
\nc{\ooM}{{\overset{\circ}{M}}{}}
\nc{\osM}{{\overset{\circ}{\mathsf M}}{}}
\nc{\vM}{{\overset{\bullet}{\mathcal M}}{}}
\nc{\nM}{{\underset{\bullet}{\mathcal M}}{}}
\nc{\oD}{{\overset{\circ}{\mathcal D}}{}}
\nc{\obD}{{\overset{\circ}{\mathbf D}}{}}
\nc{\oA}{{\overset{\circ}{A}}{}}
\nc{\op}{{\overset{\bullet}{\mathbf p}}{}}
\nc{\cp}{{\overset{\circ}{\mathbf p}}{}}
\nc{\oU}{{\overset{\bullet}{\mathcal U}}{}}
\nc{\oZ}{{\overset{\circ}{\mathcal Z}}{}}
\nc{\ofZ}{{\overset{\circ}{\mathfrak Z}}{}}
\nc{\oF}{{\overset{\circ}{\fF}}}

\nc{\fa}{{\mathfrak{a}}}
\nc{\ofa}{\overset{\circ}{\mathfrak{a}}}
\nc{\fb}{{\mathfrak{b}}}
\nc{\fd}{{\mathfrak{d}}}
\nc{\ff}{{\mathfrak{f}}}
\nc{\fg}{{\mathfrak{g}}}
\nc{\fgl}{{\mathfrak{gl}}}
\nc{\fh}{{\mathfrak{h}}}
\nc{\fj}{{\mathfrak{j}}}
\nc{\fl}{{\mathfrak{l}}}
\nc{\fm}{{\mathfrak{m}}}
\nc{\ofm}{\overset{\circ}{\mathfrak{m}}}
\nc{\fn}{{\mathfrak{n}}}
\nc{\fu}{{\mathfrak{u}}}
\nc{\fp}{{\mathfrak{p}}}
\nc{\fr}{{\mathfrak{r}}}
\nc{\fs}{{\mathfrak{s}}}
\nc{\ft}{{\mathfrak{t}}}
\nc{\oft}{\overset{\circ}{\mathfrak{t}}}
\nc{\fz}{{\mathfrak{z}}}
\nc{\fsl}{{\mathfrak{sl}}}
\nc{\hsl}{{\widehat{\mathfrak{sl}}}}
\nc{\hgl}{{\widehat{\mathfrak{gl}}}}
\nc{\hg}{{\widehat{\mathfrak{g}}}}
\nc{\chg}{{\widehat{\mathfrak{g}}}{}^\vee}
\nc{\hn}{{\widehat{\mathfrak{n}}}}
\nc{\chn}{{\widehat{\mathfrak{n}}}{}^\vee}

\nc{\fA}{{\mathfrak{A}}}
\nc{\fB}{{\mathfrak{B}}}
\nc{\fD}{{\mathfrak{D}}}
\nc{\fE}{{\mathfrak{E}}}
\nc{\fF}{{\mathfrak{F}}}
\nc{\fG}{{\mathfrak{G}}}
\nc{\fK}{{\mathfrak{K}}}
\nc{\fL}{{\mathfrak{L}}}
\nc{\fM}{{\mathfrak{M}}}
\nc{\fN}{{\mathfrak{N}}}
\nc{\fP}{{\mathfrak{P}}}
\nc{\fU}{{\mathfrak{U}}}
\nc{\fV}{{\mathfrak{V}}}
\nc{\fZ}{{\mathfrak{Z}}}

\nc{\ba}{{\mathbf{a}}}
\nc{\bb}{{\mathbf{b}}}
\nc{\bc}{{\mathbf{c}}}
\nc{\bd}{{\mathbf{d}}}
\nc{\bbf}{{\mathbf{f}}}
\nc{\be}{{\mathbf{e}}}
\nc{\bi}{{\mathbf{i}}}
\nc{\bj}{{\mathbf{j}}}
\nc{\bh}{{\mathbf{h}}}
\nc{\bm}{{\mathbf{m}}}
\nc{\bn}{{\mathbf{n}}}
\nc{\bo}{{\mathbf{o}}}
\nc{\bp}{{\mathbf{p}}}
\nc{\bq}{{\mathbf{q}}}
\nc{\bu}{{\mathbf{u}}}
\nc{\bv}{{\mathbf{v}}}
\nc{\bx}{{\mathbf{x}}}
\nc{\bs}{{\mathbf{s}}}
\nc{\by}{{\mathbf{y}}}
\nc{\bw}{{\mathbf{w}}}
\nc{\bA}{{\mathbf{A}}}
\nc{\bK}{{\mathbf{K}}}
\nc{\bB}{{\mathbf{B}}}
\nc{\bC}{{\mathbf{C}}}
\nc{\bG}{{\mathbf{G}}}
\nc{\bD}{{\mathbf{D}}}
\nc{\bE}{{\mathbf{E}}}
\nc{\bH}{{{\mathbf{H}}}}
\nc{\bM}{{\mathbf{M}}}
\nc{\bN}{{\mathbf{N}}}
\nc{\bO}{{\mathbf{O}}}
\nc{\bQ}{{\mathbf{Q}}}
\nc{\bV}{{\mathbf{V}}}
\nc{\bW}{{\mathbf{W}}}
\nc{\bX}{{\mathbf{X}}}
\nc{\bZ}{{\mathbf{Z}}}
\nc{\bS}{{\mathbf{S}}}

\nc{\sA}{{\mathsf{A}}}
\nc{\sB}{{\mathsf{B}}}
\nc{\sC}{{\mathsf{C}}}
\nc{\sD}{{\mathsf{D}}}
\nc{\sF}{{\mathsf{F}}}
\nc{\sG}{{\mathsf{G}}}
\nc{\sH}{{\mathsf{H}}}
\nc{\sK}{{\mathsf{K}}}
\nc{\sM}{{\mathsf{M}}}
\nc{\sN}{{\mathsf{N}}}
\nc{\sO}{{\mathsf{O}}}
\nc{\sW}{{\mathsf{W}}}
\nc{\sQ}{{\mathsf{Q}}}
\nc{\sP}{{\mathsf{P}}}
\nc{\sR}{{\mathsf{R}}}
\nc{\sT}{{\mathsf{T}}}
\nc{\sZ}{{\mathsf{Z}}}
\nc{\sfp}{{\mathsf{p}}}
\nc{\sfq}{{\mathsf{q}}}
\nc{\sft}{{\mathsf{t}}}
\nc{\sr}{{\mathsf{r}}}
\nc{\bk}{{\mathsf{k}}}
\nc{\sa}{{\mathsf{s}}}
\nc{\sg}{{\mathsf{g}}}
\nc{\sn}{{\mathsf{n}}}
\nc{\sh}{{\mathsf{h}}}
\nc{\sff}{{\mathsf{f}}}
\nc{\sfb}{{\mathsf{b}}}
\nc{\sfc}{{\mathsf{c}}}
\nc{\sfe}{{\mathsf{e}}}
\nc{\sd}{{\mathsf{d}}}

\nc{\BK}{{\bar{K}}}

\nc{\tA}{{\widetilde{\mathbf{A}}}}
\nc{\tB}{{\widetilde{\mathcal{B}}}}
\nc{\tg}{{\widetilde{\mathfrak{g}}}}
\nc{\tG}{{\widetilde{G}}}
\nc{\TM}{{\widetilde{\mathbb{M}}}{}}
\nc{\tO}{{\widetilde{\mathsf{O}}}{}}
\nc{\tU}{{\widetilde{\mathfrak{U}}}{}}
\nc{\TZ}{{\tilde{Z}}}
\nc{\tx}{{\tilde{x}}}
\nc{\tbv}{{\tilde{\bv}}}
\nc{\tfP}{{\widetilde{\mathfrak{P}}}{}}
\nc{\tz}{{\tilde{\zeta}}}
\nc{\tmu}{{\tilde{\mu}}}

\nc{\urho}{\underline{\rho}}
\nc{\uB}{\underline{B}}
\nc{\uC}{{\underline{\mathbb{C}}}}
\nc{\ui}{\underline{i}}
\nc{\uj}{\underline{j}}
\nc{\ofP}{{\overline{\mathfrak{P}}}}
\nc{\oB}{{\overline{\mathcal{B}}}}
\nc{\og}{{\overline{\mathfrak{g}}}}
\nc{\oI}{{\overline{I}}}

\nc{\eps}{\varepsilon}
\nc{\hrho}{{\hat{\rho}}}

\nc{\one}{{\mathbf{1}}}
\nc{\two}{{\mathbf{t}}}

\nc{\Rep}{{\mathop{\operatorname{\rm Rep}}}}
\nc{\Tot}{{\mathop{\operatorname{\rm Tot}}}}
\nc{\Ker}{{\mathop{\operatorname{\rm Ker}}}}
\nc{\im}{{\mathop{\operatorname{\rm Im}}}}
\nc{\Hilb}{{\mathop{\operatorname{\rm Hilb}}}}
\nc{\End}{{\mathop{\operatorname{\rm End}}}}
\nc{\Ext}{{\mathop{\operatorname{\rm Ext}}}}
\nc{\CHom}{{\mathop{\operatorname{{\mathcal{H}}\it om}}}}
\nc{\CEnd}{{\mathop{\operatorname{{\mathcal{E}}\it nd}}}}
\nc{\GL}{{\mathop{\operatorname{\rm GL}}}}
\nc{\gr}{{\mathop{\operatorname{\rm gr}}}}
\nc{\HN}{{\mathop{\operatorname{\rm HN}}}}
\nc{\Id}{{\mathop{\operatorname{\rm Id}}}}
\nc{\de}{{\mathop{\operatorname{\rm def}}}}
\nc{\length}{{\mathop{\operatorname{\rm length}}}}
\nc{\supp}{{\mathop{\operatorname{\rm supp}}}}

\nc{\Cliff}{{\mathsf{Cliff}}}
\nc{\Fl}{\on{Fl}}
\nc{\Fib}{{\mathsf{Fib}}}
\nc{\Coh}{{\on{Coh}}}
\nc{\QCoh}{{\on{QCoh}}}
\nc{\IndCoh}{{\on{IndCoh}}}
\nc{\FCoh}{{\mathsf{FCoh}}}

\nc{\reg}{{\text{\rm reg}}}

\nc{\cplus}{{\mathbf{C}_+}}
\nc{\cminus}{{\mathbf{C}_-}}
\nc{\cthree}{{\mathbf{C}_\bullet}}
\nc{\Qbar}{{\bar{Q}}}
\nc\Eis{\on{Eis}}
\nc\Eisb{\ol\Eis{}}
\nc\Eisr{\on{Eis}^{rat}{}}
\nc\wh{\widehat}
\nc{\Def}{\on{Def_{\check{\fb}}(E)}}
\nc{\barZ}{\overline{Z}{}}
\nc{\barbarZ}{\overline{\barZ}{}}
\nc{\barpi}{\overline\pi}
\nc{\barbarpi}{\overline\barpi}
\nc{\barpip}{\overline\pi{}^+}
\nc{\barpim}{\overline\pi{}^-}

\nc{\fq}{\mathfrak q}

\nc{\fqb}{\ol{\sfq}{}}
\nc{\fpb}{\ol{\sfp}{}}
\nc{\fpr}{{\mathsf{pair}^{rat}}{}}
\nc{\fqr}{{\sfq^{rat}}{}}

\nc{\hattimes}{\wh\otimes}

\nc{\bOmega}{{\overline{\Omega(\check \fn)}}}

\nc{\seq}[1]{\stackrel{#1}{\sim}}

\nc{\cT}{{\check{T}}}
\nc{\cG}{{\check{G}}}
\nc{\cM}{{\check{M}}}
\nc{\cB}{{\check{B}}}

\nc{\ct}{{\check{\mathfrak t}}}
\nc{\cg}{{\check{\fg}}}
\nc{\cb}{{\check{\fb}}}
\nc{\cn}{{\check{\fn}}}

\nc{\cLambda}{{\check\Lambda}}

\nc{\cla}{{\check\lambda}}
\nc{\cmu}{{\check\mu}}
\nc{\cnu}{{\check\nu}}
\nc{\ceta}{{\check\eta}}

\nc{\DefbE}{{\on{Def}_{\cB}(E_\cT)}}

\nc{\imathb}{{\ol{\imath}}}
\nc{\rlr}{\overset{\longrightarrow}{\underset{\longrightarrow}\longleftarrow}}

\nc{\oBun}{\overset{\circ}\Bun}
\nc{\LocSys}{\on{LocSys}}
\nc{\BunBbb}{\ol{\ol{Bun}}_B}
\nc{\BunBr}{\Bun_B^{rat}}
\nc{\BunBrsg}{\Bun_B^{rat,\on{s.g.}}}
\nc{\BunBrp}{\Bun_B^{rat,polar}}
\nc{\BunBrpbg}{\Bun_B^{rat,polar,\on{b.g.}}}
\nc{\BunBrpsg}{\Bun_B^{rat,polar,\on{s.g.}}}
\nc{\BunTrp}{\Bun_T^{rat,polar}}
\nc{\BunTrpbg}{\Bun_T^{rat,polar,\on{b.g.}}}
\nc{\BunTrpsg}{\Bun_T^{rat,polar,\on{s.g.}}}
\nc{\BunNr}{\Bun_N^{rat}}
\nc{\BunNre}{\Bun_N^{enh,rat}}
\nc{\BunTr}{\Bun_T^{rat}}
\nc{\Vect}{\on{Vect}}
\nc{\Whit}{\on{Whit}}
\nc{\CTb}{\ol{\on{CT}}}
\nc{\Ran}{{\on{Ran}}}
\nc{\CTr}{\on{CT}^{rat}{}}
\nc\jmathr{\jmath^{rat}{}}
\nc{\ux}{\underline{x}}
\nc{\clambda}{{\check\lambda}}
\nc{\calpha}{{\check\alpha}}
\nc{\ind}{{\mathbf{ind}}}
\nc{\coinv}{{\mathbf{coinv}}}
\nc{\oblv}{{\mathbf{oblv}}}
\nc{\free}{{\mathbf{free}}}
\nc{\ox}{{\overline{x}}}
\nc{\cLa}{\check{\Lambda}}
\nc{\StinftyCat}{\on{DGCat}}
\nc{\inftyCat}{\infty\on{-Cat}}
\nc{\inftygroup}{\infty\on{-Grpd}}
\nc{\Dmod}{\on{D-mod}}
\nc{\CMaps}{{\mathcal Maps}}
\nc{\Maps}{\on{Maps}}
\nc{\affSch}{\on{Sch}^{\on{aff}}}
\nc{\dr}{{\on{dR}}}
\nc{\oCF}{\overset{\circ}\CF}
\nc{\oCY}{\overset{\circ}\CY}
\nc{\opi}{\overset{\circ}\pi}
\nc{\leqG}{\underset{G}\leq}
\nc{\leqM}{\underset{M}\leq}
\nc{\leqGad}{\underset{G_{ad}}\leq}
\nc{\leqMad}{\underset{M_{ad}}\leq}
\nc{\Tr}{\on{Tr}}
\nc{\Frob}{{\on{Frob}}}
\nc{\DGCat}{\on{DGCat}}
\nc{\tDGCat}{2\on{-DGCat}_{\on{u.g.}}}
\nc{\ev}{\on{ev}}
\nc{\mmod}{\on{-}\mathbf{mod}}
\nc{\sotimes}{\overset{!}\otimes}
\nc{\Shv}{\on{Shv}}
\nc{\Spc}{\on{Spc}}
\nc{\LS}{\Lisse}
\nc{\Res}{\on{Res}}
\nc{\bDelta}{{\mathbf{\Delta}}}
\nc{\bMaps}{{\mathbf{Maps}}}
\nc{\cD}{\mathcal D}
\nc{\ocD}{\overset{\circ}\cD}
\nc{\ppart}{(\!(t)\!)}
\nc{\qqart}{[\![t]\!]}
\nc{\oCU}{\overset{\circ}{\CU}}
\nc{\Exc}{{\mathcal{E}xc}}
\nc{\Sht}{\on{Sht}}
\nc{\Nilp}{{\on{Nilp}}}
\nc{\Drinf}{\on{Drinf}}
\nc{\Sing}{\on{Sing}}
\nc{\IndLisse}{\Lisse}
\nc{\Shvl}{\on{Shv}_{\on{lisse}}} 
\nc{\Lisse}{\on{Lisse}}
\nc{\Mir}{\on{Mir}}
\nc{\fSet}{\on{fSet}}
\nc{\qLisse}{\on{QLisse}}
\nc{\Ev}{\on{Ev}}
\nc{\Sat}{\on{Sat}}
\nc{\Se}{\on{Se}}
\nc{\coSht}{\on{co-Sht}}
\nc{\coCK}{\on{co-}\!\CK}

\begin{document}

\title{Automorphic functions as the trace of Frobenius}


\author{D.~Arinkin, D.~Gaitsgory, D.~Kazhdan, 
S.~Raskin, N.~Rozenblyum, Y.~Varshavsky}

\date{\today}

\begin{abstract}
We prove that the trace of the Frobenius endofunctor of 
the category of automorphic sheaves with nilpotent singular support is isomorphic
to the space of unramified automorphic functions, settling a conjecture
from \cite{AGKRRV1}. More generally, we show that traces of 
Frobenius-Hecke functors produce shtuka cohomologies.
\end{abstract} 


\maketitle

\tableofcontents

\section*{Introduction}

\ssec{Overview} 

This work is part of a series, following \cite{AGKRRV1} and \cite{AGKRRV2},
attempting to understand the (unramified, function field) 
arithmetic Langlands conjectures via geometric and categorical techniques. 
We begin with an overview of the problems considered here.

\sssec{}

Some of the most striking applications of geometric representation theory pass through
the sheaves-functions dictionary of Grothendieck--Deligne. 

\medskip 

Recall the setting: one has an algebraic stack $\CY$ over $\ol\BF_q$, assumed to be defined 
over $\BF_q$, and an $\ell$-adic Weil sheaf $\CF$ on $\CY$. For a rational point 
$y \in \CY_0(\BF_q)$, one takes the trace of Frobenius on the stalk 
$y^*(\CF)$ to obtain an element of $\ol\BQ_\ell$; this defines a function
$\on{funct}(\CF):\CY_0(\BF_q) \to \ol\BQ_\ell$.

\medskip 

One finds that many functions of interest in harmonic analysis over
finite fields arise by this procedure, and that the perspective offered by sheaf
theory provides deep insights into function theory.

\medskip 

For example, this is the case in the theory of automorphic functions 
(for function fields), whose realizations via $\ell$-adic sheaves exhibit
explicit constructions of Langlands's conjectures.

\sssec{} 

In this paper, we establish a categorical analogue of 
the sheaves-functions dictionary: instead of passing from 
sheaves to functions, we pass from categories (of sheaves) 
to vector spaces (of functions). 

\medskip 

As in the previous setting, we decategorify using trace of Frobenius.
However, whereas before we considered the trace of a Frobenius endomorphism
of a vector space (and thus produced a scalar), we now consider the trace 
of a Frobenius endofunctor of a category (and thus produce a vector space). 

\medskip 

Unlike the usual Grothendieck--Deligne paradigm, where one may take a general 
algebraic stack $\CY$, our results are specialized to spherical automorphic
functions. The results of this paper establish conjectures from \cite{AGKRRV1}, 
which specify a relation between the category of automorphic sheaves and the
vector space of automorphic functions via the trace of Frobenius. 

\sssec{}

We give a precise formulation of our main results below. 
However, by way of motivation, we highlight the following application, 
which illustrates how sheaf-theoretic considerations provide insights
into the classical theory of automorphic functions. 

\medskip 

In \cite{AGKRRV1}, we introduced a version of the geometric Langlands
conjecture suitable for $\ell$-adic sheaves. From this conjecture, combined
with the \emph{Trace Conjecture}, proved in this paper, 
we deduced that the space of compactly supported spherical automorphic functions
(denoted $\on{Funct}_c(\Bun_G(\BF_q))$ in the body of this text) 
can be described as
\begin{equation} \label{e:function conj intro}
\on{Funct}_c(\Bun_G(\BF_q))\simeq \Gamma(\LocSys^{\on{arthm}}_\cG(X),\omega_{\LocSys^{\on{arthm}}_\cG(X)}).
\end{equation}

In the above formula, $X$ is the smooth projective curve corresponding to our choice of 
function field and 
$\LocSys^{\on{arthm}}_\cG(X)$ is the algebraic stack over 
$\ol{\BQ}_\ell$, defined in \cite{AGKRRV1},
parametrizing (unramified) Langlands parameters.\footnote{An alternative construction 
of $\LocSys^{\on{arthm}}_\cG(X)$ due to P.~Scholze and X.~Zhu, may be 
found in \cite{Zhu}. Their construction proceeds along very different lines,
but conjecturally produces an equivalent object.} 

\medskip 

In the everywhere unramified function field setting, this conjecture provides an interesting 
alternative to Langlands's original perspective: for general reductive $G$, 
the above conjecture yields a compete description of the space of automorphic functions in terms
of Galois data, whereas classical Langlands conjectures only concern $L$-packets.

\begin{rem}

The work \cite{VLaf} of V.~Lafforgue and its extension \cite{Xue1} by C.~Xue
provide a decomposition of the space of (not necessarily unramified!) 
automorphic functions in terms of Langlands parameters. It should come as no 
surprise that our results are closely related to their work. 

\medskip 

First, our work is also based on considering cohomologies of shtukas. 

\medskip 

And second, our constructions show that $\on{Funct}_c(\Bun_G(\BF_q))$
arises as global sections of \emph{some} quasi-coherent sheaf $\Drinf^{\on{arithm}}$ on 
$\LocSys^{\on{arthm}}_\cG(X)$. (The conjecture expressed
by formula \eqref{e:function conj intro} says that $\Drinf^{\on{arithm}}$ is the dualizing sheaf of $\LocSys^{\on{arthm}}_\cG(X)$.)
The existence of $\Drinf^{\on{arithm}}$ recovers the spectral decomposition of
$\on{Funct}_c(\Bun_G(\BF_q))$ along the set of classical Langlands parameters,
see Remark \ref{r:Vinc decomp}. 

\medskip 

We refer to \cite[Sect. 24]{AGKRRV1} for further discussion of the relation to V.~Lafforgue's work.

\end{rem}

\begin{rem}

The principle that geometric methods enrich Langlands's conjectures is an old one,
dating (at least) to \cite{De} and \cite{Dr}. But precise refinements of 
Langlands's conjectures using geometric ideas have emerged recently. 
The present paper provides one example. 
Similarly, the Fargues--Scholze geometrization program (see \cite{FS}) aims
to produce a more robust form of the local Langlands conjectures. 

\end{rem}

\ssec{Formulation of the main result}

We now proceed to the statement of our main result.

\sssec{Notation} \label{sss:From intro}

Throughout the paper, we use algebraic geometry over the two fields 
$k := \ol{\BF}_q$ and $\sfe := \ol{\BQ}_\ell$ (where $\ell \in \BF_q^{\times}$).

\medskip 

When we work over $\ol{\BF}_q$, we generally work with algebraic stacks
$\CY$ that are assumed to be defined over $\BF_q$; we abuse notation somewhat
in letting $\CY(\BF_q)$ denote the groupoid of $\BF_q$-points of the corresponding stack.
We let $\Frob_{\CY}:\CY \to \CY$ denote the geometric ($q$-)Frobenius morphism,
whose stack of fixed-points $\CY^{\Frob}$ is a \emph{discrete} stack, and identifies with 
(the \'etale sheafification of) the groupoid $\CY(\BF_q)$.

\medskip 

Let $X$ be a smooth projective curve 
over $\ol\BF_q$, but assumed defined over $\BF_q$. 
Let $G/\ol\BF_q$ be a reductive group, considered over $\BF_q$ via its split form. 
Let $\Bun_G$ denote the moduli stack of principal $G$-bundles on $X$.

\medskip 

We refer to \secref{ss:notation} for further details on our conventions.

\sssec{Categorical trace}

Throughout this paper, all DG categories are enriched
over the field $\sfe$. In particular, $\Vect = \Vect_\sfe$
(the DG category of chain complexes of $\sfe$-vector spaces), and so on.

\medskip 

We remind that the category $\DGCat$ of presentable DG categories 
is equipped with a canonical
symmetric monoidal structure $\otimes$, the \emph{Lurie tensor product}.
On general grounds, this means one may speak about categorical duals and traces
as follows.

\medskip 

If $\bC \in \DGCat$ is dualizable, there is another DG category 
$\bC^{\vee}$ equipped with canonical unit and counit maps
$$\on{u}_\bC:\Vect \to \bC \otimes \bC^{\vee}$$
and
$$\on{ev}_\bC:\bC \otimes \bC^{\vee} \to \Vect.$$
For an endofunctor $\Phi:\bC \to \bC$ of $\bC$, we have
$\on{Tr}(\Phi,\bC) \in \Vect$ defined as 
$$\on{Tr}(\Phi,\bC) := \on{ev}_\bC\big((\Phi\otimes \on{Id})(\on{u}_\bC)\big).$$

We refer to \cite{GKRV} for further discussion. 

\sssec{Categories of sheaves}

We consider the category $\Shv(\Bun_G)$ of \emph{automorphic sheaves}. 
Precisely, $\Shv(\Bun_G)$ is the category of ind-constructible 
$\ol{\BQ}_\ell$-sheaves on $\Bun_G$.
As in \cite{AGKRRV1}, we also consider its full DG subcategory
\begin{equation} \label{e:Nilp intro}
\Shv_\Nilp(\Bun_G)\subset \Shv(\Bun_G)
\end{equation} 
of objects with singular support in the global nilpotent cone.

\medskip 

The categories $\Shv(\Bun_G)$ and $\Shv_\Nilp(\Bun_G)$ have favorable
finiteness properties: they are 
compactly generated and therefore dualizable in $\DGCat$.
Moreover, if we assume \cite[Conjecture 14.1.8]{AGKRRV1}, 
the embedding $\Shv_\Nilp(\Bun_G)\to \Shv(\Bun_G)$ preserves
compactness.

\sssec{}

Pushforward with respect to the geometric Frobenius endomorphism (see \secref{sss:Frob})
defines an auto-equivalence 
$(\Frob_{\Bun_G})_*$ of $\Shv(\Bun_G)$, which preserves the subcategory $\Shv_\Nilp(\Bun_G)$.

\medskip

Hence, it makes sense to consider the categorical trace
$$\Tr((\Frob_{\Bun_G})_*,\Shv_\Nilp(\Bun_G)) \in \Vect.$$ 

\sssec{}

At this point we can state one of three main results of this paper (it appears as \corref{c:trace} in the main body of the paper):

\begin{mainthm} \label{t:Tr conj intro}
There exists \emph{a} canonical isomorphism 
$$\Tr((\Frob_{\Bun_G})_*,\Shv_\Nilp(\Bun_G)) \simeq \on{Funct}_c(\Bun_G(\BF_q)).$$
\end{mainthm}

This statement appears in \cite{AGKRRV1} as Conjecture 22.3.7(a).

\ssec{Insertion of Hecke functors}

So far, we have not mentioned Langlands duality. Remarkably, 
it plays an essential role in the proof of \thmref{t:Tr conj intro}. 

\medskip 

In fact, our proof of \thmref{t:Tr conj intro} 
constructs isomorphisms between two families 
of objects, which are, roughly speaking, indexed by Hecke functors. 
As we explain below, this amounts to a proof of \cite[Conjecture 22.5.7]{AGKRRV1},
referred to as the \emph{Shtuka Conjecture} of \emph{loc. cit.}

\medskip 

Moreover, our method of proof relies in an essential way on consideration
of \emph{all} Hecke functors simultaneously. 

\medskip 

We presently explain what these objects are, and what our results about them assert.

\sssec{}

Let $\cG$ denote the Langlands dual group of $G$ 
considered as an algebraic group over the coefficient field
$\sfe:=\ol\BQ_\ell$.
We let $\Rep(\cG)$ denote the (symmetric monoidal) DG category
of representations of $\cG$. 

\sssec{}

First, recall the formalism of Hecke functors in geometric Langlands.
The key feature is that they may be indexed by tuples of moving, possibly colliding
points.

\medskip 

Precisely, given 
a finite set $I$ and an object $V\in \Rep(\cG)^{\otimes I}$, we can consider the
Hecke functor
\begin{equation}\label{eq:hecke intro}
\sH(V,-):\Shv(\Bun_G)\to \Shv(\Bun_G\times X^I).
\end{equation}
Restricting to $\Shv_\Nilp(\Bun_G)\subset \Shv(\Bun_G)$, we obtain a functor that lands in the subcategory
$$\Shv_\Nilp(\Bun_G) \otimes \qLisse(X)^{\otimes I}\subset \Shv(\Bun_G\times X^I)$$
(see \corref{c:Hecke on Nilp}), where $\qLisse(X)$ denotes the category of lisse sheaves on $X$, 
see \secref{sss:sheaves}. 

\sssec{}

As we recall in \secref{sss:Rep Ran}, there is a symmetric monoidal category 
$\Rep(\cG)_\Ran$, the \emph{Ran version} of the category $\Rep(\cG)$, 
which is the universal source of Hecke functors. 

\medskip

More precisely, there is an action of $\Rep(\cG)_\Ran$ on $\Shv(\Bun_G)$ by
\emph{integral Hecke functors},
preserving the subcategory $\Shv_\Nilp(\Bun_G)$. 
By design, this action encodes
the Hecke actions for varying finite sets $I$.
We refer to \secref{sss:Ran action} for the construction. 

\medskip 

We denote the monoidal product on $\Rep(\cG)_\Ran$ by $\star$ and
its monoidal unit by $\one_{\Rep(\cG)_\Ran}$.

\sssec{}  \label{sss:Ran I intro}

Below, we will be considering functors $\CS:\Rep(\cG)_\Ran \to \Vect$. 

\medskip 

By the definition of $\Rep(\cG)_\Ran$, such functors amount to compatible systems of functors
$$\CS_I:\Rep(\cG)^{\otimes I} \to \Shv(X^I)$$
defined for $I \in \fSet$.

\medskip 

In examples, the functors $\CS_I$ tend to be more familiar avatars of the
functor $\CS$, so it is
convenient to reference them.

\sssec{}

On the one hand, we have a functor
$$\Sht^{\Tr}:\Rep(\cG)_\Ran \to \Vect$$
constructed as the composition 
$$\Rep(\cG)_\Ran \to \End_{\DGCat}(\Shv_{\Nilp}(\Bun_G)) 
\overset{-\circ 
(\Frob_{\Bun_G})_*}{\to} 
\End_{\DGCat}(\Shv_{\Nilp}(\Bun_G)) 
\overset{\on{Tr}}{\to} \Vect.$$ 
In other words, we take $\CV \in \Rep(\cG)_\Ran$, form the corresponding
Hecke functor, compose with Frobenius, and take the trace of the resulting
endofunctor of $\Shv_\Nilp(\Bun_G)$.

\medskip 

By construction, we have
$$\Sht^{\Tr}(\one_{\Rep(\cG)_\Ran}) = 
\Tr((\Frob_{\Bun_G})_*,\Shv_\Nilp(\Bun_G)).$$

\begin{rem}

To make the above more explicit, we describe the corresponding functors
$\Sht_I^{\Tr}:\Rep(\cG)^{\otimes I} \to \Shv(X^I)$ for 
a finite set $I$ (see \secref{sss:Ran I intro} just above). 

\medskip 

Precomposing \eqref{eq:hecke intro} with $(\Frob_{\Bun_G})_*$, we obtain a functor
$$\sH(V,-)\circ (\Frob_{\Bun_G})_*: \Shv_\Nilp(\Bun_G)\to \Shv_\Nilp(\Bun_G) \otimes \qLisse(X)^{\otimes I},$$
and we can consider its \emph{parameterized} trace
$$\Tr(\sH(V,-)\circ (\Frob_{\Bun_G})_*,\Shv_\Nilp(\Bun_G))\in \qLisse(X)^{\otimes I}.$$

\medskip 

Unwinding the constructions, the resulting functor
$$\Rep(\cG)^{\otimes I}\to \qLisse(X)^{\otimes I},$$
followed by the embedding
$$\qLisse(X)^{\otimes I}\hookrightarrow \Shv(X^I),$$
is our $\Sht_I^{\Tr}$.

\end{rem}

\sssec{}

On the other hand, following \cite{VLaf}, to the data 
$(I \in \fSet,V \in \Rep(\cG)^{\otimes I})$ we can attach the \emph{compactly supported shtuka cohomology},
which is an object 
$$\Sht_I(V)\in \Shv(X^I),$$
see \secref{ss:sht} in the main body of the 
paper. These functors satisfy the requisite compatibilities 
needed to define a functor $$\Sht:\Rep(\cG)_\Ran \to \Vect.$$

\begin{example}

For $I = \emptyset$, the functor $\Sht_{\emptyset}$
amounts to a map $\Vect \to \Vect$, i.e., a vector space, which corresponds
to $\Sht(\one_{\Rep(\cG)_\Ran})$.
This vector space is $\on{Funct}_c(\Bun_G(\BF_q))$.

\end{example}

%
%

\sssec{}

Our second main result (it appears as \thmref{t:main} in the main body of the paper) asserts:

\begin{mainthm} \label{t:Sht conj intro}
There is a canonical equivalence 
$$\Sht^{\Tr} \simeq \Sht$$ 
of functors $\Rep(\cG)_\Ran \to \Vect$.
\end{mainthm}

In particular, \thmref{t:Sht conj intro} implies that the functors $\Sht_I^{\Tr}$ and $\Sht_I$ are canonically 
isomorphic. This is exactly the assertion of the Shtuka Conjecture 
\cite[Conjecture 22.5.7]{AGKRRV1}.

\medskip

\thmref{t:Tr conj intro} stated above is obtained from \thmref{t:Sht conj intro} by evaluating 
both functors on the object $\one_{\Rep(\cG)_\Ran}\in \Rep(\cG)_\Ran$.

\ssec{How do we compute the trace?}

In this subsection we explain a key ingredient that goes into the computation of the trace
of the Frobenius endofunctor. 

\sssec{} \label{sss:trace Betti}

By way of motivation, let us try to compute the trace of an endofunctor of the form
$\phi_!$, where:

\begin{itemize}

\item $\CY$ is an algebraic stack over $\BC$;

\item $\phi$ is its endomorphism;

\item $\phi_!$ is the direct image with compact supports functor on the category
$\Shv^{\on{Betti}}(\CY)$ of all\footnote{I.e., not necessarily ind-constructible.} Betti sheaves on $\CY$ with coefficients in $\sfe$-vector
spaces (see, e.g., \cite[Sects. G.1 and G.7]{AGKRRV1}). 

\end{itemize}

We claim that there is a canonical isomorphism
\begin{equation} \label{e:Trace Betti}
\Tr(\phi_!,\Shv^{\on{Betti}}(\CY))\simeq \on{C}^\cdot_c(\CY^\phi,\ul\sfe_{\CY^\phi}),
\end{equation}
where $\CY^\phi$ is the stack of $\phi$-fixed points on $\CY$.

\medskip

The computation proceeds as follows. First, one shows that the external tensor product
functor
\begin{equation} \label{e:Kunneth Betti}
\Shv^{\on{Betti}}(\CY)\otimes \Shv^{\on{Betti}}(\CY) \to \Shv^{\on{Betti}}(\CY\times \CY).
\end{equation}
is an equivalence (see \cite[Corollary A.2.9]{GKRV}). 

\medskip

Thus, we can consider the object
$$\on{ps-u}_\CY:=(\Delta_\CY)_!(\ul\sfe_\CY)\in \Shv^{\on{Betti}}(\CY\times \CY)$$
as an object of $\Shv^{\on{Betti}}(\CY)\otimes \Shv^{\on{Betti}}(\CY)$. 

\medskip

Base change shows that the above object and the pairing
$$\ev_\CY^l:\Shv^{\on{Betti}}(\CY)\otimes \Shv^{\on{Betti}}(\CY) \to \Vect, \quad 
\CF_1,\CF_2\mapsto \on{C}^\cdot_c(\CY,\CF_1\overset{*}\otimes \CF_2)$$
define an identification
$$\Shv^{\on{Betti}}(\CY)\simeq \Shv^{\on{Betti}}(\CY)^\vee.$$

Using the above identification, we compute $\Tr(\phi_!,\Shv^{\on{Betti}}(\CY))$
as pull-push along the following diagram, where we apply $*$-pullback along
vertical arrows and $!$-pushforward along horizontal ones:
$$
\CD
& & & & & & \CY @>>> \on{pt} \\
& & & & & & @VV{\Delta_\CY}V \\
\CY @>{\Delta_\CY}>> \CY \times \CY @>{\phi\times \on{id}}>> \CY \times \CY \\
@VVV  \\
\on{pt}.
\endCD
$$

By base change, we can replace this diagram by
$$
\CD
\CY^\phi @>>> \CY @>>> \on{pt} \\
@VVV \\
\CY \\
@VVV \\
\on{pt},
\endCD
$$
and pull-push along the latter diagram exactly gives $\on{C}^\cdot_c(\CY^\phi,\ul\sfe_{\CY^\phi})$.

\sssec{}

Let us now take $\CY$ to be an algebraic stack over $\ol\BF_q$,
defined over $\BF_q$ (see \secref{sss:From intro}), and $\phi=\Frob_\CY$. Instead of 
$\Shv^{\on{Betti}}(\CY)$, we take the category of ind-constructible sheaves $\Shv(\CY)$,
as defined in \cite[Sect. F.1]{AGKRRV1}. 

\medskip

Let us try to compute $\Tr((\Frob_\CY)_!,\Shv(\CY))$. Note that $(\Frob_\CY)_!\simeq (\Frob_\CY)_*$,
since $\Frob_\CY$ is a proper morphism. 

\medskip

Note that $\CY^{\Frob_\CY}\simeq \CY(\BF_q)$, viewed as a \emph{discrete} algebraic stack. 
Hence, the counterpart of the right-hand side of \eqref{e:Trace Betti} yields
$$\on{Funct}_c(\CY(\BF_q)).$$

\medskip 

In the example, of $\CY=\Bun_G$, we would thus obtain the space of unramified automorphic functions
with compact supports. 

\medskip

However, the counterpart of the computation of the left-hand side of \eqref{e:Trace Betti} along the
lines of \secref{sss:trace Betti} \emph{invalid} and, as a result, it is, in general, emphatically \emph{not} true that 
$\Tr((\Frob_\CY)_!,\Shv(\CY))$ is isomorphic\footnote{There is always a natural map $\Tr((\Frob_\CY)_!,\Shv(\CY))\to \on{Funct}_c(\CY(\BF_q))$,
see \eqref{e:LT intro} below.} to $\on{Funct}_c(\CY(\BF_q))$. 

\medskip 

The reason for the failure of the calculation is that the corresponding functor
\begin{equation} \label{e:Kunneth intro}
\Shv(\CY)\otimes \Shv(\CY) \to \Shv(\CY\times \CY)
\end{equation}
is fully faithful, but \emph{not} an equivalence. 

\medskip

In particular, the object 
$$\on{ps-u}_\CY:=(\Delta_\CY)_!(\ul\sfe_\CY)\in \Shv(\CY\times \CY)$$
does not belong to the essential image of \eqref{e:Kunneth intro}.

\medskip

Moreover, it is in general not true that the pairing
\begin{equation} \label{e:left pairing intro}
\ev_\CY^l:\Shv(\CY)\otimes \Shv(\CY) \to \Vect, \quad 
\CF_1,\CF_2\mapsto \on{C}^\cdot_c(\CY,\CF_1\overset{*}\otimes \CF_2)
\end{equation}
is the counit of \emph{any} self-duality on $\Shv(\CY)$.

\begin{rem} \label{r:Verdier intro}

That said, the category $\Shv(\CY)$ is self-dual (at least when $\CY$ is quasi-compact) 
by a different procedure, for which the counit is given by 
$$\ev_\CY:\Shv(\CY)\otimes \Shv(\CY) \to \Vect, \quad 
\CF_1,\CF_2\mapsto \on{C}^\cdot_\blacktriangle(\CY,\CF_1\sotimes \CF_2).$$

The corresponding equivalence
$$(\Shv(\CY)^c)^{\on{op}}\to \Shv(\CY)^c$$
is given by the Verdier duality involution.

\medskip

We refer to the resulting self-duality of $\Shv(\CY)$ as the \emph{Verdier self-duality}.

\medskip

When $\CY$ is not quasi-compact, a variant of this construction is still applicable,
but the dual of $\Shv(\CY)$ is the category $\Shv(\CY)_{\on{co}}$, see 
\cite[Sects. C.2 and C.3]{AGKRRV2}.

\end{rem}

\sssec{}

Thus, as was mentioned above, it is \emph{not} true, in general, that  $\Tr((\Frob_\CY)_*,\Shv(\CY))$ is isomorphic to 
$\on{Funct}_c(\CY(\BF_q))$. 

\medskip

We now take $\CY=\Bun_G$, but instead of all of $\Shv(\Bun_G)$, we take 
$\Shv_\Nilp(\Bun_G)$. 

\medskip

A key observation is that if we restrict the pairing \eqref{e:left pairing intro}, the resulting pairing
$$\Shv_\Nilp(\Bun_G)\otimes \Shv_\Nilp(\Bun_G)\to \Vect$$
is \emph{miraculously} the counit of a self-duality. This is the content of the main result of
\cite{AGKRRV2}, stated therein as Theorem 3.2.2.

\begin{rem}
The usage of the word ``miraculous"  in the above phrase is not accidental:

\medskip

In \cite[Theorem 3.3.3]{AGKRRV2}, we show that the self-duality of $\Shv_\Nilp(\Bun_G)$
induced by \eqref{e:left pairing intro} is related to the Verdier self-duality of Remark 
\ref{r:Verdier intro} via the \emph{miraculous functor}, denoted $\Mir_{\Bun_G}$. 

\end{rem}

\sssec{} \label{sss:ps-u nilp intro}

The self-duality of $\Shv_\Nilp(\Bun_G)$, induced by \eqref{e:left pairing intro} paves a way
to computing traces of endofunctors of this category. 

\medskip 

However, in order to do so, we need one more ingredient: we need to have an explicit
description of the unit of this self-duality. Another key takeaway from the paper \cite{AGKRRV2} 
is such a description: 

\medskip

There exists a particular object 
$$\sR\in \Rep(\cG)_\Ran$$
(termed \emph{Beilinson's spectral projector}, see \secref{sss:Beil intro} below) such that when we apply it along the left (or right) factor 
of $\Bun_G\times \Bun_G$ to $\on{ps-u}_{\Bun_G}$, the resulting object belongs to the essential
image of the functor
$$\Shv_\Nilp(\Bun_G)\otimes \Shv_\Nilp(\Bun_G)\hookrightarrow
\Shv(\Bun_G)\otimes \Shv(\Bun_G)\overset{\boxtimes}\hookrightarrow \Shv(\Bun_G\times \Bun_G)$$
and is the unit of the above self-duality of $\Shv_\Nilp(\Bun_G)$. 

\ssec{Outline of the proof}

We now give an overview of the proof of \thmref{t:Sht conj intro}.
It requires a few additional objects and some relations between them.

\sssec{} \label{sss:Loc intro}

First, we recall the (non-algebraic) derived stack
$$\LocSys_\cG^{\on{restr}}(X)$$
over $\sfe$ introduced in \cite{AGKRRV1}, the 
\emph{stack of local systems with restricted variation}.

\medskip 

We have a naturally defined symmetric monoidal localization functor
$$\Loc:\Rep(\cG)_\Ran\to \QCoh(\LocSys_\cG^{\on{restr}}(X)),$$
see \secref{ss:Loc}. 

\sssec{} \label{sss:Beil intro}

We use \emph{Beilinson's spectral projector}, $\sR\in \Rep(\cG)_\Ran$, see  \secref{sss:ps-u nilp intro}. 

\medskip

It has the following property
vis-a-vis $\LocSys_\cG^{\on{restr}}(X)$:
$$\Loc(\sR)\simeq \CO_{\LocSys_\cG^{\on{restr}}(X)},$$
see \thmref{t:diagonal LocSys}, which is a restatement
of \cite[Theorem 12.7.4]{AGKRRV1}.

\sssec{}  \label{sss:quote Xue intro}

Using the result of \cite{Xue2} mentioned above, we show in \corref{c:Cong cor}
that the functor $\Sht$ factors canonically as a composition
$$\Rep(\cG)_\Ran \overset{\Loc}{\longrightarrow} 
\QCoh(\LocSys^{\on{restr}}_\cG(X)) \overset{\Sht_{\Loc}}{\longrightarrow} \Vect$$
for a certain functor $\Sht_{\Loc}$. 

\sssec{}
 
At this point, the assertion of \thmref{t:Sht conj intro} follows easily from the properties
of $\sR$ mentioned above:

\medskip

The property of $\sR$ mentioned in \secref{sss:ps-u nilp intro} implies that we have a canonical
isomorphism
$$\Sht^{\Tr}(-) \simeq  \Sht(\sR \star -);$$
this follows by essentially rerunning the calculation in \secref{sss:trace Betti}.

\medskip

Applying \secref{sss:quote Xue intro}, we rewrite this further as 
$$\Sht_{\Loc}\circ \Loc(\sR \star -).$$ 

\medskip

Since $\Loc$ is monoidal and
sends $\sR$ to the unit of $\QCoh(\LocSys^{\on{restr}}_\cG(X))$, we have 
$$\Loc(\sR \star -)\simeq \Loc(-).$$

\medskip

Hence, combining, we obtain
$$\Sht^{\Tr}(-) \simeq \Sht(\sR \star -)\simeq 
\Sht_{\Loc}\circ \Loc(\sR \star -) \simeq \Sht_{\Loc}\circ \Loc(-) \simeq \Sht(-),$$
which is exactly the assertion of \thmref{t:Sht conj intro}. 

%
%
%
%

\ssec{Relation to the classical sheaves-function dictionary} 

The Trace Conjecture as stated in \cite{AGKRRV1} is somewhat stronger
than \thmref{t:Tr conj intro} stated above, in that it specifies a particular morphism
from $\Tr((\Frob_{\Bun_G})_*,\Shv_\Nilp(\Bun_G))$ to $\on{Funct}_c(\Bun_G(\BF_q))$
that is supposed to be an isomorphism, see \cite[Conjecture 22.3.7(b) and Remark 22.3.9]{AGKRRV1}. 

\medskip

In this subsection we describe the results of the present paper in this direction.

\sssec{}

Let $\CF$ be a compact object of $\Shv_\Nilp(\Bun_G)$, equipped with a structure of 
\emph{weak Weil equivariance}, i.e., a morphism
$$\alpha:\CF\to (\Frob_{\Bun_G})_*(\CF),$$
which is equivalent to the datum of a map 
$$\alpha':(\Frob_{\Bun_G})^*(\CF)\to \CF.$$

By functoriality of the categorical trace, we can attach to the pair $(\CF,\alpha)$ its class
$$\on{cl}(\CF,\alpha)\in \Tr((\Frob_{\CY})_*,\Shv_\Nilp(\CY)).$$ 

Hence, applying \thmref{t:Tr conj intro}, we can further attach to it an element of $\on{Funct}_c(\CY(\BF_q))$.

\medskip

We wish to say that this function equals the function attached to $(\CF,\alpha')$ by the classical 
sheaves-function correspondence, i.e., by taking pointwise traces of the Frobenius on them stalks 
of $\CF$. 

\medskip

However, here is an issue: we do not know that $\CF$ is constructible as a sheaf on $\Bun_G$.
So, the above pointwise trace operation is not a priori well-defined, as the stalks in question may
be inifinite-dimensional.  

\sssec{}

To overcome to issue mentioned above, 
for the duration of this subsection we will assume \cite[Conjecture 14.1.8]{AGKRRV1}. This
conjecture has several equivalent formulations:

\begin{itemize}

\item The category $\Shv_\Nilp(\Bun_G)$ is generated by objects that are compact in the ambient
category $\Shv(\Bun_G)$;

\item Compact objects of $\Shv_\Nilp(\Bun_G)$ are compact as objects of $\Shv(\Bun_G)$;

\item Compact objects of $\Shv_\Nilp(\Bun_G)$ are constructible as objects of $\Shv(\Bun_G)$;

\item Compact objects of $\Shv_\Nilp(\Bun_G)$ are bounded below in the natural t-structure on 
$\Shv(\Bun_G)$.

\end{itemize}

An analog of \cite[Conjecture 14.1.8]{AGKRRV1} is known when our ground field $k$ has characteristic $0$,
and $\Shv(-)$ is either the category of ind-holonomic D-modules or the category of ind-constructible Betti 
sheaves (see \cite[Theorems 16.4.3 and 16.4.10]{AGKRRV1}). Furthermore, we believe that we can prove it
when $G=GL_n$.

\medskip

To summarize, this conjecture is a rather plausible statement.

\sssec{}

Assuming \cite[Conjecture 14.1.8]{AGKRRV1}, we obtain that the embedding \eqref{e:Nilp intro}
admits a continuous right adjoint.

\medskip

Hence, by functoriality of categorical trace, we obtain a map
\begin{equation} \label{e:forget Nilp intro}
\Tr((\Frob_{\Bun_G})_*,\Shv_\Nilp(\Bun_G))\to \Tr((\Frob_{\Bun_G})_*,\Shv(\Bun_G)).
\end{equation} 

\sssec{}

For any quasi-compact algebraic stack $\CY$ over $\ol\BF_q$ and defined over $\BF_q$, 
there is a canonical
\emph{local term} map (see \cite[Sect. 22.2]{AGKRRV1}), 
\begin{equation} \label{e:LT intro}
\on{LT}_\CY:\Tr((\Frob_{\CY})_*,\Shv(\CY)) \to \on{Funct}(\CY(\BF_q)),
\end{equation}
where $\on{Funct}(-)$ stands for the (classical) vector space $\sfe$-valued functions. 

\medskip 

In addition, in \emph{loc. cit}., we extended this construction to non-quasi-compact
stacks (such as $\Bun_G$). In this setting, the above map takes the form
$$\on{LT}_\CY:\Tr((\Frob_{\CY})_*,\Shv(\CY)) \to \on{Funct}_c(\CY(\BF_q)),$$  
where $\on{Funct}_c(\CY(\BF_q)) \subset \on{Funct}(\CY(\BF_q))$ is the subspace
of compactly supported $\sfe$-valued functions.

\begin{rem}\label{r:lt function}

By functoriality of traces, any (possibly lax) Weil sheaf 
$$(\CF \in \Shv(\CY)^c,\alpha:\CF \to (\Frob_{\CY})_*(\CF))$$ 
on $\CY$ defines an element
$$\on{cl}(\CF,\alpha)\in \Tr((\Frob_{\CY})_*,\Shv(\CY)).$$ 
According to \cite[Sect. 22.2]{AGKRRV1}, the value of the map
$\on{LT}_\CY$ on $\on{cl}(\CF,\alpha)$ produces the corresponding Grothendieck--Deligne
function, denoted $\on{funct}(\CF)$.

\end{rem}

\sssec{}

By the above, we obtain a map
\begin{equation} \label{e:trace map}
\Tr((\Frob_{\Bun_G})_*,\Shv_\Nilp(\Bun_G)) \overset{\text{\eqref{e:forget Nilp intro}}}\longrightarrow
\Tr((\Frob_{\Bun_G})_*,\Shv(\Bun_G)) \overset{\on{LT}_{\Bun_G}}\to 
\on{Funct}_c(\Bun_G(\BF_q)),
\end{equation} 
where we note that the vector space $\on{Funct}_c(\Bun_G(\BF_q))$ is the space of 
compactly supported unramified automorphic functions.

\sssec{}

We can now state our third main (it appears as \thmref{t:lt} in the main body of the paper): 

\begin{mainthm} \label{t:Tr conj intro strong}
The map \eqref{e:trace map} equals the \emph{isomorphism} of \thmref{t:Tr conj intro}.
 \end{mainthm}
 
\begin{cor} \label{c:Tr conj intro strong}
The map \eqref{e:trace map} is an
isomorphism. 
\end{cor}

\begin{rem}

Informally, one can view \corref{c:Tr conj intro strong} as saying that ``there are enough 
weak Weil sheaves on $\Bun_G$ with nilpotent singular support 
to recover all automorphic functions, and any relations
between the automorphic functions defined by such sheaves
have categorical origins." 

\medskip

Of course, the fact that we are considering sheaves with nilpotent singular support 
is crucial here. If we did not have the singular support condition, we would obviously
have enough sheaves to recover all functions.  However, the relations imposed by sheaves
would not match the relations on functions.

\end{rem}

\begin{rem}

The phrase in quotation makes in the previous remark would be a correct assertion if the phrase
inside the quotation marks was understood in the derived sense.

\medskip

More precisely, for a (compactly generated) category $\bC$ with an endofunctor $\Phi$,
the trace object $\on{Tr}(\Phi,\bC)\in \Vect$ is computed as the geometric realization of a 
certain canonically defined simplicial object of $\Vect$; let us denote it $\on{Tr}(\Phi,\bC)^\bullet$.

\medskip

In particular, we have a map $\on{Tr}(\Phi,\bC)^0\to \on{Tr}(\Phi,\bC)$ (here the superscript $0$
denotes the space of $0$-simplices), and hence a map
$$H^0(\on{Tr}(\Phi,\bC)^0)\to H^0(\on{Tr}(\Phi,\bC)).$$

Now, the image of the latter map is the span of the classes 
$$\on{cl}(\bc,\alpha), \quad \bc\in \bC^c,\alpha:\bc\to \Phi(\bc).$$

However, it is not true, in general, that $H^0(\on{Tr}(\Phi,\bC))$ is spanned by the above
classes: higher cohomologies of higher simplices can also contribute.

\end{rem}

\ssec{Organization of the paper}

We now describe how the present paper is structured.

\sssec{}

In \secref{s:Hecke} we review the formalism of Hecke functors acting
on $\Shv(\Bun_G)$.

\medskip

In particular, we introduce the Ran version of the category $\Rep(\cG)$, denoted $\Rep(\cG)_\Ran$,
which is a monoidal category that acts on $\Shv(\Bun_G)$ by \emph{integral Hecke functors}. 

\medskip

This section does not contain any new results. 

\sssec{}

In \secref{s:LocSys} we review some notions associated with the \emph{stack of local systems with restricted variation},
introduced in \cite{AGKRRV1}, and denoted $\LocSys^{\on{restr}}_\cG(X)$.

\medskip

Apart from the definition of $\LocSys^{\on{restr}}_\cG(X)$, the main points are:

\medskip

\noindent(i) Description of the dual category $\QCoh(\LocSys^{\on{restr}}_\cG(X))^\vee$ as the category of compatible
collections of functors 
$$\Rep(\cG)^{\otimes I}\to \qLisse(X)^{\otimes I}, \quad I\in \fSet;$$

\medskip 

\noindent(ii) The localization functor
$$\Loc:\Rep(\cG)_\Ran\to \QCoh(\LocSys^{\on{restr}}_\cG(X));$$

\medskip

\noindent(iii) Construction of Beilinson's spectral projector, which is an explicit object
$\sR\in \Rep(\cG)_\Ran$, one of whose main properties is the isomorphism
$$\Loc(\sR)\simeq \CO_{\LocSys^{\on{restr}}_\cG(X)}.$$

\medskip

\noindent(iv) \corref{c:Loc dual}, which asserts that a functor
$\CS:\Rep(\cG)_\Ran \to \Vect$ factors (uniquely) 
through the localization functor $\Loc$ exactly when the functors $\CS_I$ are valued
in $\qLisse(X)^{\otimes I}\subset \Shv(X^I)$.

\medskip 

The entirety of the material of this section is a reformulation of the results in Parts I and II of \cite{AGKRRV1}. 

\sssec{}

In \secref{s:shtuka} we review the shtuka construction and some of its variants.

\medskip 

First, we introduce the functor $\Sht:\Rep(\cG)_\Ran \to \Vect$. 

\medskip 

We then quote the main result from \cite{Xue2}, which we interpret as saying that
the functors $\Sht_I$ take values in $\qLisse(X)^{\otimes I}\subset \Shv(X^I)$. 
Using (iv) above, this yields the existence of the functor 
$\Sht_{\Loc}$ from above. 

%

\sssec{}

In \secref{s:main} we formulate and prove the main result of this paper, 
\thmref{t:main} (which is the same as \thmref{t:Sht conj intro} above).
As particular cases, this statement contains both the (unrefined) Trace Conjecture and 
the Shtuka Conjecture.

\medskip

The argument is the one outlined above. 

\sssec{}

In Sects. \ref{s:LT} and \ref{s:serre-true} we assume the validity of \cite[Conjecture 14.1.8]{AGKRRV1}, 
and we show that the isomorphism 
$$\Tr(\Frob_*,\Shv_\Nilp(\Bun_G))\simeq \on{Funct}_c(\Bun_G(\BF_q))$$
of \thmref{c:trace} is induced by the local term map \eqref{e:trace map}, as in the statement of
\thmref{t:Tr conj intro strong}.

%
%
%


\ssec{Notations and conventions}\label{ss:notation}

The notations in this paper largely follow those of \cite{AGKRRV1} and \cite{AGKRRV2}.

\sssec{Algebraic geometry}

There will be disjoint ``two algebraic geometries" at play in this paper: one on the automorphic side,
and another on the spectral side.

\medskip 

On the automorphic side, our algebraic geometry will be over the ground field $k$, which in this paper is 
$\ol\BF_q$. Our algebro-geometric objects will be either schemes or algebraic stacks locally of finite type over
$k$. In this paper we will not need more general prestacks.  Moreover, the algebraic geometry that we consider
over $k$ is classical (i.e. \emph{not} derived).

\medskip

On the spectral side, our algebraic geometry will be over the field of coefficients $\sfe:=\ol\BQ_\ell$,
see below. We will consider just one algebro-geometric object over $\sfe$--the (pre)stack $\LocSys^{\on{restr}}_\cG(X)$
(see \secref{s:LocSys}), but it will play quite a prominent role.  Importantly, the algebraic geometry we
consider over $\sfe$ is \emph{derived}: by default, all schemes, stacks, etc. over $\sfe$ are derived.

%

\sssec{Frobenius endomorphism} \label{sss:Frob}

Let $\CY$ be an algebraic stack over $\ol\BF_q$, but defined over $\BF_q$. In this case,
we can consider the geometric Frobenius endomorphism of $\CY$, denoted
$$\Frob_\CY:\CY\to \CY.$$

\medskip

Thus, whenever we refer to the Frobenius endomorphism of $\CY$, we will assume 
that $\CY$ is defined over $\BF_q$. This is the case of our curve $X$, the reductive group $G$,
and the stack $\Bun_G$ of principal $G$-bundles on $X$. 

\sssec{Higher algebra}

We will work with DG categories over the field of coefficients $\sfe:=\ol\BQ_\ell$.

\medskip

All our conventions and notations regarding DG categories are imported from \cite[Sects. 0.5.2-0.5.3]{AGKRRV2}.

\medskip

There will be two kinds of sources that feed into higher algebra, i.e., the sources of DG categories.

\medskip

One will be various categories produced out of $\ell$-adic sheaves on the automorphic side. 
Another will be categories of quasi-coherent sheaves on the spectral side (specifically, the category 
of quasi-coherent sheaves on $\LocSys^{\on{restr}}_\cG(X)$). 

\sssec{Sheaves} \label{sss:sheaves}

For a scheme $S$ of finite type, we let $\Shv(S)^{\on{constr}}$ denote the category of constructible
$\ol\BQ_\ell$-adic sheaves on $S$, viewed as a (small) DG category over the field of coefficients
$\sfe=\ol\BQ_\ell$.

\medskip

We let $\Shv(S)$ denote the (cocomplete) DG category $\on{Ind}(\Shv(S)^{\on{constr}})$.
We extend the assignment 
$$S\mapsto \Shv(S)$$
from schemes to algebraic stacks by the procedure explained in \cite[Sect. A.1]{AGKRRV2}.

\medskip

For a given stack $\CY$, we will denote by
$$\ul\sfe_\CY,\omega_\CY\in \Shv(\CY)$$
the constant and dualizing sheaves, respectively. 

\medskip

If $\CY$ is smooth, inside $\Shv(\CY)$ we consider a full subcategory
$$\qLisse(\CY)\subset \Shv(\CY),$$
defined as in \cite[Sect. 1.2]{AGKRRV1}.

\sssec{Singular support} 

Let $\CY$ be an algebraic stack and $\CN$ a conical Zariski-closed subset of $T^*(\CY)$. We will denote by
$$\Shv_\CN(\CY)\subset \Shv(\CY)$$
the corresponding full subcategory, defined as in \cite[Sects. E.5 and F.6]{AGKRRV1}.

\medskip

If $\CY$ is smooth and $\CN$ is the zero-section, usually denoted $\{0\}$, we will also use the notation $\qLisse(\CY)$
for $\Shv_{\{0\}}(\CY)$.

%
%
%
%
%
%

\sssec{Functors (co)defined by kernels} \label{sss:ker}

In a few places in this paper we will make reference to functors \emph{defined} or \emph{codefined}
by a kernel. We refer the reader to \cite[Sect. B]{AGKRRV2}, where these notions are 
introduced.

\medskip

Following {\it loc. cit.}, given a functor $\sF:\Shv(\CY_1)\to \Shv(\CY_2)$, defined or codefined by a kernel,
we will denote by
$$\on{Id}_\CZ\boxtimes \sF: \Shv(\CZ\times \CY_1)\to \Shv(\CZ\times \CY_2)$$
the corresponding functor for an algebraic stack $\CZ$ (thought of as a stack of parameters). 

\ssec{Acknowledgements}

We are grateful to C.~Xue for communicating to us her results from \cite{Xue2} before making them publicly available.
We are also grateful to A.~Beilinson and V.~Lafforgue for inspiring discussions.

\medskip

The entire project was supported by David Kazhdan's ERC grant No 669655
and BFS grant 2020189. The work of D.K. and Y.V. was supported BSF grant 2016363. 

\medskip

The work of D.A. was supported by NSF grant DMS-1903391. 
The work of D.G. was supported by NSF grant DMS-2005475. 
The work of D.K. was partially supported by ISF grant 1650/15. 
The work of S.R. was supported by NSF grant DMS-2101984.
The work of Y.V. was partially supported by ISF grants 822/17 and 2019/21. 

\section{The Hecke action} \label{s:Hecke}

In this section we will recall the pattern of Hecke acton of the category $\Rep(\cG)$ on
$\Shv(\Bun_G)$, and some related formalism. The section contains no original material. 

\ssec{Hecke functors}

\sssec{}

In this paper, by the phenomenon of Hecke action we will understand a system of functors, defined for every finite set $I$ and every 
algebraic stack $\CZ$ (thought of as a stack of parameters):
\begin{equation} \label{e:Hecke !}
\on{Id}_\CZ \boxtimes \sH:\Rep(\cG)^{\otimes I}\otimes \Shv(\CZ\times \Bun_G)\to \Shv(\CZ\times \Bun_G\times X^I).
\end{equation}

For a fixed $V\in \Rep(\cG)^{\otimes I}$, we will denote by $\on{Id}_\CZ \boxtimes \sH(V,-)$ the resulting functor
$$\Shv(\CZ\times \Bun_G)\to \Shv(\CZ\times \Bun_G\times X^I).$$

\medskip

When $\CZ=\on{pt}$, we will simply write $\sH$ (resp., $\sH(V,-)$). 

\sssec{}\label{sss:associativity of Hecke}
The Hecke functions \eqref{e:Hecke !} are associative in the following sense: we have a natural commutative diagram
of functors
$$
\CD
\Rep(\cG)^{\otimes I} \otimes \Rep(\cG)^{\otimes I} \otimes \Shv(\CZ\times \Bun_G) @>{\on{mult}}>>  \Rep(\cG)^{\otimes I} \otimes \Shv(\CZ\times \Bun_G) \\
@V{\on{Id}\otimes (\on{Id}_\CZ \boxtimes \sH)}VV \\
\Rep(\cG)^{\otimes I} \otimes \Shv(\CZ\times \Bun_G\times X^I) & & @VV{\on{Id}_\CZ \boxtimes \sH}V \\
@V{\on{Id}_\CZ \boxtimes \sH}VV \\
 \Shv(\CZ \times \Bun_G \times X^I \times X^I) @>{(\on{Id} \times \Delta)^!}>> \Shv(\CZ\times \Bun_G\times X^I),
\endCD
$$
%
where $$\on{mult}: \on{Rep}(\cG)^{\otimes I} \otimes \on{Rep}(\cG)^{\otimes I} \to \on{Rep}(\cG)^{\otimes I}$$ is the tensor product functor and 
$\Delta_{X^I}: X^I \to X^I \times X^I$ is the diagonal embedding.  The data of associativity of \eqref{e:Hecke !} comes additionally with higher coherence for higher powers of $\Rep(\cG)^{\otimes I}$.

\medskip

We can rephrase this as saying that the category $\Shv(\CZ \times \Bun_G \times X^I)$ is a module category for the  monoidal category $\on{Rep}(\cG)^{\otimes I}$, 
and the action is $\Shv(X^I)$-linear (i.e. it is a module category for the monoidal category $\on{Rep}(\cG)^{\otimes I} \otimes \Shv(X^I)$),
where the symmetric monoidal structure on $\Shv(X^I)$ is given by $\sotimes$.

\sssec{}

The functors \eqref{e:Hecke !} are naturally compatible with maps between finite sets. Namely, for a map 
$\psi:I\to J$, we have a data of commutativity for the diagram
\begin{equation} \label{e:IJ com}
\CD
\Rep(\cG)^{\otimes I}\otimes \Shv(\CZ\times \Bun_G) @>{\on{Id}_\CZ \boxtimes \sH}>>  \Shv(\CZ\times \Bun_G\times X^I) \\
@V{\on{Id}\otimes \on{mult}^\psi}VV @VV{(\on{Id}\times \Delta_\psi)^!}V \\
\Rep(\cG)^{\otimes J}\otimes \Shv(\CZ\times \Bun_G) @>{\on{Id}_\CZ \boxtimes \sH}>>  \Shv(\CZ\times \Bun_G\times X^J)
\endCD
\end{equation}
where
$$\on{mult}^\psi:\Rep(\cG)^{\otimes I}\to \Rep(\cG)^{\otimes J}$$
is the functor given by the symmetric monoidal structure on $\Rep(\cG)$, and 
$$\Delta_\psi:X^J\to X^I$$
the diagonal map defined by $\psi$. 

\medskip

The above data of commutativity are endowed with a homotopy coherent system of compatibilities for
compositions of maps of finite sets.

\medskip
Moreover, this data is compatible with the associativity described in \secref{sss:associativity of Hecke}.  Namely,
the functor
$$ (\on{Id} \times \Delta_{\psi})^!: \Shv(\CZ \times \Bun_G \times X^I) \to \Shv(\CZ \times \Bun_G \times X^J) $$
is a functor of $\Shv(X^I) \otimes \Rep(\cG)^{\otimes I}$-module categories.

\sssec{}\label{sss:hecke defined/codefined}

A feature of the functors $\on{Id}_\CZ \boxtimes \sH(V,-)$ is that they are functors that are both 
\emph{defined and codefined by kernels}; see \secref{sss:ker} for what this means. In practical 
terms, this implies that for a map $\CZ_1\to \CZ_2$ between algebraic stacks, we have a datum
of commutativity for the diagrams
$$
\CD
\Shv(\CZ_1\times \Bun_G) @>{\on{Id}_\CZ \boxtimes \sH(V,-)}>>  \Shv(\CZ_1\times \Bun_G\times X^I) \\
@VVV @VVV \\
\Shv(\CZ_2\times \Bun_G) @>{\on{Id}_\CZ \boxtimes \sH(V,-)}>>  \Shv(\CZ_2\times \Bun_G\times X^I),
\endCD
$$
where the vertical arrows are given by either $\blacktriangle$- or !- pushforwards, and also for the diagrams
$$
\CD
\Shv(\CZ_1\times \Bun_G) @>{\on{Id}_\CZ \boxtimes \sH(V,-)}>>  \Shv(\CZ_1\times \Bun_G\times X^I) \\
@AAA @AAA \\
\Shv(\CZ_2\times \Bun_G) @>{\on{Id}_\CZ \boxtimes \sH(V,-)}>>  \Shv(\CZ_2\times \Bun_G\times X^I),
\endCD
$$
where the vertical arrows are given by either !- or *- pullbacks. 

\medskip

Moreover, this datum of commutativity is functorial in $V\in \Rep(\cG)^{\otimes I}$, and is compatible with the datum
of commutativity of the diagrams \eqref{e:IJ com}. 


\ssec{The ULA property of the Hecke action}

\sssec{} \label{sss:ULA Hecke}

Another key feature of the functors $\on{Id}_\CZ \boxtimes \sH$ is that for $\CF\in \Shv(\CZ\times \Bun_G)^{\on{constr}}$ and 
$V\in (\Rep(\cG)^{\otimes I})^c$, the object
$$(\on{Id}_\CZ \boxtimes \sH)(V,\CF)\in \Shv(\CZ\times \Bun_G\times X^I)$$
is ULA with respect to the projection
$$\CZ\times \Bun_G\times X^I\to X^I.$$

\sssec{} \label{sss:Satake ! and *}

Let us denote by $\on{Id}_\CZ \boxtimes \sH^l$ the functor
$$\Rep(\cG)^{\otimes I}\otimes \Shv(\CZ\times \Bun_G)\to \Shv(\CZ\times \Bun_G\times X^I)$$
defined as
$$(\on{Id}_\CZ \boxtimes \sH^l)(V,\CF):=(\on{Id}_\CZ \boxtimes \sH)(V,\CF)\sotimes p_3^!(\ul\sfe_{X^I}).$$

Note that for a fixed $I$, the difference between $\on{Id}_\CZ \boxtimes \sH$ and $\on{Id}_\CZ \boxtimes \sH^l$ amounts to a cohomological
shift by $2|I|$ since $\ul\sfe_{X^I}\simeq \omega_{X^I}[-2|I|]$. 

\sssec{} \label{sss:ULA Hecke bis}

The ULA property of the objects $(\on{Id}_\CZ \boxtimes \sH)(V,\CF)$ implies that we have canonical isomorphisms
\begin{equation} \label{e:ULA Hecke}
(\on{Id}_\CZ \boxtimes \sH)(V,\CF)\sotimes p_3^!(\CM) \simeq (\on{Id}_\CZ \boxtimes \sH^l)(V,\CF)\overset{*}\otimes p_3^*(\CM), \quad \CM\in \Shv(X^I).
\end{equation}

\medskip

Furthermore, for a map of finite sets $\psi:I\to J$, we have a data of commutativity for the diagram
\begin{equation} \label{e:IJ com *}
\CD
\Rep(\cG)^{\otimes I}\otimes \Shv(\CZ\times \Bun_G) @>{\on{Id}_\CZ \boxtimes \sH^l}>>  \Shv(\CZ\times \Bun_G\times X^I) \\
@V{\on{mult}^\psi\otimes \on{Id}}VV @VV{(\on{Id}\times \Delta_\psi)^*}V \\
\Rep(\cG)^{\otimes J}\otimes \Shv(\CZ\times \Bun_G) @>{\on{Id}_\CZ \boxtimes \sH^l}>>  \Shv(\CZ\times \Bun_G\times X^J),
\endCD
\end{equation}
endowed with a homotopy coherent system of compatibilities for
compositions of maps of finite sets.

\sssec{}

Thus, we can regard the functors $\on{Id}_\CZ \boxtimes \sH^l(V,-)$ also as defined and codefined by kernels, and they
have the formal properties parallel to those of the functors $\on{Id}_\CZ \boxtimes \sH(V,-)$. 

\ssec{Hecke action on $\Shv_\Nilp(\Bun_G)$}

\sssec{}

One of the main actor in this paper is the full subcategory
$$\Shv_\Nilp(\Bun_G)\subset \Shv(\Bun_G).$$

\medskip

The following result, essentially due to \cite{NY}, describes the behavior of this subcategory under the Hecke functors
(this is stated as \cite[Theorem 14.2.4]{AGKRRV1}):


\begin{thm} \label{t:Hecke on Nilp}
The Hecke functor $\sH$ for $I=\{*\}$ sends the full subcategory 
$$\Rep(\cG)\otimes \Shv_{\Nilp}(\Bun_G) \subset \Rep(\cG)\otimes \Shv(\Bun_G)$$
to the full subcategory
$$\Shv_{\Nilp \times \{0\}}(\Bun_G\times X) \subset 
\Shv(\Bun_G\times X).$$
\end{thm} 

\sssec{}

Recall also (see \cite[Theorem F.9.7 combined with Corollary E.4.7]{AGKRRV1}) that for any algebraic stack $\CY$ and a conical half-dimensional
Zariski-closed subset $\CN\subset T^*(\CY)$, the (a priori fully faithful) functor
$$\Shv_\CN(\CY)\otimes \qLisse(X) \to \Shv_{\CN\times \{0\}}(\CY\times X)$$
is an equivalence. 

\medskip

Thus, from \thmref{t:Hecke on Nilp} we obtain:


\begin{cor} \label{c:Hecke on Nilp single}
The Hecke functor $\sH$ for $I=\{*\}$  sends the full subcategory 
$$\Rep(\cG)\otimes \Shv_{\Nilp}(\Bun_G) \subset \Rep(\cG)\otimes \Shv(\Bun_G)$$
to the full subcategory
$$\Shv_{\Nilp}(\Bun_G)\otimes \qLisse(X) \subset 
\Shv(\Bun_G\times X).$$
\end{cor} 

Iterating, from \corref{c:Hecke on Nilp single} we further obtain:


\begin{cor} \label{c:Hecke on Nilp}
The Hecke functors $\sH$ map the full subcategory 
$$\Rep(\cG)^{\otimes I}\otimes \Shv_{\Nilp}(\Bun_G) \subset \Rep(\cG)^{\otimes I}\otimes \Shv(\Bun_G)$$
to the full subcategory
$$\Shv_{\Nilp}(\Bun_G)\otimes \qLisse(X)^{\otimes I} \subset 
\Shv(\Bun_G\times X^I).$$
\end{cor} 

\sssec{} \label{sss:* and ! lisse}

Note that for a scheme or stack $\CY$, we can consider $\qLisse(\CY)$ as a full subcategory
of $\Shv(\CY)$ in two different ways.

\medskip

One is the tautological embedding\footnote{Recall that our conventions are such that the default pullback functor is $!$-pullback, and therefore, by definition, lisse sheaves are those that are locally $!$-pulled back from a point.  As explained below, it will be important to also consider the usual notion of lisse sheaves, i.e. those which are locally $*$-pulled back from a point.} 
\begin{equation} \label{e:! emb}
\qLisse(\CY)\hookrightarrow \Shv(\CY), \quad \CL\mapsto \CL;
\end{equation}
it endows $\qLisse(\CY)$ with a symmetric monoidal structure induced by the !-tensor product on $\Shv(\CY)$.

\medskip

We also have a different embedding:
\begin{equation} \label{e:* emb}
\qLisse(\CY)\hookrightarrow \Shv(\CY), \quad \CL\mapsto \CL\sotimes \ul\sfe_\CY;
\end{equation}
it endows $\qLisse(\CY)$ with a symmetric monoidal structure induced by the *-tensor product on $\Shv(\CY)$.

\medskip

However, it follows tautologically that the two symmetric monoidal structures on $\qLisse(\CY)$ coincide. Moreover,
the operations
$$\qLisse(\CY)\otimes \Shv(\CY)\overset{\text{\eqref{e:! emb}}\otimes \on{Id}}\longrightarrow \Shv(\CY)\otimes \Shv(\CY)
\overset{\sotimes}\to \Shv(\CY)$$
and 
$$\qLisse(\CY)\otimes \Shv(\CY)\overset{\text{\eqref{e:* emb}}\otimes \on{Id}}\longrightarrow \Shv(\CY)\otimes \Shv(\CY)
\overset{\overset{*}\otimes}\to \Shv(\CY)$$
define the \emph{same} monoidal action of $\qLisse(\CY)$ on $\Shv(\CY)$. 

\medskip

We will apply this discussion to $\CY=X^I$. 

\begin{rem} \label{r:shift}
Note also that when $\CY$ is smooth of dimension $n$, the embeddings \eqref{e:! emb} and \eqref{e:* emb}
differ by the cohomological shift $[2n]$. Yet they should not be confused. 
\end{rem}

\sssec{}\label{sss:Hecke on Nilp}

An assertion parallel to \corref{c:Hecke on Nilp single} holds for the $\sH^l$ functors. Namely, these functors send

$$\Rep(\cG)^{\otimes I}\otimes \Shv_{\Nilp}(\Bun_G) \subset \Rep(\cG)^{\otimes I}\otimes \Shv(\Bun_G)$$
to
$$\Shv_{\Nilp}(\Bun_G)\otimes \qLisse(X)^{\otimes I} \subset 
\Shv(\Bun_G\times X^I).$$
where we will think of the embedding $\qLisse(X)^{\otimes I} \hookrightarrow \Shv(X^I)$ as given by \eqref{e:* emb}. 

\sssec{}

Thus, we can think of the Hecke action on $\Shv_{\Nilp}(\Bun_G)$ either by means of the functors 
\begin{equation} \label{e:H! qLisse}
\sH:\Rep(\cG)^{\otimes I}\otimes \Shv_{\Nilp}(\Bun_G) \to \Shv_{\Nilp}(\Bun_G)\otimes \qLisse(X)^{\otimes I},
\end{equation}
when we think of $\qLisse(X^I)$ as embedded into $\Shv(X^I)$ via \eqref{e:! emb}, or, equivalently, as
\begin{equation} \label{e:H* qLisse}
\sH^l:\Rep(\cG)^{\otimes I}\otimes \Shv_{\Nilp}(\Bun_G) \to \Shv_{\Nilp}(\Bun_G)\otimes \qLisse(X)^{\otimes I},
\end{equation}
when we think of $\qLisse(X^I)$ as embedded into $\Shv(X^I)$ via \eqref{e:* emb}. 

\medskip

The functors \eqref{e:H! qLisse} and \eqref{e:H* qLisse} are canonically isomorphic. 
Thus, in what follows we will not distinguish notationally between $\sH$ and $\sH^l$, when applied to objects from $\Shv_{\Nilp}(\Bun_G)$,
and just use the notation 
\begin{equation} \label{e:H!* qLisse}
\sH:\Rep(\cG)^{\otimes I}\otimes \Shv_{\Nilp}(\Bun_G) \to \Shv_{\Nilp}(\Bun_G)\otimes \qLisse(X)^{\otimes I}.
\end{equation}

For a map of finite sets $\psi:I\to J$, we have a data of commutativity for the diagram
\begin{equation} \label{e:IJ qLisse}
\CD
\Rep(\cG)^{\otimes I}\otimes \Shv_{\Nilp}(\Bun_G) @>{\sH}>>  \Shv(\Bun_G)\otimes \qLisse(X)^{\otimes I} \\
@V{\on{Id}\otimes \on{mult}^\psi}VV @VV{\on{Id}\times \on{mult}^\psi}V \\
\Rep(\cG)^{\otimes J}\otimes \Shv_{\Nilp}(\Bun_G) @>{\sH}>>  \Shv(\Bun_G)\otimes \qLisse(X)^{\otimes J},
\endCD
\end{equation}
where in the right vertical arrow the functor 
$$\on{mult}^\psi:\qLisse(X)^{\otimes I}\to \qLisse(X)^{\otimes J}$$
is given by the symmetric monoidal structure on $\qLisse(X)$.

\medskip

These data of commutativity are endowed with a homotopy coherent system of compatibilities for
compositions of maps of finite sets.

\ssec{The category $\Rep(\cG)_\Ran$}   \label{ss:Rep Ran}

We will now introduce a device that allows us to express the Hecke action 
on $\Shv(\Bun_G)$ in terms of a single monoidal category, the \emph{Ran version} of
$\Rep(\cG)$, denoted $\Rep(\cG)_\Ran$. 

\medskip

We refer the reader to \cite[Sect. 11]{AGKRRV1} for a detailed discussion of this construction. 

\sssec{} \label{sss:Rep Ran}

Let $\CC$ be a symmetric monoidal DG category.
We define a new symmetric monoidal DG category $\CC_\Ran$ by
the following construction.

\medskip

Let  $\on{TwArr}(\on{fSet})$ be the category of
\emph{twisted arrows} on $\fSet$, see \cite[Sect. 1.2.2]{GKRV}. 

\medskip

The category $\CC_\Ran$ is the colimit over $\on{TwArr}(\on{fSet})$ of the functor
\begin{equation}\label{e:Ran diagram}
\on{TwArr}(\on{fSet})\to \DGCat
\end{equation}
that sends
$$(I\to J) \mapsto \CC^{\otimes I}\otimes \Shv(X^J).$$

\medskip

Here for a map
\begin{equation} \label{e:mor Tw arr}
\CD
I_1 @>>> J_1 \\
@V{\phi_I}VV @AA{\phi_J}A \\
I_2 @>>> J_2,
\endCD
\end{equation} 
in $\on{TwArr}(\on{fSet})$,
the corresponding functor
$$\CC^{\otimes I_1}\otimes \Shv(X^{J_1})\to \CC^{\otimes I_2}\otimes \Shv(X^{J_2})$$
is given by the tensor product functor along the fibers of $\phi_I$
\begin{equation} \label{e:mult phi}
\on{mult}^{\phi_I}:\CC^{\otimes I_1}\to \CC^{\otimes I_2}
\end{equation} 
and the functor
\begin{equation} \label{e:Delta phi}
(\Delta_{\phi_J})_!:\Shv(X^{J_1})\to \Shv(X^{J_2}),
\end{equation} 
where $\Delta_{\phi_J}:X^{J_2}\to X^{J_1}$ is the diagonal map induced by $\phi_J$. 

\sssec{}

The functor \eqref{e:Ran diagram} is naturally right-lax symmetric monoidal.  Therefore, the colimit $\CC_\Ran$
carries a natural symmetric monoidal structure.  Explicitly, this symmetric monoidal structure can be described
as follows.
For
$$V_1\otimes \CM_1\in \CC^{\otimes I_1}\otimes \Shv(X^{J_1}) \text{ and }
V_2\otimes \CM_2\in \CC^{\otimes I_2}\otimes \Shv(X^{J_2}),$$
the tensor product of their images in $\CC_\Ran$ is the image of the object
$$(V_1\otimes V_2)\otimes (\CM_1\boxtimes \CM_2)\in \CC^{\otimes (I_1\sqcup I_2)}\otimes \Shv(X^{J_1\sqcup J_2}).$$

\medskip

We will denote the resulting monoidal operation on $\CC_\Ran$ by
$$\CV_1,\CV_2\mapsto \CV_1\star \CV_2.$$

\medskip

We denote the unit object by $\one_{\CC_\Ran}$. It arises
from $(\Id:\emptyset \to \emptyset) \in \on{TwArr}(\fSet)$ and the
corresponding map 
$$\Vect = \CC^{\otimes \emptyset} \otimes \Shv(X^{\emptyset}) \to \CC_\Ran.$$

\sssec{}\label{sss:rep ran lim radj}

Let $(\psi:I \to J) \in \on{TwArr}(\fSet)$ be given.

\medskip 

We denote by 
$$
\on{ins}_\psi:\CC^{\otimes I} \otimes \Shv(X^J) \to 
\CC_\Ran
$$ 
the corresponding functor. 

\medskip

In the important special case $\psi = \on{Id}_I:I \to I$, we 
use the notation $\on{ins}_I$ in place of $\on{ins}_{\on{Id}_I}$.

\sssec{}

We will apply the above discussion to the case $\CC = \Rep(\cG)$.

\medskip

We denote the resulting (symmetric monoidal) category by $\Rep(\cG)_\Ran$. 

\ssec{A Ran version of the Hecke action}

\sssec{} \label{sss:Ran action}

We now claim that the datum of the functors \eqref{e:Hecke !} together with the compatibilities 
\eqref{e:IJ com} allow to define an action of $\Rep(\cG)_\Ran$ on $\Shv(\CZ \times \Bun_G)$.

\medskip

Namely, for $(I\overset{\psi}\to J)\in \on{TwArr}(\fSet)$ and
$$V\otimes \CM\in \Rep(\cG)^{\otimes I}\otimes \Shv(X^J),$$
we let the corresponding endofunctor of $\Shv(\CZ\times \Bun_G)$ be the composition
\begin{multline} \label{e:Hecke action !}
\Shv(\CZ\times \Bun_G) \overset{\on{Id}_\CZ \boxtimes \sH(V,-)}\longrightarrow \Shv(\CZ\times \Bun_G\times X^I)
\overset{-\sotimes p_3^!((\Delta_\psi)_*(\CM))}\longrightarrow \\
\to \Shv(\CZ\times \Bun_G\times X^I)
\overset{(p_{1,2})_*}\longrightarrow 
\Shv(\CZ\times \Bun_G).
\end{multline}

\sssec{}

Note, however, that using \eqref{e:ULA Hecke}, and the fact that $X$ is proper, 
we can rewrite the expression in \eqref{e:Hecke action !} as 
\begin{multline} \label{e:Hecke action *}
\Shv(\CZ\times \Bun_G) \overset{\on{Id}_\CZ \boxtimes \sH^l(V,-)}\longrightarrow \Shv(\CZ\times \Bun_G\times X^I)
\overset{-\overset{*}\otimes p_3^*((\Delta_\psi)_*(\CM))}\longrightarrow \\
\to \Shv(\CZ\times \Bun_G\times X^I)
\overset{(p_{1,2})_!}\longrightarrow \Shv(\CZ\times \Bun_G).
\end{multline}

\sssec{}

The interpretation of the Hecke action via \eqref{e:Hecke action !} implies that it commutes with
!-pullbacks and $\blacktriangle$-pushforwards along maps $f:\CZ_1\to \CZ_2$. And the interpretation of the Hecke 
action via \eqref{e:Hecke action *} implies that it commutes with *-pullbacks and !-pushforwards along 
maps $f:\CZ_1\to \CZ_2$.

\medskip

This implies that for a given $\CV\in \Rep(\cG)_{\Ran}$ its Hecke action is a functor both 
\emph{defined and codefined by a kernel}, (see \secref{sss:ker} for what this means). 

\medskip

We will denote the resulting endofunctor of 
$\Shv(\Bun_G)$ by $\sH_\CV$, and of $\Shv(\CZ\times \Bun_G)$
by $\on{Id}_\CZ \boxtimes \sH_\CV$ (see \secref{sss:ker} for the $\otimes$ notation). 

\medskip

We will refer to endofunctors of $\Shv(\Bun_G)$ (or, more generally, $\Shv(\CZ\times \Bun_G)$) that arise in this
way as \emph{integral Hecke functors}. 

\sssec{} \label{sss:conv notation} \label{sss:kernel notation}

For later use, we introduce the following notation. For $\CV\in \Rep(\cG)_{\Ran}$ we let 
$$\CK_\CV\in \Shv(\Bun_G \times \Bun_G)$$ 
denote the object equal to
$$(\on{Id}_{\Bun_G}\boxtimes \sH_\CV)(\on{ps-u}_{\Bun_G}),$$
where 
$$\on{ps-u}_{\Bun_G}:=(\Delta_{\Bun_G})_!(\ul\sfe_{\Bun_G}).$$

\sssec{}

By \thmref{t:Hecke on Nilp}, the action of $\Rep(\cG)_{\Ran}$ on $\Shv(\Bun_G)$
preserves the full subcategory
$$\Shv_{\Nilp}(\Bun_G)\subset \Shv(\Bun_G).$$

\ssec{The dual category of $\Rep(\cG)_\Ran$}

For what follows we will need to recall some constructions pertaining to duality on $\Rep(\cG)_\Ran$.
 
\medskip

We refer the reader to \cite[Sects. 11.3 and 11.4]{AGKRRV1} for a detailed discussion. 
 
\sssec{}

Let $\CC$ be a general symmetric monoidal DG category. 

\medskip

Assume that $\CC$ is dualizable, and that for every $(I\overset{\psi}\to J)\in \on{TwArr}(\on{fSet})$, the functor
$$\on{mult}^\psi:\CC^{\otimes I}\to \CC^{\otimes J}$$
is such that the dual functor $(\CC^{\otimes J})^\vee\to (\CC^{\otimes I})^\vee$ 
admits a left adjoint. 

\medskip

In this case, one shows that the category $\CC_\Ran$ is dualizable (see, e.g., \cite[Chapter 1, Proposition 6.3.4]{GR}).

\sssec{}

The dual category $(\CC_\Ran)^\vee$ is the category of continuous functors $\CC_\Ran\to \Vect$, and hence it can be described as 
\begin{equation} \label{e:dua Ran as limit prel prel}
\underset{(I\overset{\psi}\to J)\in \on{TwArr}(\fSet)^{op}}{\on{lim}} \, (\CC^{\otimes I} \otimes \Shv(X^J))^\vee,
\end{equation}
where the limit is formed using the functors dual to the ones used in the formation of the colimit in \secref{sss:Rep Ran}. 

\medskip

Using the Verdier self-duality on $\Shv(X^J)$, we can rewrite 
\begin{equation} \label{e:dua Ran as limit prel}
(\CC_\Ran)^\vee \simeq \underset{(I\overset{\psi}\to J)\in \on{TwArr}(\fSet)^{op}}{\on{lim}} \, 
\bMaps_{\DGCat}(\CC^{\otimes I},\Shv(X^J)),
\end{equation}
where $\bMaps_{\DGCat}(-,-)$ stands for the DG category of continuous $\sfe$-linear functor
between two objects of $\DGCat$.

\medskip

Explicitly, the transition functors in \eqref{e:dua Ran as limit prel} are defined as follows. For a morphism in $\on{TwArr}(\fSet)$
given by \eqref{e:mor Tw arr}, the corresponding functor
$$\bMaps_{\DGCat}(\CC^{\otimes I_2},\Shv(X^{J_2})) \to \bMaps_{\DGCat}(\CC^{\otimes I_1},\Shv(X^{J_1}))$$
is given by precomposition \eqref{e:mult phi} and postcomposition with 
$$(\Delta_{\phi_J})^!:\Shv(X^{J_2})\to \Shv(X^{J_1}),$$
which is the functor dual to \eqref{e:Delta phi}, under the Verdier self-duality of $\Shv(X^?)$. 

\begin{rem} \label{r:Ran self dual}
Suppose for a moment that $\CC$ is rigid (see \cite[Chapter 1, Sect. 9.1]{GR} for what this means). 
In this case, we have a natural identification $\CC^\vee\simeq \CC$, and 
we can further rewrite the right-hand side
in \eqref{e:dua Ran as limit prel} as 
$$\underset{(I\overset{\psi}\to J)\in \on{TwArr}(\fSet)^{op}}{\on{lim}} \, 
\CC^{\otimes I}\otimes \Shv(X^J),$$
where the limit is formed using the functors \emph{right adjoint} to the ones used in the formation of the colimit in \secref{sss:Rep Ran}. 
Hence, the above limit is isomorphic to the colimit
$$\underset{(I\overset{\psi}\to J)\in \on{TwArr}(\fSet)}{\on{colim}} \, 
\CC^{\otimes I}\otimes \Shv(X^J)$$
(see \cite[Chapter 1, Proposition 2.5.7]{GR}), i.e., to $\CC_\Ran$ itself. 

\medskip

This implies that for $\CC$ rigid, the category $\CC_\Ran$ is naturally self-dual\footnote{In fact, the category $\CC_\Ran$ is itself
rigid, see \cite[Sect. 11.3]{AGKRRV1}.}. However, we will not use this
self-duality for the purposes of this paper. 
\end{rem} 

\sssec{} \label{sss:DGCat fSet}

Consider the $(\infty,2)$-category 
$$\DGCat^{\fSet}:=\on{Funct}(\fSet,\DGCat).$$

There will be several DG categories of interest in this paper that will arise as
$$\bMaps_{\DGCat^{\fSet}}(\fC_1,\fC_2)$$
for some particular $\fC_1,\fC_2\in \DGCat^{\fSet}$.

\medskip

Concretely, an object of $\bMaps_{\DGCat^{\fSet}}(\fC_1,\fC_2)$ is a collection of functors between DG categories
$$\fC_1(I)\to \fC_2(I), \quad I\in \fSet$$
that make the diagrams
$$
\CD 
\fC_1(I) @>>>  \fC_2(I) \\
@V{\fC_1(\psi)}VV @VV{\fC_2(\psi)}V \\
\fC_1(J) @>>>  \fC_2(J)
\endCD
$$
commute for $I\overset{\psi}\to J$, along with a homotopy coherent system of higher compatibilities.

\sssec{} \label{sss:obj in DGCat fSet}

Here are the first few objects of $\DGCat^{\fSet}$ that we will need.

\medskip

One is the object denoted $\CC^{\otimes \fSet}$ and defined by
$$I\in \fSet \, \rightsquigarrow\, \CC^{\otimes I}\in \DGCat,$$
where the functoriality is furnished by the symmetric monoidal structure on $\CC$.

\medskip

Another object, denoted $\Shv^!(X^{\fSet})$, is defined by
$$I\in \fSet \, \rightsquigarrow\, \Shv(X^I)\in \DGCat,$$
where for $I \overset{\phi}\to J$, the corresponding functor
$\Shv(X^I)\to \Shv(X^J)$ is $(\Delta_\phi)^!$.

\sssec{}

Consider the category 
$$\bMaps_{\DGCat^{\fSet}}(\CC^{\otimes \fSet},\Shv^!(X^{\fSet})).$$

Note that this category identifies with the limit \eqref{e:dua Ran as limit prel} (see e.g. \cite[Lemma 1.3.12]{GKRV}).

\sssec{} \label{sss:dual of Ran}

Thus, to summarize, we obtain a canonical equivalence
\begin{equation} \label{e:three incarnations}
\bMaps_{\DGCat^{\fSet}}(\CC^{\otimes \fSet},\Shv^!(X^{\fSet})) \simeq (\CC_\Ran)^{\vee}.
\end{equation} 

Explicitly, given 
$$\CS:\CC_\Ran\to \Vect,$$
the corresponding system of functors
$$\CS_I:\CC^{\otimes I}\to  \Shv(X^I)$$
is recovered as follows:

\medskip

We precompose $\CS$ with $\on{ins}_I$ to obtain a functor
$$\CC^{\otimes I}\otimes \Shv(X^I)\to \Vect.$$

By Verdier duality, the datum of the latter functor is equivalent to the datum of a 
functor $\CS_I$: 
$$\CS\circ \on{ins}_I(c\otimes \CM)=\on{C}^\cdot(X^I,\CS_I(c)\sotimes \CM), \quad c\in \CC^{\otimes I},\,\CM\in \Shv(X^I).$$

\sssec{}

Vice versa, the pairing 
$$\CC_{\Ran} \otimes \bMaps_{\DGCat^{\fSet}}(\CC^{\otimes \fSet},\Shv^!(X^{\fSet})) \to \Vect$$
is explicitly given as follows:

\medskip

For an object $\{\CS_I\}\in  \bMaps_{\DGCat^{\fSet}}(\CC^{\otimes \fSet},\Shv^!(X^{\fSet}))$, the corresponding
functor $$\CS:\CC_{\Ran}\to \Vect$$ is such that for $(I\overset{\psi}\to J)\in \on{TwArr}(\on{fSet})$, the resulting functor 
$$\CC^{\otimes I} \otimes \Shv(X^J) \overset{\on{ins}_\psi}\longrightarrow \CC_{\Ran}\overset{\CS}\to \Vect,$$
equals
$$\CC^{\otimes I} \otimes \Shv(X^J) \overset{\on{mult}^\psi\otimes \on{Id}}\longrightarrow 
\CC^{\otimes J} \otimes \Shv(X^J) \overset{\CS_J\otimes \on{Id}}\longrightarrow  
\Shv(X^J) \otimes \Shv(X^J) \overset{\ev_{X^J}}\longrightarrow \Vect,$$
where $\ev_{X^J}$ is the Verdier duality pairing on $\Shv(X^J)$, i.e., 
$$\Shv(X^J) \otimes \Shv(X^J) \overset{\Delta_{X^J}^!}\to \Shv(X^J) \overset{\on{C}^\cdot(X^J,-)}\to \Vect.$$

\section{Quasi-coherent sheaves on $\LocSys^{\on{restr}}_\cG(X)$} \label{s:LocSys}

Although the statement of the Trace Conjecture does not involve Langlands duality, we will
need some of its ingredients for the proof. Indeed, one of the key tools in the proof will be the
category of quasi-coherent sheaves on the (pre)stack $\LocSys^{\on{restr}}_\cG(X)$, classifying   
local systems with restricted variation with respect to the Langlands dual group $\cG$ of $G$. 

\ssec{The (pre)stack $\LocSys^{\on{restr}}_\cG(X)$}

We start by recalling the definition of the prestack $$\LocSys^{\on{restr}}_\cG(X),$$ following \cite[Sect. 1.4]{AGKRRV1}.

\sssec{}

For a test affine (derived) scheme $S$, we let $\Maps(S,\LocSys^{\on{restr}}_\cG(X))$ be the space of right t-exact 
symmetric monoidal functors 
$$\Rep(\cG)\to \QCoh(S)\otimes \qLisse(X).$$

\medskip

According to \cite[Theorem 1.4.5]{AGKRRV1}, the prestack $\LocSys^{\on{restr}}_\cG(X)$ can be written
as the quotient $\CZ/\cG$, where $\CZ$ is a disjoint union of \emph{formal affine schemes} locally
almost of finite type (over the field of coefficients $\sfe$). 

\sssec{} \label{sss:ls recap}

The main results of this paper will be based on considering the (symmetric monoidal)
DG category 
$$\QCoh(\LocSys^{\on{restr}}_\cG(X)).$$

\medskip

We will now explain a certain feature that this category possesses, which is a consequence of
properties of $\LocSys^{\on{restr}}_\cG(X)$ as a prestack.

\medskip

\noindent(i) First, according to \cite[Lemma 7.3.2 and Sect. 7.9.1]{AGKRRV1}, the diagonal map
$$\Delta_{\LocSys^{\on{restr}}_\cG(X)}:\LocSys^{\on{restr}}_\cG(X)\to \LocSys^{\on{restr}}_\cG(X)\times \LocSys^{\on{restr}}_\cG(X)$$
is affine, so the functor
$$(\Delta_{\LocSys^{\on{restr}}_\cG(X)})_*:\QCoh(\LocSys^{\on{restr}}_\cG(X)) \to
\QCoh(\LocSys^{\on{restr}}_\cG(X)\times \LocSys^{\on{restr}}_\cG(X))$$
is continuous.

\medskip

\noindent(ii) Second, according to \cite[Corollary 7.1.8(b) and Sect. 7.9.1]{AGKRRV1}, the DG category $\QCoh(\LocSys^{\on{restr}}_\cG(X))$ is dualizable.
By \cite[Chapter 3, Proposition 3.1.7]{GR}, this implies that for any prestack $\CY$ over $\sfe$, the functor of external tensor product
$$\QCoh(\LocSys^{\on{restr}}_\cG(X)) \otimes \QCoh(\CY)\to
\QCoh(\LocSys^{\on{restr}}_\cG(X)\times \CY)$$
is an equivalence.

\medskip

In particular, we can view
$$(\Delta_{\LocSys^{\on{restr}}_\cG(X)})_*(\CO_{\LocSys^{\on{restr}}_\cG(X)})$$
as an object of 
$$\QCoh(\LocSys^{\on{restr}}_\cG(X))\otimes \QCoh(\LocSys^{\on{restr}}_\cG(X)).$$

\medskip

\noindent(iii) And third, according to \cite[Proposition 7.5.4 and Sect. 7.9.1]{AGKRRV1}, the above object
$$(\Delta_{\LocSys^{\on{restr}}_\cG(X)})_*(\CO_{\LocSys^{\on{restr}}_\cG(X)}) \in 
\QCoh(\LocSys^{\on{restr}}_\cG(X))\otimes \QCoh(\LocSys^{\on{restr}}_\cG(X))$$
defines the unit of a self-duality on $\QCoh(\LocSys^{\on{restr}}_\cG(X))$.

\begin{rem}

Points (i) and (ii) above mean that the prestack $\LocSys^{\on{restr}}_\cG(X)$ is \emph{semi-passable}
in the terminology of \cite[Sect. 7.4.3]{AGKRRV1}. 

\medskip

Point (iii) is the combination of \cite[Lemma 7.4.2]{AGKRRV1},
which says that $\QCoh(-)$ on a semi-passable prestack is a semi-rigid symmetric monoidal category
(see \cite[Appendix C]{AGKRRV1} for what this means), combined with the description of the unit of the canonical
self-duality on a semi-rigid symmetric monoidal category (see \cite[Lemma C.3.3]{AGKRRV1}). 

\end{rem}

\sssec{The functor $\Gamma_!$} \label{sss:coSect}

We can view the structure sheaf 
$\CO_{\LocSys^{\on{restr}}_\cG(X)}\in \QCoh(\LocSys^{\on{restr}}_\cG(X))$ 
as defining a functor
\begin{equation} \label{e:structure sheaf LocSys}
\Vect\to \QCoh(\LocSys^{\on{restr}}_\cG(X)).
\end{equation} 

\noindent Note that the object 
$\CO_{\LocSys^{\on{restr}}_\cG(X)}\in \QCoh(\LocSys^{\on{restr}}_\cG(X))$ is not
compact, so the functor of global sections
$$\Gamma(\LocSys^{\on{restr}}_\cG(X),-):\QCoh(\LocSys^{\on{restr}}_\cG(X))\to \Vect,$$
the right adjoint to \eqref{e:structure sheaf LocSys}, is \emph{not} continuous. 

\medskip

However, due to the self-duality of $\QCoh(\LocSys^{\on{restr}}_\cG(X))$ of \secref{sss:ls recap}(iii), we can consider the functor
\emph{dual} to \eqref{e:structure sheaf LocSys}, which is a functor
\begin{equation} \label{e:cosec LocSys}
\QCoh(\LocSys^{\on{restr}}_\cG(X))\to \Vect,
\end{equation} 
denoted $\Gamma_!(\LocSys^{\on{restr}}_\cG(X),-)$.
 
\medskip 

We refer the reader to \cite[Sects. 7.6 and 7.7]{AGKRRV1} for a more detailed discussion of
this functor.

 
%
%
%
%
%
%
%
%

\sssec{The tautological objects} \label{sss:taut obj}

For a finite set $I$ and an object $V\in \Rep(\cG)^{\otimes I}$, let
$$\Ev(V)\in \QCoh(\LocSys^{\on{restr}}_\cG(X))\otimes \qLisse(X)^{\otimes I}$$
be the corresponding tautological object:

\medskip

For $S\to \LocSys^{\on{restr}}_\cG(X)$ corresponding to a symmetric monoidal functor
$$\Phi_S:\Rep(\cG)\to \QCoh(S)\otimes \qLisse(X),$$
the pullback of $\on{Ev}(V)$ to $S$, viewed as an object in $\QCoh(S)\otimes  \qLisse(X)^{\otimes I}$ equals the value on $V$ of the composition
$$\Rep(\cG)^{\otimes I}\overset{\Phi_S^{\otimes I}}\to (\QCoh(S)\otimes \qLisse(X))^{\otimes I} \to \QCoh(S)\otimes \qLisse(X)^{\otimes I},$$
where the second arrow is given by the $I$-fold tensor product functor
$$\QCoh(S)^{\otimes I}\to \QCoh(S).$$

\ssec{Description of the category $\QCoh(\LocSys^{\on{restr}}_\cG(X))$}

The prestack $\LocSys^{\on{restr}}_\cG(X))$ was defined using the symmetric monoidal 
categories $\Rep(\cG)$ and $\qLisse(X)$. Therefore, it would not be very surprising to have
a description of the category $\QCoh(\LocSys^{\on{restr}}_\cG(X))$ purely in terms of functors
between the above two categories.

\medskip

In this subsection, we will provide such a description, following \cite{AGKRRV1}.

\sssec{} 

Recall the category $\DGCat^{\fSet}$, see \secref{sss:DGCat fSet}, and the object
$$\qLisse(X)^{\otimes \fSet}\in \DGCat^{\fSet},$$
see \secref{sss:obj in DGCat fSet}.

\medskip

Note that we could also consider the object $\qLisse(X^{\fSet})$ of $\DGCat^{\fSet}$:
$$I\in \fSet \, \rightsquigarrow\, \qLisse(X^I)\in \DGCat.$$

\medskip

We have a naturally defined map in $\DGCat^{\fSet}$
\begin{equation} \label{e:Kun fSet}
\qLisse(X)^{\otimes \fSet} \to \qLisse(X^{\fSet}). 
\end{equation}

\medskip

However, from \cite[Theorem E.9.9 and Corollary E.4.7]{AGKRRV1}, we obtain:

\begin{lem} \label{l:Kun fSet}
The map \eqref{e:Kun fSet} is an isomorphism.
\end{lem}



\sssec{}\label{sss:qcoh recap}

Consider now the DG category $\bMaps_{\DGCat^{\fSet}}(\Rep(\cG)^{\otimes \fSet},\qLisse(X)^{\otimes \fSet})$, and the following functor,
to be denoted 
\begin{equation} \label{e:descr QCoh}
\coLoc:\QCoh(\LocSys^{\on{restr}}_\cG(X)) \to
\bMaps_{\DGCat^{\fSet}}(\Rep(\cG)^{\otimes \fSet},\qLisse(X)^{\otimes \fSet}).
\end{equation}

Namely, $\coLoc$ sends $\CF\in \QCoh(\LocSys^{\on{restr}}_\cG(X))$ to the collection of functors
$$\CF_I:\Rep(\cG)^{\otimes I}\to \qLisse(X)^{\otimes I}$$
defined by
$$\CF_I(V):=(\Gamma_!(\LocSys^{\on{restr}}_\cG(X),-) \otimes \on{Id}_{\qLisse(X)^{\otimes I}})
(\CF\otimes \on{Ev}(V)),$$
where $\Gamma_!$ is as in \secref{sss:coSect}, and $\on{Ev}(V)$ is as in \secref{sss:taut obj},
and we view $\CF\otimes \on{Ev}(V)$ as an object of
$$\LocSys^{\on{restr}}_\cG(X)\otimes \qLisse(X)^{\otimes I}.$$

\sssec{}

The following is a reformulation of one of the main results of the paper \cite{AGKRRV1}:

\begin{thm} \label{t:descr QCoh}
The functor $\coLoc$ is an equivalence.
\end{thm}

\begin{proof}

Recall the category 
$$\Rep(\cG)^{\otimes X\on{-lisse}}=\coHom(\Rep(\cG),\qLisse(X)),$$
see \cite[Sects. 8.2.4 and 8.4.2]{AGKRRV1}.  

\medskip

According to \cite[Lemma 8.2.8(b)]{AGKRRV1}, the category
$$\bMaps_{\DGCat^{\fSet}}(\Rep(\cG)^{\otimes \fSet},\qLisse(X)^{\otimes \fSet})$$
identifies with the category 
$$\bMaps_{\DGCat}(\coHom(\Rep(\cG),\qLisse(X)),\Vect).$$

Consider now the functor
\begin{equation} \label{e:pre-Loc AGKRRV}
\coHom(\Rep(\cG),\qLisse(X))\to \QCoh(\LocSys^{\on{restr}}_\cG(X))
\end{equation}
of \cite[Equation (8.8)]{AGKRRV1}. 

\medskip

Unwinding the definitions, we obtain that the composite functor
\begin{multline*}
\bMaps_{\DGCat}(\QCoh(\LocSys^{\on{restr}}_\cG(X)),\Vect) \simeq 
\QCoh(\LocSys^{\on{restr}}_\cG(X))^\vee \overset{\text{\secref{sss:ls recap}(iii)}}\simeq \\
\simeq \QCoh(\LocSys^{\on{restr}}_\cG(X)) \overset{\coLoc}\longrightarrow
\bMaps_{\DGCat^{\fSet}}(\Rep(\cG)^{\otimes \fSet},\qLisse(X)^{\otimes \fSet}) \simeq \\
\simeq \bMaps_{\DGCat}(\coHom(\Rep(\cG),\qLisse(X)),\Vect)
\end{multline*}
identifies with the dual of \eqref{e:pre-Loc AGKRRV}.

\medskip

Now, \cite[Theorem 8.3.7]{AGKRRV1} says that \eqref{e:pre-Loc AGKRRV} is an equivalence.
Hence, so is $\coLoc$.

\end{proof}

%
%

\ssec{Localization} \label{ss:Loc}

In this subsection we will recall how $\Rep(\cG)_\Ran$ is related to $\QCoh(\LocSys^{\on{restr}}_\cG(X))$.

\sssec{}

Consider the symmetric monoidal category $\Rep(\cG)_\Ran$, 
see \secref{sss:Rep Ran}. We are going
to construct a symmetric monoidal functor
\begin{equation} \label{e:functor Loc}
\on{Loc}:\Rep(\cG)_\Ran\to \QCoh(\LocSys^{\on{restr}}_\cG(X))
\end{equation}
that plays a key role in this work.

\medskip

For an individual 
$$(I\overset{\psi}\to J) \in \on{TwArr}(\on{fSet}),$$
the corresponding functor
$$\on{Loc}^{I\overset{\psi}\to J}:\Rep(\cG)^{\otimes I}\otimes \Shv(X^J)\to  \QCoh(\LocSys^{\on{restr}}_\cG(X))$$
sends $V\in \Rep(\cG)^{\otimes I}$ to the functor $\Shv(X^J)\to \QCoh(\LocSys^{\on{restr}}_\cG(X))$ equal to the composition 
\begin{multline*}
\Shv(X^J) \overset{\on{Ev}(V)\otimes \on{Id}}\longrightarrow \QCoh(\LocSys^{\on{restr}}_\cG(X))\otimes \qLisse(X)^{\otimes I} \otimes  \Shv(X^J) 
\overset{\on{Id}\otimes \on{mult}^\psi\otimes \on{Id}}\longrightarrow \\
\to \QCoh(\LocSys^{\on{restr}}_\cG(X))\otimes \qLisse(X)^{\otimes J} \otimes  \Shv(X^J) \longrightarrow \\
\to \QCoh(\LocSys^{\on{restr}}_\cG(X))\otimes \Shv(X^J) 
\overset{\on{Id}\otimes \on{C}_c^\cdot(X^J,-)}\longrightarrow \QCoh(\LocSys^{\on{restr}}_\cG(X)),
\end{multline*}
where the third arrow using the canonical action of $\qLisse(X)^{\otimes J}\simeq \qLisse(X^J)$ on $\Shv(X^J)$,
see \secref{sss:* and ! lisse}. 

\medskip

It is easy to see that the functors $\on{Loc}^{I\overset{\psi}\to J}$ indeed combine to define a functor, to be denote $\Loc$, 
as in \eqref{e:functor Loc}. Moreover, this functor carries a naturally defined symmetric monoidal structure. 

\sssec{} 

Consider the dual functor
$$\Loc^{\vee}:\QCoh(\LocSys^{\on{restr}}_\cG(X))^\vee \to (\Rep(\cG)_\Ran)^\vee.$$

Recall now that the category $\QCoh(\LocSys^{\on{restr}}_\cG(X))$ is self-dual (see \secref{sss:ls recap}(iii)),
and that the category $(\Rep(\cG)_\Ran)^\vee$ can be described as 
$\bMaps_{\DGCat^{\fSet}}(\CC^{\otimes \fSet},\Shv^!(X^{\fSet}))$ (see \eqref{e:three incarnations}).

\medskip

Thus, we can view $\Loc^\vee$ as a functor
$$\QCoh(\LocSys^{\on{restr}}_\cG(X)) \to \bMaps_{\DGCat^{\fSet}}(\CC^{\otimes \fSet},\Shv^!(X^{\fSet})).$$

Unwinding the definitions, we obtain that $\Loc^\vee$ identifies with the composition
\begin{multline*}
\QCoh(\LocSys^{\on{restr}}_\cG(X)) \overset{\coLoc}\to 
\bMaps_{\DGCat^{\fSet}}(\CC^{\otimes \fSet},\qLisse(X)^{\otimes \fSet})\simeq \\
\simeq \bMaps_{\DGCat^{\fSet}}(\CC^{\otimes \fSet},\qLisse(X^{\fSet}))
\to \bMaps_{\DGCat^{\fSet}}(\CC^{\otimes \fSet},\Shv^!(X^{\fSet})),
\end{multline*}
where the last arrow is given by the (fully faithful) functor \eqref{e:! emb}. 

\medskip

Hence, combining with \thmref{t:descr QCoh}, we obtain: 

\begin{cor} \hfill \label{c:Loc dual}

\smallskip

\noindent{\em(a)} The functor $\Loc^\vee$ is fully faithful. 

\smallskip

\noindent{\em(b)}  An object $\CS\in (\Rep(\cG)_\Ran)^\vee$ lies in the essential image of $\Loc^\vee$ if and
only if the corresponding family of functors $\{\CS_I\}$
$$\Rep(\cG)^{\otimes I}\to \Shv(X^I)$$
takes values in $\qLisse(X^I)\subset \Shv(X^I)$. 

\end{cor}


\sssec{}

Let 
$$(\Rep(\cG)_\Ran)^\vee_{\qLisse} \subset (\Rep(\cG)_\Ran)^\vee$$
be the full subcategory that under the equivalence \eqref{e:three incarnations} corresponds to
$$\bMaps_{\DGCat^{\fSet}}(\CC^{\otimes \fSet},\qLisse(X^{\fSet})) \subset \bMaps_{\DGCat^{\fSet}}(\CC^{\otimes \fSet},\Shv^!(X^{\fSet})),$$
where the embedding $\qLisse(X^{\fSet})\hookrightarrow \Shv^!(X^{\fSet})$ is \eqref{e:! emb}. 

\medskip

Thus, \corref{c:Loc dual} can be reformulated as follows:

\begin{cor} \label{c:Loc dual bis}
The functor $\Loc^\vee$ defines an equivalence 
$$\QCoh(\LocSys^{\on{restr}}_\cG(X)) \to (\Rep(\cG)_\Ran)^\vee_{\qLisse}.$$
\end{cor}

\ssec{Beilinson's spectral projector} \label{ss:proj}

In this subsection we will consider a certain object
$$\sR\in \Rep(\cG)_\Ran,$$
that will play a crucial role in the proof of the main results in this paper. 

\sssec{}

Let 
$$\sR_{\Rep(\cG),\Ran}\in \Rep(\cG)_\Ran\otimes \Rep(\cG)_\Ran$$
be the object introduced in \cite[Sect. 11.5]{AGKRRV1}.

\medskip

The quickest way to define it is as the value on the unit object 
$$\one_{\Rep(\cG)}\in  \Rep(\cG)_\Ran$$
of the right adjoint to the monoidal operation
$$\Rep(\cG)_\Ran\otimes  \Rep(\cG)_\Ran\to  \Rep(\cG)_\Ran.$$

Since the above right adjoint carries a right-lax symmetric monoidal structure, 
the object $\sR_{\Rep(\cG),\Ran}$ is naturally a commutative algebra in
$\Rep(\cG)_\Ran\otimes \Rep(\cG)_\Ran$. 

\begin{rem}
As was mentioned in Remark \ref{r:Ran self dual}, the symmetric monoidal category 
$\Rep(\cG)_\Ran$ is rigid. The object $\sR_{\Rep(\cG),\Ran}$
is the unit of the canonical self-duality defined for any rigid symmetric monoidal category.
\end{rem}

%
%
%
%
%

\sssec{} \label{sss:R ls}

We define the commutative algebra object 
$$\sR_{\LocSys^{\on{restr}}_\cG(X)}\in \QCoh(\LocSys^{\on{restr}}_\cG(X))\otimes \Rep(\cG)_\Ran$$
to be
$$(\on{Loc}\otimes \on{Id})(\sR_{\Rep(\cG),\Ran}).$$

\sssec{}

Consider also the commutative algebra object 
\begin{multline} \label{e:diagonal LocSys}
(\on{Id}\otimes \on{Loc})(\sR_{\LocSys^{\on{restr}}_\cG(X)})=\\
=(\on{Loc}\otimes \on{Loc})(\sR_{\Rep(\cG),\Ran})\in \QCoh(\LocSys^{\on{restr}}_\cG(X))\otimes \QCoh(\LocSys^{\on{restr}}_\cG(X)).
\end{multline}

\medskip

\noindent The following is \cite[Theorem 12.7.4]{AGKRRV1}:

\begin{thm} \label{t:diagonal LocSys}
The image of $(\on{Id}\otimes \on{Loc})(\sR_{\LocSys^{\on{restr}}_\cG(X)})$ under
$$\QCoh(\LocSys^{\on{restr}}_\cG(X))\otimes \QCoh(\LocSys^{\on{restr}}_\cG(X))\overset{\sim}\to
\QCoh(\LocSys^{\on{restr}}_\cG(X)\times \LocSys^{\on{restr}}_\cG(X)),$$
identifies canonically with 
$$(\Delta_{\LocSys^{\on{restr}}_\cG(X)})_*(\CO_{\LocSys^{\on{restr}}_\cG(X)}),$$
as commutative algebra objects.  
\end{thm}

\begin{rem}

In other words, \thmref{t:diagonal LocSys} says that we have a canonical isomorphism
$$(\Loc \otimes \Loc)(\sR_{\Rep(\cG),\Ran})\simeq (\Delta_{\LocSys^{\on{restr}}_\cG(X)})_*(\CO_{\LocSys^{\on{restr}}_\cG(X)}),$$
as commutative algebra objects in $\QCoh(\LocSys^{\on{restr}}_\cG(X)\times \LocSys^{\on{restr}}_\cG(X))$.

\end{rem}

\sssec{} \label{sss:the projector}

Denote 
$$\sR:=(\Gamma_!(\LocSys^{\on{restr}}_\cG(X),-)\otimes \on{Id})(\sR_{\LocSys^{\on{restr}}_\cG(X)})\in \Rep(\cG)_\Ran,$$
where $\Gamma_!(\LocSys^{\on{restr}}_\cG(X),-)$ is as in \secref{sss:coSect}.

\medskip 

In other words,
$$\sR=((\Gamma_!(\LocSys^{\on{restr}}_\cG(X),-) \circ \Loc)\otimes \on{Id})(\sR_{\Rep(\cG),\Ran}).$$

Recall (see \cite[Sect. 7.6.1]{AGKRRV1}) that the functor $\Gamma_!(\LocSys^{\on{restr}}_\cG(X),-)$ carries
a canonically defined (non-unital) right-lax symmetric monoidal structure. Hence, $\sR$ is naturally a commutative algebra in $\Rep(\cG)_\Ran$.

\sssec{}

We claim: 

\begin{cor} \label{c:diagonal LocSys}
The object $$\Loc(\sR)\in \QCoh(\LocSys^{\on{restr}}_\cG(X))$$ identifies canonically with $\CO_{\LocSys^{\on{restr}}_\cG(X)}$,
as commutative algebra objects.  
\end{cor}

\begin{proof}

Applying \thmref{t:diagonal LocSys}, it suffices to show that 
$$(\Gamma_!(\LocSys^{\on{restr}}_\cG(X),-)\otimes \on{Id}) \circ (\Delta_{\LocSys^{\on{restr}}_\cG(X)})_*(\CO_{\LocSys^{\on{restr}}_\cG(X)})
\simeq \CO_{\LocSys^{\on{restr}}_\cG(X)}$$
as commutative algebra objects.  

\medskip

We claim that 
$$(\Gamma_!(\LocSys^{\on{restr}}_\cG(X),-)\otimes \on{Id}) \circ (\Delta_{\LocSys^{\on{restr}}_\cG(X)})_*
\simeq \on{Id}$$
as symmetric monoidal endofunctors of $\QCoh(\LocSys^{\on{restr}}_\cG(X))$. 

\medskip

This is a feature of any semi-rigid category, see \cite[Appendix C]{AGKRRV1}. Indeed, 
the identification as plain endofunctors follows by passage to the dual functors in 
\begin{multline*}
\QCoh(\LocSys^{\on{restr}}_\cG(X)) \simeq \Vect\otimes \QCoh(\LocSys^{\on{restr}}_\cG(X))
\overset{\CO_{\LocSys^{\on{restr}}_\cG(X)}\otimes \on{Id}}\longrightarrow \\
\to \QCoh(\LocSys^{\on{restr}}_\cG(X))\otimes \QCoh(\LocSys^{\on{restr}}_\cG(X))
\overset{\on{mult}_{\QCoh(\LocSys^{\on{restr}}_\cG(X))}}\longrightarrow \QCoh(\LocSys^{\on{restr}}_\cG(X)),
\end{multline*}
where we identify the dual of $\on{mult}_{\QCoh(\LocSys^{\on{restr}}_\cG(X))}$ with its adjoint
(see \cite[Lemma C.3.5]{AGKRRV1}).

\medskip

The compatibility with the symmetric monoidal structures follows by unwinding the definition of 
the symmetric monoidal structure on $\Gamma_!(\LocSys^{\on{restr}}_\cG(X),-)$ in 
\cite[Sect. C.3.8]{AGKRRV1}. 

\end{proof}

%
%



\section{The reciprocity law for shtuka cohomology} \label{s:shtuka}

The main result of this section is \corref{c:Cong cor}, which asserts
that the functor $\Sht$ introduced below  
factors through the localization functor $\Loc$. We will deduce it from a theorem
of C.~Xue on lisseness of shtuka cohomology.  

\medskip 

As an application, we construct an object 
$\Drinf \in \QCoh(\LocSys^{\on{restr}}_{\cG})$ that encodes the cohomology of shtuka moduli spaces. 

\ssec{Functorial shtuka cohomology} \label{ss:sht}

In this subsection we will interpret shtuka cohomology as a functor 
$$\Sht:\Rep(\cG)_\Ran \to \Vect,$$
i.e., as an object of the category $(\Rep(\cG)_\Ran)^\vee$.

\sssec{} \label{sss:sht}

Recall the functor
$$\CV\in \Rep(\cG)_\Ran\, \rightsquigarrow \CK_\CV\in \Shv(\Bun_G\times \Bun_G),$$
see \secref{sss:kernel notation}. 

\medskip

Let $\on{Graph}_{\Frob_{\Bun_G}}:\Bun_G\to \Bun_G\times \Bun_G$
be the graph of Frobenius, i.e., the map
$$\Bun_G \overset{\Delta_{\Bun_G}}\longrightarrow \Bun_G\times \Bun_G
\overset{\Frob_{\Bun_G}\times \on{Id}}\longrightarrow \Bun_G\times \Bun_G.$$

\medskip

We define the functor $\Sht:\Rep(\cG)_\Ran \to \Vect$ as the composition 
$$\Rep(\cG)_\Ran \overset{\CV \mapsto \CK_\CV}{\longrightarrow} 
\Shv(\Bun_G \times \Bun_G) \overset{(\on{Graph}_{\Frob_{\Bun_G}})^*}{\longrightarrow}
\Shv(\Bun_G) \overset{\on{C}^\cdot_c(\Bun_G,-)}{\longrightarrow} \Vect.$$

\sssec{}

We will prove:

\begin{thm} \label{t:Cong cor}
The object $\Sht\in (\Rep(\cG)_\Ran)^\vee$ belongs to the full subcategory
$$(\Rep(\cG)_\Ran)^\vee_{\qLisse}\subset (\Rep(\cG)_\Ran)^\vee.$$
\end{thm}

Applying \corref{c:Loc dual bis}, from \thmref{t:Cong cor}, we obtain:

\begin{cor} \label{c:Cong cor}
The functor $\Sht:\Rep(\cG)_\Ran \to \Vect$ factors as $\Sht_{\Loc} \circ \Loc$ for a uniquely defined functor 
$\Sht_{\Loc}:\QCoh(\LocSys^{\on{restr}}_\cG(X)) \to \Vect$.
\end{cor}

\ssec{Proof of \thmref{t:Cong cor} and relation to the usual shtuka cohomology}

Once we relate the functor $\Sht$ to shtuka cohomology, the proof of \thmref{t:Cong cor} will be almost 
immediate from a recent theorem of C.~Xue \cite{Xue2}. 

\sssec{} \label{sss:classical shtuka}

For a finite set $I$, consider the functor 
$$\Rep(\cG)^{\otimes I} \to \Shv(X^I),$$
to be denoted $\Sht_I$, that sends $V\in \Rep(\cG)^{\otimes I}$ to
\begin{equation} \label{e:classical shtuka}
\Sht_I(V):=(p_{2})_!\circ (\on{Graph}_{\Frob_{\Bun_G}})^*\circ (\on{Id}_{\Bun_G}\boxtimes \sH^l(V,-)) \circ \Delta_!(\ul\sfe_{\Bun_G}).
\end{equation} 

\medskip

The above functor $\Sht_I$ is the usual functor of (compactly supported) shtuka cohomology,
studied by \cite{VLaf}.

\sssec{}

We now quote the following crucial result of \cite{Xue2}:

\begin{thm} \label{t:Cong}
The functor $\Sht_I$ takes values in $\qLisse(X^I) \subset \Shv(X^I)$.
\end{thm}

\sssec{}

Let us show how \thmref{t:Cong} implies the assertion of \thmref{t:Cong cor}. 

\medskip

By base change (and using the fact that $X$ is proper), we obtain that for $I\in \fSet$, the functor
$$\Rep(\cG)^{\otimes I}\otimes \Shv(X^I) \overset{\on{ins}_I}\to \Rep(\cG)_\Ran \overset{\Sht}\to \Vect$$
is given by sending 
$$V\in \Rep(\cG)^{\otimes I},\,\CM\in \Shv(X^I)\,\, \mapsto \on{C}^\cdot\left(X^I,(\Sht_I(V)\overset{*}\otimes \omega_{X^I})\sotimes \CM\right).$$

Thus, the object of $\bMaps_{\DGCat^{\fSet}}(\Rep(\cG)^{\otimes \fSet},\Shv^!(X^{\fSet}))$ corresponding to $\Sht$
is given by
$$V\in \Rep(\cG)^{\otimes I}\,\mapsto \Sht_I(V)\overset{*}\otimes \omega_{X^I}\simeq \Sht_I(V)[2|I|],$$
see \secref{sss:dual of Ran}.

\medskip

This object belongs to $\qLisse(X^I)$ by \thmref{t:Cong}, as required.

\begin{rem}
Note that by \thmref{t:Cong}, the expression $\on{C}^\cdot\left(X^I,(\Sht_I(V)\overset{*}\otimes \omega_{X^I})\sotimes \CM\right)$,
which appears above, can be also rewritten as
$$\on{C}^\cdot\left(X^I,\Sht_I(V)\overset{*}\otimes \CM\right),$$
see \secref{sss:ULA Hecke bis}. 
\end{rem}

\begin{rem} \label{r:! vs *}
Along with the object
$\Shv^!(X^{\fSet})\in \DGCat^{\fSet}$, we can consider the object, denoted $\Shv^*(X^{\fSet})$, whose value 
on $I$ is again the category $\Shv(X^I)$, but now for $I\overset{\psi}\to J$ we use the functor
$$(\Delta_\psi)^*: \Shv(X^I)\to \Shv(X^J).$$

As in \secref{sss:* and ! lisse}, we have the fully faithful embeddings
$$\Shv^*(X^{\fSet}) \hookleftarrow \qLisse(X^{\fSet}) \hookrightarrow \Shv^!(X^{\fSet}),$$
given by \eqref{e:! emb} and \eqref{e:* emb}.

\medskip

By its construction, the system of functors $\{\Sht_I\}$ is naturally an object of the category
$\bMaps_{\DGCat^{\fSet}}(\Rep(\cG)^{\otimes \fSet},\Shv^*(X^{\fSet}))$, and 
we can view \thmref{t:Cong} as saying that it actually belongs to the essential image of the functor
$$\bMaps_{\DGCat^{\fSet}}(\Rep(\cG)^{\otimes \fSet},\qLisse(X^{\fSet})) \hookrightarrow \bMaps_{\DGCat^{\fSet}}(\Rep(\cG)^{\otimes \fSet},\Shv^*(X^{\fSet})).$$

\medskip

Now, the object of $\bMaps_{\DGCat^{\fSet}}(\Rep(\cG)^{\otimes \fSet},\Shv^!(X)^{\fSet})$ corresponding to $\Sht$ equals the image of the
resulting object of $\bMaps_{\DGCat^{\fSet}}(\Rep(\cG)^{\otimes \fSet},\qLisse(X^{\fSet}))$ under
$$\bMaps_{\DGCat^{\fSet}}(\Rep(\cG)^{\otimes \fSet},\qLisse(X^{\fSet})) \hookrightarrow \bMaps_{\DGCat^{\fSet}}(\Rep(\cG)^{\otimes \fSet},\Shv^!(X^{\fSet})).$$

Thus, we denote by the same symbol $\{\Sht_I\}$ the above objects of the categories
\begin{multline*}
\bMaps_{\DGCat^{\fSet}}(\Rep(\cG)^{\otimes \fSet},\Shv^*(X^{\fSet})) \hookleftarrow 
\bMaps_{\DGCat^{\fSet}}(\Rep(\cG)^{\otimes \fSet},\qLisse(X^{\fSet})) \hookrightarrow \\
\hookrightarrow \bMaps_{\DGCat^{\fSet}}(\Rep(\cG)^{\otimes \fSet},\Shv^!(X^{\fSet})).
\end{multline*}

\end{rem}

\begin{example}\label{e:sht triv}

We have a canonical isomorphism
$$\Sht(\one_{\Rep(\cG)_\Ran}) \simeq \Sht_\emptyset(\sfe)  = \\
\on{C}_c\left(\Bun_G,(\on{Graph}_{\Frob_{\Bun_G}})^*\circ (\Delta_{\Bun_G})_!(\ul\sfe_{\Bun_G})\right)
\simeq  \on{Funct}_c(\Bun_G(\BF_q))$$
where the last isomorphism is by base-change.

\end{example}

\ssec{Drinfeld's sheaf}\label{ss:drinf}

\sssec{}

By \corref{c:Cong cor}, there is a canonical functor
$\Sht_{\Loc}: \QCoh(\LocSys^{\on{restr}}_\cG(X)) \to \Vect$, 
i.e., $$\Sht_{\Loc} \in \QCoh(\LocSys^{\on{restr}}_\cG(X))^{\vee}.$$
As $\QCoh(\LocSys^{\on{restr}}_\cG(X))$ is canonically
self-dual (cf. \secref{sss:ls recap}(iii)),
to $\Sht_{\Loc}$ there corresponds an 
object\footnote{This object is named after V.~Drinfeld, since it was his idea, upon learning about
V.~Lafforgue's work, that shtuka cohomology should be encoded by a quasi-coherent sheaf
on the stack of Langlands parameters.}
$$\Drinf \in \QCoh(\LocSys^{\on{restr}}_\cG(X)).$$

\sssec{}

Let us observe some basic properties of $\Drinf$.

\medskip 

First, we have
$$\Gamma_!(\LocSys^{\on{restr}}_\cG(X),\Drinf) \simeq 
\Sht_{\Loc}(\CO_{\LocSys^{\on{restr}}_\cG(X)}).$$
by duality. 

\medskip 

As $$\CO_{\LocSys^{\on{restr}}} = \Loc(\one_{\Rep(\cG)_{\Ran}}),$$
we deduce from \exref{e:sht triv} that there is a canonical isomorphism
\begin{equation} \label{e:sect Drinf}
\Gamma_!(\LocSys^{\on{restr}}_\cG(X),\Drinf) = 
\Sht(\one_{\Rep(\cG)_{\Ran}}) = \on{Funct}_c(\Bun_G(\BF_q)).
\end{equation} 

\sssec{}

More generally, recall from \secref{sss:qcoh recap}
that for any $\CF \in \QCoh(\LocSys^{\on{restr}}_\cG(X))$ and a finite set $I$, 
one defines a functor 
$$\CF_I:\Rep(\cG)^{\otimes I} \to \qLisse(X^I).$$

\medskip 

By construction, for $\CF = \Drinf$, the functor $\Drinf_I$ coincides
with $\Sht_I$. In this manner, we see
that $\Drinf$ encodes cohomology of shtuka moduli spaces.

\ssec{Some remarks}

We now provide some additional remarks on the object $\Drinf$.

%
%
%
%
%
%

\begin{rem} \label{r:partial Frob}
Recall (see \cite[Sect. 24.1]{AGKRRV1}) that the stack $\LocSys^{\on{arithm}}_\cG(X)$ is defined  as
$$(\LocSys^{\on{restr}}_\cG(X))^{\Frob}.$$
Let $\iota$ denote the forgetful map
$$\LocSys^{\on{arithm}}_\cG(X)\to \LocSys^{\on{restr}}_\cG(X).$$

A priori, $\LocSys^{\on{arithm}}_\cG(X)$ is a \emph{formal stack} (i.e., a quotient of a formal affine 
scheme by an action of $\sG$), but \cite[Theorem 24.1.4]{AGKRRV1} says that $\LocSys^{\on{arithm}}_\cG(X)$ is
actually an algebraic stack. 

\medskip

One can show that the object $\Drinf$ can be \emph{a priori} obtained as
$\iota_*(\Drinf^{\on{arithm}})$
for a canonically defined object 
$$\Drinf^{\on{arithm}}\in \QCoh(\LocSys^{\on{arithm}}_\cG(X)).$$

This additional structure on $\Drinf$ encodes the equivariance of the objects
$$\Sht_I(V)\in \Shv(X^I), \quad V\in \Rep(\cG)^{\otimes I}$$
with respect to the \emph{partial Frobenius} maps acting on $X^I$. See also \secref{sss:partial Frob post}. 

\end{rem}

\begin{rem} \label{r:Vinc decomp}

The object $\Drinf^{\on{arithm}}$ allows to recover the spectral decomposition of the space 
of automorphic functions along classical Langlands parameters, established in \cite{VLaf} 
for the cuspidal subspace and extended in \cite{Xue1} to the entire space. 

\medskip

Namely, by \eqref{e:sect Drinf}, we have:
$$\Gamma(\LocSys^{\on{arithm}}_\cG(X),\Drinf^{\on{arithm}})\simeq \on{Funct}_c(\Bun_G(\BF_q)).$$

Set 
$$\CA:=\Gamma(\LocSys^{\on{arithm}}_\cG(X),\CO_{\LocSys^{\on{arithm}}_\cG(X)});$$
this is  is a commutative DG algebra over $\sfe$ that lives in non-positive cohomological
degrees. Set
$$\LocSys^{\on{arithm,coarse}}_\cG(X):=\Spec(\CA);$$
this is an affine (derived) scheme over $\sfe$.

\medskip

Let $\CA^0$ denote the $0$-th cohomology of $\CA$, so that $\Spec(\CA^0)$ is the classical affine scheme
$^{\on{cl}}\!\LocSys^{\on{arithm,coarse}}_\cG(X)$ underlying $\LocSys^{\on{arithm,coarse}}_\cG(X)$

\medskip

Now, one can show that the set of $\sfe$-points of $\Spec(\CA^0)$ is in bijection with isomorphism classes of semi-simple 
Frobenius-equivariant $\cG$-local systems on $X$, see \cite[Theorem 4.6.5]{AGKRRV1}.
I.e., we can view $^{\on{cl}}\!\LocSys^{\on{arithm,coarse}}_\cG(X)$ as the scheme of classical Langlands parameters.

\medskip

By construction, $\CA$ acts on the space of global sections of any object of $\QCoh(\LocSys^{\on{arithm}}_\cG(X))$.
In particular, we obtain an action of $\CA$ on $\on{Funct}_c(\Bun_G(\BF_q))$. However, since $\on{Funct}_c(\Bun_G(\BF_q))$
sits in cohomological degree $0$, this action factors through an action of $\CA^0$ on $\on{Funct}_c(\Bun_G(\BF_q))$.

\medskip

Thus, we can view $\on{Funct}_c(\Bun_G(\BF_q))$ as global sections of a canonically defined quasi-coherent sheaf
on $^{\on{cl}}\!\LocSys^{\on{arithm,coarse}}_\cG(X)$. This indeed may be viewed as a spectral decomposition of
$\on{Funct}_c(\Bun_G(\BF_q))$ over classical Langlands parameters.

\medskip

Furthermore, one can show that $\CA^0$ is a quotient of V.~Lafforgue's algebra of excursion operators. So the above
action of $\CA^0$ recovers the action of the excursion algebra on $\on{Funct}_c(\Bun_G(\BF_q))$, established in
\cite{Xue1}.

\end{rem}

\begin{rem} \label{r:partial Frob level}

The above construction of the object $\Drinf$ (resp., $\Drinf^{\on{arithm}}$) was specific to the 
\emph{everywhere unramified} situation. In a subsequent publication, we will show this construction
can be generalized to allow for level structure. 

\medskip

I.e., given a divisor $D\subset X$ defined over $\BF_q$,
one can construct objects
$$\Drinf_D\in \LocSys^{\on{restr}}_\cG(X-D) \text{ and }
\Drinf^{\on{arithm}}_D\in \LocSys^{\on{arithm}}_\cG(X-D)$$
that encode shtuka cohomology with level structure.

\medskip

What is for now a far-fetched goal is to interpret $\Drinf^{\on{arithm}}_D$ also as categorical
trace, see \secref{sss:partial Frob post level} for what we mean by that. 

\end{rem}

\section{Calculating the trace} \label{s:main}

In this section we will prove the main result of this paper,
\corref{c:trace}, which asserts that the space of (compactly supported) automorphic functions
can be obtained as the (categorical) trace of Frobenius on the category of automorphic sheaves
with nilpotent singular support. 

\ssec{Traces of Frobenius-Hecke operators}

\sssec{}

Recall that the $\Rep(\cG)_\Ran$-action on $\Shv(\Bun_G)$ preserves
its subcategory $\Shv_\Nilp(\Bun_G)$. 
Therefore, we obtain a functor
$$\Rep(\cG)_\Ran \to \bMaps_{\DGCat}(\Shv_\Nilp(\Bun_G),\Shv_\Nilp(\Bun_G))$$
sending $\CV$ to the functor $\CV \star -$. 

\medskip

Note also that the subcategory $\Shv_\Nilp(\Bun_G)\subset \Shv(\Bun_G)$ is preserved by
the endofunctor $(\Frob_{\Bun_G})_*$, see \cite[Sect. 22.3.3]{AGKRRV1}.

\medskip 

We define a functor $$\Sht^{\Tr}:\Rep(\cG)_\Ran \to \Vect$$
as the functor 
$$\CV \mapsto \Tr\big(\sH_\CV \ \circ 
(\Frob_{\Bun_G})_*,\Shv_{\Nilp}(\Bun_G),\big).$$
In other words, we compose $\sH_\CV$ with pushforward along
the geometric Frobenius $\Frob_{\Bun_G}$ endomorphism of $\Bun_G$
and form the trace (as an endofunctor of the (dualizable) DG category
$\Shv_{\Nilp}(\Bun_G)$).

\medskip 

Our main theorem asserts:

\begin{thm}\label{t:main}

There is a canonical isomorphism of functors $\Sht \simeq \Sht^{\Tr}:
\Rep(\cG)_\Ran \to \Vect$.

\end{thm}

We will prove this result in \secref{ss:main pf}.

\sssec{}

For the moment, we assume \thmref{t:main} and deduce further results from it.

\medskip 

First, observe that by definition, we have
$$\Sht^{\Tr}(\one_{\Rep(\cG)_\Ran}) = \Tr((\Frob_{\Bun_G})_*,\Shv_{\Nilp}(\Bun_G)).$$
On the other hand, by \exref{e:sht triv} we have
$$\Sht(\one_{\Rep(\cG)_\Ran}) = 
\on{Funct}_c(\Bun_G(\BF_q)).$$

\medskip 

Hence, we obtain: 

\begin{cor} \label{c:trace}
There exists a canonical isomorphism in $\Vect$
$$\Tr((\Frob_{\Bun_G})_*,\Shv_\Nilp(\Bun_G))\simeq \on{Funct}_c(\Bun_G(\BF_q)).$$
\end{cor}

%
%

\sssec{} 

By \eqref{e:three incarnations}, the functor $\Sht^{\Tr}$ corresponds to a system of functors $\{\Sht^{\Tr}_I\}$
\begin{equation} \label{e:Sht Tr I}
\Sht^{\Tr}_I:\Rep(\cG)^{\otimes I}\to \Shv(X^I).
\end{equation} 

From \thmref{t:main} we immediately obtain:

\begin{cor} \label{c:main}
For an individual finite set $I$, the functors $\Sht_I$ and $\Sht^{\Tr}_I$
$$\Rep(\cG)^{\otimes I}\to \qLisse(X)^{\otimes I}$$
are canonically isomorphic.
\end{cor}

Let us describe the functors $\Sht^{\Tr}_I$ explicitly.

\sssec{} 

We will use the following construction.

\medskip 

Let $\bC$ be a dualizable DG category. Let
$$\Vect\overset{\on{u}_\bC}\to \bC^\vee\otimes \bC \text{ and } \bC^\vee\otimes \bC\overset{\on{ev}_\bC}\to \Vect$$
be the unit and counit of the duality, respectively.

\medskip

Let $T:\bC\to \bC\otimes \bD$ be a functor, where $\bD$ is another DG category.
We can consider the \emph{relative trace} object
$$\Tr(T,\bC)\in \bD$$ defined as the composition
$$\Vect \overset{\on{u}_\bC}\to \bC^\vee\otimes \bC \overset{\on{Id}\otimes T}\to
\bC^\vee\otimes \bC \otimes \bD \overset{\on{ev}_\bC\otimes \on{Id}}\to \bD.$$

\sssec{}

Let $I$ be a finite set. Recall that we have a canonically defined functor
$$\sH:\Rep(\cG)^{\otimes I}\otimes \Shv_\Nilp(\Bun_G) \to \Shv_\Nilp(\Bun_G)\otimes \qLisse(X^I),$$
see \eqref{e:H!* qLisse}.

\medskip

In particular, for $V\in \Rep(\cG)^{\otimes I}$ we can consider the object
$$\Tr(\sH(V,-)\circ (\Frob_{\Bun_G})_*,\Shv_\Nilp(\Bun_G))\in \qLisse(X^I),$$
and this operation defines a functor
\begin{equation} \label{e:Tr Frob Hecke}
\Rep(\cG)^{\otimes I} \to \qLisse(X^I).
\end{equation}

Unwinding the definitions, it is easy to see that the functor $\Sht^{\Tr}_I$ of \eqref{e:Sht Tr I} identifies with the composition
of \eqref{e:Tr Frob Hecke} and the embedding $\qLisse(X^I)\to \Shv(X^I)$ of \eqref{e:! emb}.

\sssec{}

Hence, from \corref{c:main} we obtain:

\begin{cor} \label{c:shtuka}
For an individual finite set $I$, the functor 
$$\Sht_I:\Rep(\cG)^{\otimes I}\to \qLisse(X^I)$$ 
of \eqref{e:classical shtuka} identifies canonically with the functor
$$V \mapsto \Tr(\sH(V,-)\circ (\Frob_{\Bun_G})_*,\Shv_\Nilp(\Bun_G)).$$
\end{cor}

\corref{c:shtuka} is the \emph{Shtuka Conjecture} from
\cite{AGKRRV1} (Conjecture 22.5.7 in {\it loc.cit.}). 

\sssec{}

Finally, we note that \thmref{t:main} may be tautologically be reformulated in terms
of Drinfeld sheaves. 

\medskip 

First, as was noted above, the functors $\Sht_I^{\Tr}$ take
values in $\qLisse(X^I)$. Therefore, by \corref{c:Loc dual bis}, 
$\Sht^{\Tr}$ factors uniquely as $\Sht^{\Tr}_{\Loc} \circ \Loc$
for a uniquely defined functor 
$$\Sht^{\Tr}_{\Loc}:\QCoh(\LocSys^{\on{restr}}_\cG(X)) \to \Vect.$$

\medskip 

As in \secref{ss:drinf}, by self-duality of $\QCoh(\LocSys^{\on{restr}}_\cG(X))$, to $\Sht^{\Tr}_{\Loc}$
there corresponds an object
$\Drinf^{\Tr} \in \QCoh(\LocSys^{\on{restr}}_\cG(X))$.

\medskip

From \thmref{t:main}, we obtain:

\begin{cor} \label{c:drinf}
The objects $\Drinf$ and $\Drinf^{\Tr}$ of $\QCoh(\LocSys_\cG^{\on{restr}})$
are canonically isomorphic.
\end{cor}

\ssec{Proof of Theorem \ref{t:main}} \label{ss:main pf}

\sssec{}

Recall the object $\sR\in \Rep(\cG)_\Ran$ of \secref{ss:proj}. The proof of \thmref{t:main} will
be based on the following result.

\begin{thm}\label{t:trace r}
There is a canonical isomorphism $$\Sht^{\Tr}(-) \simeq \Sht(\sR \star -)$$
as functors $\Rep(\cG)_\Ran \to \Vect$.
\end{thm}

\begin{rem}
As the proof of \thmref{t:trace r} will show, we will rather construct an isomorphism
$$\Sht^{\Tr}(-) \simeq \Sht(-\star \sR),$$
and we will swap $\Sht(-\star \sR)$ for $\Sht(\sR \star -)$ using the fact that the category
$\Rep(\cG)_{\Ran}$ is \emph{symmetric monoidal}. 

\medskip

Our preference of one over the other is purely notational.

\end{rem} 

\sssec{}

We postpone the proof of \thmref{t:trace r} to \secref{ss:trace r pf}. It is
essentially a calculation using the results of \cite{AGKRRV2}, which 
we recall in \secref{ss:nilp self duality}.

\medskip 

For the present, we assume \thmref{t:trace r} and deduce 
\thmref{t:main} from it.

\sssec{} 

The argument is straightforward at this point using 
\thmref{t:Cong cor}. Recall that \corref{c:Cong cor}
provides a factorization $\Sht = \Sht_{\Loc} \circ \Loc$.

\medskip 

Now recall that $\Loc$ is monoidal and sends
$\sR$ to the structure sheaf by \corref{c:diagonal LocSys}.
Therefore, $$\Loc(\sR\star -) \simeq \Loc(-)$$ as functors
$\Rep(\cG)_\Ran \to \QCoh(\LocSys^{\on{restr}}_\cG(X))$. 

\medskip 

By \thmref{t:trace r}, we obtain
$$\Sht^{\Tr} \simeq \Sht(\sR \star -) \simeq \Sht_{\Loc}\circ \Loc(\sR \star -)
\simeq \Sht_{\Loc}\circ \Loc(-) \simeq \Sht$$
as desired. 

\ssec{Self-duality for $\Shv_\Nilp(\Bun_G)$}\label{ss:nilp self duality}

To prove \thmref{t:trace r}, we use the 
explicit description 
of the dual of $\Shv_\Nilp(\Bun_G)$ from \cite{AGKRRV2}. 
We review these results below.

\sssec{}\label{sss:ps-u-nilp}

%

Consider the object
$$\CK_\sR\in \Shv(\Bun_G \times \Bun_G),$$
where $\sR \in \Rep(\cG)_\Ran$ was defined in \secref{ss:proj}, and where the notation $\CK_{-}$
is an in \secref{sss:kernel notation}. 

\medskip

A priori, it is defined as an object of the category 
$\Shv(\Bun_G\times \Bun_G)$. However, we have the following
key result, see \cite[Sect. 3.1.2 and Proposition 3.1.4]{AGKRRV2}:

\begin{thm} \label{t:u Nilp Bun}
The object $\CK_\sR$ lies in the essential image of
the fully faithful functor
$$\Shv_{\Nilp}(\Bun_G)\otimes \Shv_{\Nilp}(\Bun_G) \to \Shv(\Bun_G\times \Bun_G).$$
\end{thm} 

\begin{rem}
In \cite[Sect. 3.1.2]{AGKRRV2}, the object $\CK_\sR$ was denoted $\on{ps-u}_{\Bun_G,\Nilp}$.
\end{rem}

\sssec{} \label{sss:ev*}

For a stack $\CY$ let 
$$\on{ev}^l_\CY:\Shv(\CY)\otimes \Shv(\CY)\to \Vect$$
be the functor given by
$$\CF_1,\CF_2 \mapsto \on{C}^\cdot_c(\CY,\CF_1\overset{*}\otimes \CF_2).$$

\noindent{\it Warning}: in general, the pairing $\on{ev}^l_\CY$ is \emph{not} perfect, i.e., it does not
define a self-duality on $\Shv(\CY)$.

\sssec{}

Take $\CY=\Bun_G$, and we restrict the pairing $\on{ev}^l_{\Bun_G}$ to
$$\Shv_{\Nilp}(\Bun_G)\otimes \Shv_{\Nilp}(\Bun_G)\subset \Shv(\Bun_G)\otimes \Shv(\Bun_G).$$

\medskip

We now quote the following result (see \cite[Theorem 3.2.2]{AGKRRV2}):

\begin{thm} \label{t:duality}
The object 
$$\CK_\sR\in \Shv_\Nilp(\Bun_G)\otimes \Shv_\Nilp(\Bun_G)$$
together with the pairing
$$\Shv_\Nilp(\Bun_G)\otimes \Shv_\Nilp(\Bun_G) \to 
\Shv(\Bun_G)\otimes \Shv(\Bun_G) \overset{\on{ev}^l_{\Bun_G}}\longrightarrow \Vect$$
define an identification
$$\Shv_\Nilp(\Bun_G)^\vee \simeq \Shv_\Nilp(\Bun_G).$$
\end{thm}

\begin{rem}

The trace calculation involved in the proof of \thmref{t:trace r} uses the explicit description of
the dual of $\Shv_\Nilp(\Bun_G)$, provided by \thmref{t:duality}. 
This may be the most conceptually non-trivial place in the paper:

\medskip

There is a more obvious candidate for the dual of $\Shv_\Nilp(\Bun_G)$, namely the category
denoted $\Shv_\Nilp(\Bun_G)_{\on{co}}$, see \cite[Corollary 2.6.5]{AGKRRV2}. The two descriptions 
of the dual are related by a non-trivial operation, namely the \emph{miraculous functor} $\Mir_{\Bun_G}$,
which defines an equivalence $$\Shv_\Nilp(\Bun_G)_{\on{co}}\to \Shv_\Nilp(\Bun_G),$$ see 
\cite[Corollary 2.9.5]{AGKRRV2}.  

\medskip

However, if we used the description of the dual of $\Shv_\Nilp(\Bun_G)$ as $\Shv_\Nilp(\Bun_G)_{\on{co}}$,
we would not be able to directly relate the functor $\on{Sht}^{\on{Tr}}$ to shtukas, as defined in \cite{VLaf}.
Rather, we would encounter \emph{co-shtukas}. The latter is a functor $\Rep(\cG)_\Ran\to \Vect$ given by
$$\Rep(\cG)_\Ran \overset{\CV \mapsto \on{co-}\CK_\CV} \longrightarrow
\Shv(\Bun_G \times \Bun_G)_{\on{co}_1} \overset{(\on{Graph}_{\Frob_{\Bun_G}})^!}\longrightarrow
\Shv(\Bun_G)_{\on{co}} \overset{\on{C}^\cdot_\blacktriangle(\Bun_G,-)}{\longrightarrow} \Vect,$$
where:

\begin{itemize}

\item $\Shv(\Bun_G)_{\on{co}}$ is the category defined in \cite[Sect. 2.5]{AGKRRV2};

\item $\Shv(\Bun_G \times \Bun_G)_{\on{co}_1}$ is as in \cite[Sect. C.4.5]{AGKRRV2};

\item $\on{C}^\cdot_\blacktriangle(\Bun_G,-):\Shv(\Bun_G)_{\on{co}}\to \Vect$ is the functor
of renormalized chains, see \cite[Sect. C.3.4]{AGKRRV2};

\item $\on{co-}\CK_\CV:=(\on{Id}_{\Bun_G}\boxtimes \sH_\CV)(\on{u}_{\Bun_G,\on{co}_1})$,
where $$\on{u}_{\Bun_G,\on{co}_1}\in \Shv(\Bun_G \times \Bun_G)_{\on{co}_1}$$ is as in \cite[Sect. C.4.6]{AGKRRV2}.

\end{itemize}

The problem with this approach is that co-shtukas seem to be a rather unwieldy objects, compared to shtukas.

\end{rem}

\ssec{Proof of \thmref{t:trace r}}\label{ss:trace r pf}

\sssec{}

Let us be given a dualizable DG category $\bC$ and functors $T,S:\bC \to \bC$.

\medskip

Note that we can compute $\Tr(S\circ T,\bC) \in \Vect$
as the composition
$$\Vect \overset{\on{u}_{\bC}}{\longrightarrow} \bC^{\vee} \otimes \bC 
\overset{T^{\vee} \otimes S}{\longrightarrow} \bC^{\vee} \otimes \bC 
\overset{\on{ev}_\bC \otimes \Id}{\longrightarrow} \Vect.$$

\sssec{} \label{sss:compute tr V}

For $\CV \in \Rep(\cG)_\Ran$, we deduce from \thmref{t:duality} and the above observation that
we can compute $\Sht^{\Tr}(\CV)$
as the composition
\begin{multline*}
\Vect \overset{\CK_\CR}{\longrightarrow}
\Shv_\Nilp(\Bun_G) \otimes \Shv_\Nilp(\Bun_G)  
\overset{((\Frob_{\Bun_G})_*)^{\vee} \otimes \sH_\CV} {\longrightarrow} \\
\to \Shv_\Nilp(\Bun_G) \otimes \Shv_\Nilp(\Bun_G)  
\overset{-\overset{*}{\otimes}-}{\longrightarrow}
\Shv(\Bun_G) \overset{\on{C}_c(\Bun_G,-)}{\longrightarrow} \Vect.
\end{multline*}
Here $((\Frob_{\Bun_G})_*)^{\vee}$ is the dual functor to
$(\Frob_{\Bun_G})_*$ with respect to the duality of
\thmref{t:duality}.

\sssec{}

To proceed further, we need to compute $((\Frob_{\Bun_G})_*)^{\vee}$.

\medskip

Note that the functor $(\Frob_{\Bun_G})^*$ also preserves the subcategory $\Shv_{\Nilp}(\Bun_G)\subset \Shv(\Bun_G)$
(this is established in the course of the proof of \cite[Lemma 22.3.2]{AGKRRV1}). 

\begin{lem}\label{l:frob}
In the above notation, the functor 
$$((\Frob_{\Bun_G})_*)^{\vee}:\Shv_{\Nilp}(\Bun_G) \to 
\Shv_{\Nilp}(\Bun_G)$$ 
is canonically isomorphic to $(\Frob_{\Bun_G})^*$.
\end{lem}

\begin{proof}

For an algebraic stack $\CY$, the pullback functor $(\Frob_{\CY})^*$ is a self-equivalence of 
$\Shv(\CY)$. Hence, its right adjoint, which is $(\Frob_{\CY})_*$,
is the inverse of $(\Frob_{\CY})^*$.

\medskip

Take $\CY=\Bun_G$. Since the functors $(\Frob_{\Bun_G})_*$ and $(\Frob_{\Bun_G})^*$
both preserve $\Shv_{\Nilp}(\Bun_G)\subset \Shv(\Bun_G)$, we obtain that 
$(\Frob_{\Bun_G})^*|_{\Shv_{\Nilp}(\Bun_G)}$ and $(\Frob_{\Bun_G})_*|_{\Shv_{\Nilp}(\Bun_G)}$
are mutually inverse equivalences.

\medskip

The assertion of the lemma follows now formally from the fact that counit of the self-duality on $\Shv_{\Nilp}(\Bun_G)$,
i.e., the functor $\ev^l_{\Bun_G}$, is mapped to itself by $(\Frob_{\Bun_G})^*\otimes (\Frob_{\Bun_G})^*$. To check this,
we have to show that 
$$\on{C}^\cdot_c(\Bun_G,(\Frob_{\Bun_G})^*(-))\simeq \on{C}^\cdot_c(\Bun_G,-).$$
The latter is true for any algebraic stack
\begin{multline*}
\on{C}^\cdot_c(\CY,(\Frob_{\CY})^*(-))\simeq
\on{C}^\cdot_c(\CY,(\Frob_{\CY})_! \circ (\Frob_{\CY})^*(-)) \overset{\Frob_{\CY}\text{ is proper}}\simeq  \\
\simeq \on{C}^\cdot_c(\CY,(\Frob_{\CY})_* \circ (\Frob_{\CY})^*(-))\simeq  \on{C}^\cdot_c(\CY,-).
\end{multline*}

\end{proof}

\begin{rem}
The isomorphism\footnote{Note that for an algebraic stack, the Frobenius morphsm is proper and radicial, but not necessarily schematic.}
\begin{equation}  \label{e:Frob !*}
(\Frob_{\CY})_*\simeq (\Frob_{\CY})_!
\end{equation}
implies that $(\Frob_{\CY})_!$ is also an equivalence. Hence, its right adjoint, i.e., $(\Frob_{\CY})^!$ is its inverse. From here we obtain
$$(\Frob_{\CY})^! \simeq ((\Frob_{\CY})_!)^{-1}\simeq ((\Frob_{\CY})_*)^{-1}\simeq (\Frob_{\CY})^*.$$
\end{rem}

\sssec{}

Combining \secref{sss:compute tr V} and \lemref{l:frob}, we obtain that for 
$\CV \in \Rep(\cG)_\Ran$, we can compute
$\Sht^{\Tr}(\CV)$ as the composition:
\begin{multline*}
\Vect \overset{\CK_\sR}{\longrightarrow}
\Shv_\Nilp(\Bun_G) \otimes \Shv_\Nilp(\Bun_G)  
\overset{(\Frob_{\Bun_G})^* \otimes \sH_\CV} {\longrightarrow} \\
\to \Shv_\Nilp(\Bun_G) \otimes \Shv_\Nilp(\Bun_G)  
\overset{-\overset{*}{\otimes}-}{\longrightarrow}
\Shv(\Bun_G) \overset{\on{C}_c(\Bun_G,-)}{\longrightarrow} \Vect,
\end{multline*}
which is the same as 
\begin{multline*}
\Vect \overset{\CK_\sR}{\longrightarrow}
\Shv_\Nilp(\Bun_G) \otimes \Shv_\Nilp(\Bun_G)  \overset{\boxtimes}\to \Shv(\Bun_G\times \Bun_G)
\overset{\on{Id}_{\Bun_G}\boxtimes \sH_\CV}\longrightarrow \\ 
\to \Shv(\Bun_G\times \Bun_G) 
\overset{(\Frob_{\Bun_G} \times \on{id})^*}\longrightarrow 
\Shv(\Bun_G\times \Bun_G)  \overset{\Delta_{\Bun_G}^*}\longrightarrow \\
\to \Shv(\Bun_G) \overset{\on{C}_c(\Bun_G,-)}{\longrightarrow} \Vect.
\end{multline*}

\sssec{}

Note that 
$$(\on{Id}_{\Bun_G}\boxtimes \sH_\CV)(\CK_\sR) \simeq 
\CK_{\CV \star \sR} \simeq \CK_{\sR \star \CV}.$$

\medskip

Therefore, we can compute 
$\Sht^{\Tr}(\CV)$ as the composition
$$\Vect \overset{\CK_{\sR \star \CV}}{\longrightarrow}
\Shv(\Bun_G \times \Bun_G)  
\overset{(\Frob_{\Bun_G} \times \on{id})^*} {\longrightarrow} \Shv(\Bun_G \times \Bun_G)  
\overset{\Delta_{\Bun_G}^*}{\longrightarrow}
\Shv(\Bun_G) \overset{\on{C}_c(\Bun_G,-)}{\longrightarrow} \Vect.$$
This coincides with 
$$\Vect \overset{\CK_{\sR \star \CV}}{\longrightarrow}
\Shv_\Nilp(\Bun_G) \otimes \Shv_\Nilp(\Bun_G)  
\overset{(\on{Graph}_{\Frob_{\Bun_G}})^*}{\longrightarrow}
\Shv(\Bun_G) \overset{\on{C}_c(\Bun_G,-)}{\longrightarrow} \Vect,$$
which is the same as $\Sht(\sR \star \CV)$, by definition.

\qed[\thmref{t:trace r}]

\ssec{Interpretation as enhanced trace}

The contents of this section are an extended remark 
and are not necessary for the rest of the paper.

\sssec{}

Recall (see \cite[Theorem 14.3.2]{AGKRRV1}) that the category $\Shv_{\Nilp}(\Bun_G)$ carries a monoidal
action of $\QCoh(\LocSys^{\on{restr}}_\cG(X))$.

\medskip

Moreover, the action of $(\Frob_{\Bun_G})_*$ is compatible with the action of the monoidal automorphism
of $\Frob^*$ of $\QCoh(\LocSys^{\on{restr}}_\cG(X))$, where $\Frob$ is the automorphism of $\LocSys^{\on{restr}}_\cG(X)$,
induced by Frobenius endomorphism $\Frob_X$ of $X$.

\medskip

In this case, following \cite[Sect. 3.8.2]{GKRV}, to the pair $(\Shv_{\Nilp}(\Bun_G),(\Frob_{\Bun_G})_*)$ we can attach its 
\emph{enhanced trace}, 
$$\Tr^{\on{enh}}_{\QCoh(\LocSys^{\on{restr}}_\cG(X))}((\Frob_{\Bun_G})_*,\Shv_{\Nilp}(\Bun_G))\in
\on{HH}_\bullet(\QCoh(\LocSys^{\on{restr}}_\cG(X)),\Frob^*),$$
where 
$\on{HH}_\bullet(\QCoh(\LocSys^{\on{restr}}_\cG(X)),\Frob^*)$ is the (symmetric monoidal) DG category of Hochschild chains on 
the (symmetric) monoidal category $\QCoh(\LocSys^{\on{restr}}_\cG(X))$ with respect to the (symmetric) monoidal endofunctor
$\Frob^*$.

\medskip

Furthermore, we have
$$\on{HH}_\bullet(\QCoh(\LocSys^{\on{restr}}_\cG(X)),\Frob^*)\simeq \QCoh((\LocSys^{\on{restr}}_\cG(X))^\Frob),$$
see \cite[Sect. 7.10.4]{AGKRRV1}.

\sssec{}

Recall that we denote
$$\LocSys^{\on{arithm}}_\cG(X):=(\LocSys^{\on{restr}}_\cG(X))^\Frob,$$
and by $\iota$ the forgetful map
$$\LocSys^{\on{arithm}}_\cG(X)\to \LocSys^{\on{restr}}_\cG(X),$$
see \cite[Sect. 24.1]{AGKRRV1}.

\medskip

Thus, we can interpret the above object $\Tr^{\on{enh}}_{\QCoh(\LocSys^{\on{restr}}_\cG(X))}((\Frob_{\Bun_G})_*,\Shv_{\Nilp}(\Bun_G))$
as an object, denoted
$$\Drinf^{\Tr,\on{arithm}}\in \QCoh(\LocSys^{\on{arithm}}_\cG(X)).$$

\medskip

We have
\begin{equation} \label{e:Dr Tr}
\Drinf^{\Tr}\simeq \iota_*(\Drinf^{\Tr,\on{arithm}});
\end{equation}
this assertion is essentially \cite[Theorem 4.4.4]{GKRV}, with a generalization supplied by \cite[Theorem 7.10.6]{AGKRRV1}. 

\sssec{}

Recall the objects 
$$\on{Ev}(V)\in \QCoh(\LocSys^{\on{restr}}_\cG(X))\otimes \qLisse(X)^{\otimes I}.$$

By construction, the corresponding objects
$$(\iota^*\otimes \on{Id})(\on{Ev}(V))\in \QCoh(\LocSys^{\on{arithm}}_\cG(X))\otimes \qLisse(X)^{\otimes I}$$
carry a structure of equivariance with respect to the partial Frobenius endomorphisms acting along $X^I$.

\medskip

By \eqref{e:Dr Tr} and the projection formula, we have
\begin{equation} \label{e:Sht Tr via arithm}
\Sht^{\Tr}_I(V) \simeq 
(\Gamma_!(\LocSys^{\on{arithm}}_\cG(X),-)\otimes \on{Id})\left((\iota^*\otimes \on{Id})(\on{Ev}(V))\otimes \Drinf^{\Tr,\on{arithm}})\right).
\end{equation}

This interpretation shows that the the objects 
$$\Sht^{\Tr}_I(V) \in \qLisse(X)^{\otimes I}\subset \Shv(X^I)$$
carry a natural structure of equivariance with respect to the partial Frobenius endomorphisms on $X^I$.

\begin{rem}
Recall that $\LocSys^{\on{arithm}}_\cG(X)$ is actually a quasi-compact algebraic stack, see Remark \ref{r:partial Frob}. This implies that
in formula \eqref{e:Sht Tr via arithm}, we have
$$\Gamma_!(\LocSys^{\on{arithm}}_\cG(X),-) \simeq \Gamma(\LocSys^{\on{arithm}}_\cG(X),-).$$
\end{rem}

\sssec{} \label{sss:two part Frob}

Having proved \thmref{t:main}, and hence \corref{c:main}, we obtain that the objects
$$\Sht_I(V) \in \Shv(X^I)$$
also carry a natural structure of equivariance with respect to the partial Frobenius endomorphisms on $X^I$.

\medskip

However, it is not difficult to show that this structure identifies with the structure of partial Frobenius equivariance 
on shtukas, defined in  \cite{VLaf} and \cite{Xue1}. 
The matching between the two is essentially \cite[Lemma 4.5.4]{GKRV}.

\sssec{} 

Now, given an object $\CF\in \QCoh(\LocSys^{\on{restr}}_\cG(X))$, the datum required to exhibit it as
$$\iota_*(\CF^{\on{arithm}})$$
for some $\CF^{\on{arithm}}\in \LocSys^{\on{arithm}}_\cG(X)$ is equivalent to a \emph{compatible} collection of
data of partial Frobenius equivariance on the associated functors
$$\CF_I:\Rep(\cG)^{\otimes I}\to \qLisse(X)^{\otimes I}, \quad 
V\mapsto (\Gamma_!(\LocSys^{\on{restr}}_\cG(X),-)\otimes \on{Id})(\on{Ev}(V)\otimes \CF).$$

\sssec{}  \label{sss:partial Frob post}

As was mentioned in \secref{sss:two part Frob}, the structure of partial Frobenius equivariance on the functors $\Sht_I$
is a priori defined (in \cite{VLaf} and \cite{Xue1}), thereby producing an object 
$$\Drinf^{\on{arithm}}\in \QCoh(\LocSys^{\on{arithm}}_\cG(X)).$$

Thus, the assertion in \secref{sss:two part Frob} can interpreted as an isomorphism
$$\Drinf^{\on{arithm}}\simeq \Drinf^{\Tr,\on{arithm}},$$
which induces the isomorphism
$$\Drinf\simeq \Drinf^{\Tr}$$
of \thmref{t:main} by applying $\iota_*$.

\sssec{} \label{sss:partial Frob post level}

Furthermore, as was mentioned in Remark \ref{r:partial Frob level}, the corresponding structure of partial Frobenius equivariance can
be constructed also on the functors
$$\Sht_{I,D}:\Rep(\cG)^{\otimes I}\to \qLisse(X-D)^{\otimes I}$$
that encode the cohomology of shtukas with level structure, thereby producing an object
$$\Drinf_D^{\on{arithm}}\in \QCoh(\LocSys^{\on{restr}}_\cG(X-D)).$$

What we do not have at the moment is the interpretation of this object $\Drinf_D^{\on{arithm}}$
as enhanced categorical trace (nor of the functors $\Sht_{I,D}$ as just categorical traces). 

\section{Local terms}\label{s:LT}

\ssec{A hypothesis} \label{ss:constraccess} 

In order to simplify the exposition, for the duration of this section, we will assume \cite[Conjecture 14.1.8]{AGKRRV1}. 
This conjecture says that the category $\Shv_\Nilp(\Bun_G)$ is generated by objects that are compact as objects
of the ambient category $\Shv(\Bun_G)$.

\medskip

Since we know that $\Shv_\Nilp(\Bun_G)$ is compactly generated as a DG category (see \cite[Theorem 16.1.1]{AGKRRV1}),
this conjecture is equivalent to the statement that the forgetful functor 
$$\Shv_\Nilp(\Bun_G)\hookrightarrow \Shv(\Bun_G)$$
preserves compactness, or, equivalently, that it admits a continuous right adjoint.

%

\ssec{Formulation of the problem}

\sssec{}

In this section we will construct four maps
$$\on{LT}^{\on{naive}}, \on{LT}^{\on{true}}, \on{LT}^{\on{Serre}}, \on{LT}^{\Sht}:
\Tr((\Frob_{\Bun_G})_*,\Shv_\Nilp(\Bun_G))\to \on{Funct}_c(\Bun_G(\BF_q))$$
that we refer to as \emph{local term} morphisms.

\medskip 

At this point, we have only encountered the last of these four morphisms: 
by definition, $\on{LT}^{\Sht}$ is the isomorphism of \corref{c:trace}.

\sssec{}

Assuming the construction of the other three morphisms, 
we can state the main result of this section and the next:

\begin{thm}\label{t:lt}

The four morphisms $\on{LT}^{\on{naive}}$, 
$\on{LT}^{\on{true}}$, $\on{LT}^{\on{Serre}}$, and $\on{LT}^{\Sht}$ are equal.

\end{thm}

\begin{rem}

By \corref{c:trace}, the source and target of
each local term map is in $\Vect^{\heartsuit}$, so 
equality (as opposed to homotopy) is the relevant notion.

\end{rem}

\sssec{}

As $\on{LT}^{\Sht}$ is an isomorphism, from \thmref{t:lt} we deduce:

\begin{cor}

Each of the four local term morphisms in an isomorphism.

\end{cor} 

In the case of $\on{LT}^{\on{naive}}$ (see below), 
this corollary recovers the Trace Conjecture as
formulated in \cite[Conjecture 22.3.7]{AGKRRV1}.

\medskip

We now proceed to the construction of the local term morphisms. 

\ssec{Naive and true local terms}

In what follows, we fix an algebraic stack $\CY$, which plays the role of $\Bun_G$. 
We will add certain additional assumptions to $\CY$ as we proceed. 

\medskip 

The material in this subsection closely follows \cite[Sect. 22.1-22.2]{AGKRRV1}, to which we
refer the reader for more detail.

\sssec{Naive local term}

Any algebraic stack $\CY$ has the property that $\Shv(\CY)$ is compactly generated
and every compact object is $!$-extended from some quasi-compact open 
in $\CY$ (see \cite[Sect. F.1.1]{AGKRRV1}). 

\medskip 

Suppose $y \in \CY(\BF_q)$ is given. 
In this case, the 
functor $y^*:\Shv(\CY) \to \Vect$ preserves compact objects and intertwines 
$(\Frob_{\CY})_*$ with the identity functor (by \eqref{e:Frob !*}).
Therefore, functoriality of traces yields a map
$$\on{Tr}((\Frob_{\CY})_*,\Shv(\CY)) \to 
\on{Tr}(\on{Id},\Vect) = \sfe.$$

\medskip 

In the quasi-compact case, this yields a map
$$\on{LT}_\CY^{\on{naive}}:\on{Tr}((\Frob_{\CY})_*,\Shv(\CY)) \to 
\on{Funct}(\CY(\BF_q)).$$

\begin{rem}

By definition, this construction satisfies the compatibility referenced
in Remark \ref{r:lt function}.

\end{rem}

\sssec{True local term} \label{sss:true}

Suppose first that $\CY$ is quasi-compact. Suppose in addition that 
$\CY$ is locally of the form $Z/H$, where $Z$ is a scheme of finite type and
$H$ is an affine algebraic group.

\medskip 

As in \cite[Sect. A.4]{AGKRRV2}, these assumptions imply that
$\Shv(\CY)$ is canonically self-dual (via Verdier duality), with pairing, denoted $\ev_\CY$:
$$\Shv(\CY) \otimes \Shv(\CY) \subset \Shv(\CY \times \CY)
\overset{\Delta_\CY^!}{\to} \Shv(\CY) \overset{\on{C}_{\blacktriangle}^{\cdot}(\CY,-)} \to 
\Vect.$$

\medskip 

The unit for this duality, denoted 
$$\on{u}_{\Shv(\CY)} \in \Shv(\CY) \otimes \Shv(\CY)$$
is obtained by applying the right adjoint to the embedding 
\begin{equation} \label{e:Kunneth Y}
\Shv(\CY) \otimes \Shv(\CY) \overset{\boxtimes}\hookrightarrow \Shv(\CY \times \CY)
\end{equation}
to $(\Delta_\CY)_*(\omega_{\CY})$.

\medskip

In what follows we will not distinguish notationally between $\on{u}_{\Shv(\CY)}$
and its image under the fully faithful functor \eqref{e:Kunneth Y}. Thus, by adjunction
we obtain a map
$$\on{u}_{\Shv(\CY)} \to (\Delta_\CY)_*(\omega_{\CY}).$$

\medskip 

From here, we obtain a map
\begin{multline*}
\Tr((\Frob_{\CY})_*,\Shv(\CY)) \simeq 
\on{C}_{\blacktriangle}^{\cdot}\big(\CY,
\Delta_\CY^!\circ ((\Frob_{\CY})_* \otimes \on{Id})(\on{u}_{\Shv(\CY)})\big) \simeq 
\on{C}_{\blacktriangle}^{\cdot}\big(\CY,
\Delta_\CY^!\circ (\Frob_\CY\times\on{Id})_*(\on{u}_{\Shv(\CY)})\big) 
\to \\
\to \on{C}_{\blacktriangle}^{\cdot}\big(\CY,
\Delta_\CY^!\circ (\Frob_\CY\times\on{Id})_*\circ (\Delta_\CY)_*(\omega_{\CY}))\big) \simeq 
\on{C}_{\blacktriangle}^{\cdot}(\CY^{\Frob},
\omega_{\CY^{\Frob}}) \simeq \on{Funct}(\CY(\BF_q))
\end{multline*}
whose composition we denote by 
$$\on{LT}_{\CY}^{\on{true}}:\Tr((\Frob_{\CY})_*,\Shv(\CY)) \to 
\on{Funct}(\CY(\BF_q)).$$

\sssec{}

We have (see \cite[Theorem 0.4]{GV}):

\begin{thm} \label{t:gv} \hfill 

\smallskip

\noindent{\em(a)}
The maps $\on{LT}_{\CY}^{\on{true}}$ and $\on{LT}_{\CY}^{\on{naive}}$ are canonically
homotopic.

\smallskip

\noindent{\em(b)} For a schematic map $f:\CY_1\to \CY_2$, the diagram
$$
\CD
\Tr((\Frob_{\CY_1})_*,\Shv(\CY_1)) @>{\on{LT}_{\CY_1}^{\on{true}}}>> \on{Funct}(\CY_1(\BF_q)) \\
@VVV @VVV \\
\Tr((\Frob_{\CY_2})_*,\Shv(\CY_2)) @>{\on{LT}_{\CY_2}^{\on{true}}}>> \on{Funct}(\CY_2(\BF_q)) 
\endCD
$$
is commutative, where the left vertical arrow is induced by the functor $f_!:\Shv(\CY_1)\to \Shv(\CY_2)$,
and the right vertical arrow is given by pushforward.  

\smallskip

\noindent{\em(c)} The commutative diagram in point (b) is compatible with the identification of point (a) 
with the commutative diagram
$$
\CD
\Tr((\Frob_{\CY_1})_*,\Shv(\CY_1)) @>{\on{LT}_{\CY_1}^{\on{naive}}}>> \on{Funct}(\CY_1(\BF_q)) \\
@VVV @VVV \\
\Tr((\Frob_{\CY_2})_*,\Shv(\CY_2)) @>{\on{LT}_{\CY_2}^{\on{naive}}}>> \on{Funct}(\CY_2(\BF_q)).
\endCD
$$
 \end{thm}

\sssec{} \label{sss:non-qc}

Let now $\CY$ be not necessarily quasi-compact. We will consider the poset of quasi-compact open substacks 
$$\CU \overset{j}\hookrightarrow \CY,$$
and the corresponding functors $j_!:\Shv(\CU)\to \Shv(\CY)$. 

\medskip

The category $\Shv(\CY)$ is compactly generated by the essential images of $j_!|_{\Shv(\CU)^c}$. Furthermore, 
the induced map
$$\underset{\CU}{\on{colim}}\, \Tr((\Frob_\CU)_*,\Shv(\CU))\to \Tr((\Frob_\CY)_*,\Shv(\CY))$$
is an isomorphism (see \cite[Sect. 22.1.11]{AGKRRV1}).

\medskip

Using the commutative diagrams in \thmref{t:gv}(b),(c), this allows to define the maps
$$\on{LT}_\CY^{\on{naive}} \text{ and } \on{LT}_\CY^{\on{true}}$$ 
from $\Tr((\Frob_\CY)_*,\Shv(\CY))$ to
$$\underset{\CU}{\on{colim}}\, \on{Funct}(\CU(\BF_q))\simeq \on{Funct}_c(\CY(\BF_q)).$$

Moreover, by \thmref{t:gv}(a), these two maps are canonically homotopic. 

\sssec{The case of $\Bun_G$}

We now specialize to the case of $\Bun_G$. We consider the full subcategory 
\begin{equation} \label{e:embed Nilp again}
\Shv_\Nilp(\Bun_G) \hookrightarrow \Shv(\Bun_G).
\end{equation} 

As was mentioned earlier, it is preserved by the endofunctor $(\Frob_{\Bun_G})_*$.
By the hypothesis in \secref{ss:constraccess}, the embedding \eqref{e:embed Nilp again}
admits a continuous right adjoint. Hence, the functoriality of the categorical trace construction
yields a map
$$\Tr((\Frob_{\Bun_G})_*,\Shv_\Nilp(\Bun_G))\to 
\Tr((\Frob_{\Bun_G})_*,\Shv(\Bun_G)).$$

Composing with the maps $\on{LT}^{\on{naive}}_{\Bun_G}$ and $\on{LT}^{\on{true}}_{\Bun_G}$, we obtain two maps
$$\Tr((\Frob_{\Bun_G})_*,\Shv_\Nilp(\Bun_G))\rightrightarrows \on{Funct}_c(\Bun_G(\BF_q)),$$
that we will denote $\on{LT}^{\on{naive}}$ and $\on{LT}^{\on{true}}$, respectively.

\medskip

However, the identification between $\on{LT}^{\on{naive}}_{\Bun_G}$ and $\on{LT}^{\on{true}}_{\Bun_G}$ (see \secref{sss:non-qc})
gives rise to an identification between $\on{LT}^{\on{naive}}$ and $\on{LT}^{\on{true}}$. 

\begin{rem}

Let us explain the practical implication of the equality $\on{LT}^{\on{naive}}=\on{LT}^{\on{Sht}}$, stated in \thmref{t:lt}. 

\medskip

Let $\CF$ be an an object of $\Shv_\Nilp(\Bun_G)^c$ equipped with a 
weak Weil structure, i.e., a map
\begin{equation} \label{e:}
\alpha:\CF\to (\Frob_{\Bun_G})_*(\CF).
\end{equation} 

To such a pair $(\CF,\alpha)$, we can attach its class
$$\on{cl}(\CF,\alpha)\in \Tr((\Frob_{\Bun_G})_*,\Shv_\Nilp(\Bun_G))$$
(see \cite[Sect. 3.4.3]{GKRV}). Thus, using \corref{c:trace}, to $(\CF,\alpha)$ we can attach an element of  
$$\on{Funct}_c(\Bun_G(\BF_q)),$$ i.e., a compactly supported automorphic function.

\medskip

Now, the content of \thmref{t:lt}
is that the above element of $\on{Funct}_c(\Bun_G(\BF_q))$
equals the function attached to $\CF$, viewed as a weak Weil sheaf via $\alpha$, by the usual sheaf-function
correspondence, i.e., by taking pointwise traces of the Frobenius.

\end{rem}

\ssec{Serre local term}

In this subsection we will define the last remaining map in \secref{t:lt}, denoted $\on{LT}^{\on{Serre}}$.

\sssec{} \label{sss:ps u N Y}

Let $\CY$ be an algebraic stack, and let $\CN\subset T^*(\CY)$ be a conical Zariski-closed subset. 
Consider the (fully faithful) embedding 
\begin{equation} \label{e:Kun N Y}
\Shv_\CN(\CY)\otimes \Shv_\CN(\CY)\to \Shv(\CY\times \CY).
\end{equation} 

Denote 
$$\on{ps-u}_\CY:= \Delta_!(\ul\sfe_\CY)\in \Shv(\CY\times \CY),$$
and let 
$$\on{ps-u}_{\CY,\CN}\in \Shv_\CN(\CY)\otimes \Shv_\CN(\CY)$$
be obtained by applying to $\on{ps-u}_\CY$ the right adjoint to the functor
\eqref{e:Kun N Y}. We will not distinguish notationally between $\on{ps-u}_{\CY,\CN}$
and its image along \eqref{e:Kun N Y}.

\medskip

The counit of the adjunction defines a map
\begin{equation}\label{eq:varepsilon y} 
\on{ps-u}_{\CY,\CN} \to \on{ps-u}_\CY.
\end{equation} 

\sssec{} \label{sss:Serre}

The map \eqref{eq:varepsilon y} gives rise to a natural transformation
\begin{multline} \label{e:cand for * dual}
(\ev^l_\CY\otimes \on{Id}) \circ (\on{Id}\otimes \on{ps-u}_{\CY,\CN}) =
\left((\on{C}^\cdot_c(\CY,-) \circ \Delta_\CY^*)\boxtimes \on{Id}_\CY\right)  \circ (\on{Id}_\CY\boxtimes \on{ps-u}_{\CY,\CN}) \to \\
\to \left((\on{C}^\cdot_c(\CY,-) \circ \Delta_\CY^*)\boxtimes \on{Id}_\CY\right)  \circ (\on{Id}_\CY\boxtimes \on{ps-u}_{\CY}) \simeq  \on{Id}
\end{multline} 
as endofunctors of $\Shv_\CN(\CY)$ (see \secref{sss:ev*} for the notation $\ev^l_\CY$). 

\medskip

Recall, following \cite[Definition 5.5.4]{AGKRRV2}, that the pair $(\CY,\CN)$ is said to be \emph{Serre} 
if the natural transformation \eqref{e:cand for * dual} is an isomorphism. If this is the case, then the data
of 
$$\on{ps-u}_{\CY,\CN}\in  \Shv_\CN(\CY)\otimes \Shv_\CN(\CY) \text{ and }
\ev^l_\CY:\Shv_\CN(\CY)\otimes \Shv_\CN(\CY)\to \Vect$$
define the unit and counit of a self-duality on $\Shv_\CN(\CY)$.

\sssec{} \label{sss:LT Serre}

Suppose that $(\CY,\CN)$ is Serre. As with the true local term morphism, we define
$$\on{LT}_{\CY,\CN}^{\on{Serre}}:\Tr((\Frob_{\CY})_*,\Shv_\CN(\CY)) \to 
\on{Funct}_c(\CY(\BF_q))$$
as the composition
\begin{multline*}
\Tr((\Frob_{\CY})_*,\Shv_\CN(\CY)) \simeq 
\on{C}_c^{\cdot}\big(\CY,
\Delta_\CY^*\circ ((\Frob_{\CY})_* \otimes \on{Id})(\on{ps-u}_{\CY,\CN})\big) \simeq \\
\simeq \on{C}_c^{\cdot}\big(\CY,
\Delta_\CY^*\circ ((\Frob_{\CY})_! \otimes \on{Id})(\on{ps-u}_{\CY,\CN})\big) \overset{\text{\eqref{eq:varepsilon y}}}\longrightarrow 
\on{C}_c^{\cdot}\big(\CY,
\Delta_\CY^*\circ ((\Frob_{\CY})_! \boxtimes \on{Id}_\CY)(\on{ps-u}_{\CY})\big) = \\
=\on{C}_c^{\cdot}\big(\CY,
\Delta_\CY^*\circ (\Frob_{\CY} \times \on{Id}_\CY)_!(\on{ps-u}_{\CY})\big)  = \on{C}_c^{\cdot}\big(\CY,
\Delta_\CY^*\circ (\on{Graph}_{\Frob_\CY})_!(\sfe_{\CY})\big) \simeq \\
\simeq \on{C}_{c}^{\cdot}(\CY^{\Frob},
\sfe_{\CY^{\Frob}}) \simeq \on{Funct}_c(\CY(\BF_q)).
\end{multline*}

\sssec{} \label{sss:projector notation}

We take $\CY=\Bun_G$ and $\CN=\Nilp$. We will denote the endofunctor $\sH_\sR$ by $\sP_{\Nilp,\Bun_G}$
(cf. \cite[Sect. 1.6.1]{AGKRRV2}). 

\medskip

Similarly, for a stack $\CZ$, we will write $\on{Id}_\CZ\boxtimes \sP_{\Nilp,\Bun_G}$
instead of $\on{Id}_\CZ\boxtimes \sH_\sR$. 

\medskip

Using this notation, we have
$$\CK_\sR:=(\on{Id}_{\Bun_G}\boxtimes \sP_{\Nilp,\Bun_G})(\on{ps-u}_{\Bun_G})\in \Shv(\Bun_G\times \Bun_G).$$

\sssec{}

We now quote the following result\footnote{It is this assertion that uses (and is equivalent to) the validity 
of \cite[Conjecture 14.1.8]{AGKRRV1}.}, see \cite[Proposition 2.4.6]{AGKRRV2}:

\begin{thm} \label{t:proj as adj}
For any algebraic stack $\CZ$, the endofunctor of $\Shv(\CZ\times\Bun_G)$ given by 
$$\on{Id}_\CZ \boxtimes \sP_{\Nilp,\Bun_G}$$
identifies with the precomposition of the embedding
$$\Shv(\CZ)\otimes \Shv_{\Nilp}(\Bun_G)\hookrightarrow \Shv(\CZ\times\Bun_G).$$
with its right adjoint.
\end{thm}

In what follows we will denote by $\on{Id}_\CZ \boxtimes \varepsilon$ the natural transformation 
\begin{equation} \label{e:varepsilon}
\on{Id}_\CZ \boxtimes \sP_{\Nilp,\Bun_G}\to \on{Id}_{\Shv(\CZ\times\Bun_G)}
\end{equation}
equal to the counit of the above adjunction. 

\medskip

When $\CZ=\on{pt}$, we will simply write 
$$\varepsilon: \sP_{\Nilp,\Bun_G}\to \on{Id}_{\Shv(\Bun_G)}.$$

\sssec{}

Iterating, from \thmref{t:proj as adj} we obtain:

\begin{cor} \label{c:proj as adj}
The endofunctor $\sP_{\Nilp,\Bun_G} \boxtimes \sP_{\Nilp,\Bun_G}$ of $\Shv(\Bun_G\times \Bun_G)$
identifies with the precomposition of the embedding
\begin{equation} \label{e:double Nilp}
\Shv_{\Nilp}(\Bun_G)\otimes \Shv_{\Nilp}(\Bun_G)\hookrightarrow \Shv(\Bun_G\times\Bun_G).
\end{equation}
with its right adjoint.
\end{cor}

We will denote by $\varepsilon\boxtimes \varepsilon$ the resulting natural transformation
$$\sP_{\Nilp,\Bun_G} \boxtimes \sP_{\Nilp,\Bun_G}\to \on{Id}_{\Shv(\Bun_G\times \Bun_G)}.$$

\sssec{}

Recall that \thmref{t:u Nilp Bun} says that the object $\CK_\sR$ belongs to the essential image of the functor
\eqref{e:double Nilp}. This formally implies that the map
$$(\sP_{\Nilp,\Bun_G} \boxtimes \on{Id}_{\Bun_G})(\CK_\sR) \overset{\varepsilon\boxtimes \on{Id}_{\Bun_G}}\longrightarrow 
\CK_\sR$$
is an isomorphism.

\medskip

In other words, we obtain an identification
$$\CK_\sR\simeq (\sP_{\Nilp,\Bun_G} \boxtimes \sP_{\Bun_G,\Nilp})(\on{ps-u}_{\Bun_G}).$$

By \corref{c:proj as adj}, the resulting map
$$\CK_\sR \overset{\varepsilon\boxtimes \varepsilon}\longrightarrow \on{ps-u}_{\Bun_G}$$
identifies $\CK_\sR$ with the value on $\on{ps-u}_{\Bun_G}$ of the right adjoint to \eqref{e:double Nilp}. 

\medskip

Hence, in the notations of \secref{sss:ps u N Y} above, we have
$$\CK_\sR\simeq \on{ps-u}_{\Bun_G,\Nilp}.$$

\sssec{}

Combining with \thmref{t:duality}, we conclude that the pair $(\Bun_G,\Nilp)$ is Serre.

\medskip

Thus, we obtain a well-defined map
$$\on{LT}^{\on{Serre}} := \on{LT}^{\on{Serre}} _{\Bun_G,\Nilp}.$$

\sssec{}

At this point, all the terms in \thmref{t:lt} have been defined. We have already seen that
$\on{LT}^{\on{naive}} = \on{LT}^{\on{true}}$.
 
\medskip 

The rest of this section is devoted to the proof of the equality $\on{LT}^{\on{Serre}} = \on{LT}^{\on{Sht}}$. 

\medskip

Finally, we show 
$\on{LT}^{\on{Serre}} = \on{LT}^{\on{true}}$ in \secref{s:serre-true}.

\ssec{Comparison of $\on{LT}^{\on{Serre}}$ and $\on{LT}^{\Sht}$} \label{ss:serre sht}


\sssec{} We begin by constructing upgraded versions of the two 
local term morphisms in question. 

\medskip 

More precisely, we will show that there exist natural transformations
$$
\widetilde{\on{LT}}{}^{\on{Serre}}, \widetilde{\on{LT}}{}^{\Sht}:
\Sht^{\Tr} \to \Sht
$$
of functors $$\Rep(\cG)_\Ran \to \Vect$$
with the property that 
when we evaluate either of these natural transformations on 
$\one_{\Rep(\cG)_\Ran}$, we obtain the 
relevant local term map 
$$\on{LT}^{?}:\Tr((\Frob_{\Bun_G})_*,\Shv_\Nilp(\Bun_G)) =
\Sht^{\Tr}(\one_{\Rep(\cG)_\Ran}) \to 
\Sht(\one_{\Rep(\cG)_\Ran}) \simeq \on{Funct}_c(\Bun_G(\BF_q)).$$

\sssec{}

The natural transformation $\widetilde{\on{LT}}{}^{\Sht}$ (in fact, an isomorphism) has been already defined:
this is the isomorphism arising from \thmref{t:main}. 

\sssec{}

We will now construct $\widetilde{\on{LT}}{}^{\on{Serre}}$.

\medskip

Recall the isomorphism $$\Sht^{\Tr}(-) \simeq \Sht(\sR \star -)$$ of
\thmref{t:trace r}. Thus, we can interpret 
the sought-for map $\widetilde{\on{LT}}{}^{\on{Serre}}$ as a natural transformation
\begin{equation} \label{e:LT Serre R}
\Sht(\sR \star -)\to \Sht(-).
\end{equation}

\sssec{} \label{sss:construct Serre}

By construction, the functors $\Sht(-)$ and $\Sht(\sR \star -)$ send $\CV\in \Rep(\cG)_\Ran$ 
to the vector space obtained by applying 
$$\on{C}^\cdot_c(\Bun_G,(\on{Graph}_{\Frob_{\Bun_G}})^*(-)):\Shv(\Bun_G\times \Bun_G)\to \Vect$$
to the objects
$$(\on{Id}_{\Bun_G}\boxtimes \sH_\CV)(\on{ps-u}_{\Bun_G}) \text{ and }  
(\on{Id}_{\Bun_G}\boxtimes \sH_{\sR\star \CV})(\on{ps-u}_{\Bun_G}),$$
respectively. 

\medskip

The sought-for natural transformation \eqref{e:LT Serre R} is induced by the natural transformation of functors
$$\on{Id}_{\Bun_G}\boxtimes \sH_{\sR\star \CV}\to \on{Id}_{\Bun_G}\boxtimes \sH_{\CV}$$
given by
\begin{multline*}
\on{Id}_{\Bun_G}\boxtimes \sH_{\sR\star \CV} \simeq
(\on{Id}_{\Bun_G}\boxtimes \sH_{\sR})\circ (\on{Id}_{\Bun_G}\boxtimes \sH_{\CV})=\\
=(\on{Id}_{\Bun_G}\boxtimes \sP_{\Nilp,\Bun_G})\circ (\on{Id}_{\Bun_G}\boxtimes \sH_{\CV})
\overset{(\on{Id}_{\Bun_G}\boxtimes \varepsilon)\circ {id}}\longrightarrow \on{Id}_{\Bun_G}\boxtimes \sH_{\CV},
\end{multline*} 
where $\on{Id}_{\Bun_G}\boxtimes \varepsilon$ is as in \eqref{e:varepsilon}. 
%
%
%
%
%


\sssec{}

We will prove: 

\begin{thm}\label{t:serre sht}

There is a canonical isomorphism 
$$\widetilde{\on{LT}}{}^{\on{Serre}} \simeq \widetilde{\on{LT}}{}^{\Sht}$$
of natural transformations
$$\Sht^{\Tr} \to \Sht.$$

\end{thm}

Clearly, \thmref{t:serre sht} implies the isomorphism $\on{LT}^{\on{Serre}}\simeq \on{LT}^{\Sht}$. 
The proof of \thmref{t:serre sht} will be given in \secref{ss:proof of serre sht}. 

\ssec{An algebra structure on $\sR$}

For the proof of \thmref{t:serre sht} we need to digress and discuss some properties related
to the commutative algebra structure on the object
$\sR \in \Rep(\cG)_\Ran$, see \secref{sss:the projector}. 

\sssec{} 

In what follows we will need one more compatibility property of the commutative algebra structure on $\sR$.

\medskip

Let $\CZ$ be an algebraic stack. Consider the endofunctor $(\on{Id}_\CZ\boxtimes \sP_{\Nilp,\Bun_G})$ of 
$\Shv(\CZ\times \Bun_G)$, see \secref{sss:projector notation}. 

\medskip

On the one hand, the algebra structure on $\CR$ yields a map
\begin{equation} \label{e:m map}
m:(\on{Id}_\CZ\boxtimes \sP_{\Nilp,\Bun_G}) \circ (\on{Id}_{\CZ}\boxtimes \sP_{\Nilp,\Bun_G})\to (\on{Id}_\CZ\boxtimes \sP_{\Nilp,\Bun_G}).
\end{equation} 

On the other hand, we have the map 
\begin{equation} \label{e:idemp map}
(\on{Id}_\CZ\boxtimes \sP_{\Nilp,\Bun_G}) \circ 
(\on{Id}_\CZ\boxtimes \sP_{\Nilp,\Bun_G})
\overset{(\on{Id}_\CZ\boxtimes \varepsilon) \circ \on{id}_{(\on{Id}_\CZ\boxtimes \sP_{\Nilp,\Bun_G})}}\longrightarrow 
(\on{Id}_\CZ\boxtimes \sP_{\Nilp,\Bun_G}),
\end{equation}
where $\on{Id}_\CZ\boxtimes \varepsilon$ is as in \eqref{e:varepsilon}.
(The map \eqref{e:idemp map} is in fact an isomorphism and equals the structure on
$\on{Id}_\CZ\boxtimes \sP_{\Nilp,\Bun_G}$ of idempotent endofunctor.) 

\begin{prop} \label{p:nu kr}
The maps \eqref{e:m map} and \eqref{e:idemp map} are canonically homotopic.
\end{prop}

\sssec{}
 
Before we prove \propref{p:nu kr} let us quote its corollary that we will use in the proof of
\thmref{t:serre sht}. 

\medskip

Note that for $\CV\in \Rep(\cG)_\Ran$ we have a tautological identification
\begin{equation} \label{e:V R}
(\on{Id}_{\Bun_G}\boxtimes \sP_{\Nilp,\Bun_G}) (\CK_\CV) \simeq \CK_{\sR\star \CV}
\end{equation}
as objects of $\Shv(\Bun_G\times \Bun_G)$.

\medskip

For $\CW\in  \Rep(\cG)_\Ran$, let
$$\varepsilon_{\CW}:(\on{Id}_\CZ\boxtimes \sP_{\Nilp,\Bun_G})(\CK_{\CW}) \to \CK_{\CW}$$
denote the value of the natural transformation $\on{Id}_{\Bun_G}\boxtimes \varepsilon$ on $\CK_\CW$. 

\medskip

\begin{cor} \label{c:nu kr}
For $\CV\in \Rep(\cG)_\Ran$, the diagram
$$
\CD
(\on{Id}_{\Bun_G}\boxtimes \sP_{\Nilp,\Bun_G}) \circ 
(\on{Id}_{\Bun_G}\boxtimes \sP_{\Nilp,\Bun_G}) (\CK_\CV) @>{m}>> (\on{Id}_{\Bun_G}\boxtimes \sP_{\Nilp,\Bun_G}) (\CK_\CV) \\
@V{\text{\eqref{e:V R}}}V{\sim}V   @V{\text{\eqref{e:V R}}}V{\sim}V   \\
(\on{Id}_{\Bun_G}\boxtimes \sP_{\Nilp,\Bun_G})(\CK_{\sR\star \CV})  @>{\varepsilon_{\sR\star \CV}}>> \CK_{\sR\star \CV}
\endCD
$$
commutes.
\end{cor}

\sssec{Proof of \propref{p:nu kr}}

Both the source and the target functor vanish on the subcategory $\on{ker}(\on{Id}_\CZ\boxtimes \sP_{\Nilp,\Bun_G})$. Hence, it is sufficient to establish 
the commutativity when evaluated on the full subcategory
$$\Shv(\CZ)\otimes \Shv_\Nilp(\Bun_G)\subset \Shv(\CZ\times \Bun_G).$$

Since all the functors involved act only on the second factor, it is sufficient to show that the map
$$\on{Id} \simeq (\sP_{\Nilp,\Bun_G} \circ \sP_{\Nilp,\Bun_G})|_{\Shv_\Nilp(\Bun_G)} \overset{m}\to  \sP_{\Nilp,\Bun_G}|_{\Shv_\Nilp(\Bun_G)}\simeq \on{Id}$$
is the identity map. 

\medskip 

However, this follows from \cite[Theorem 14.3.2]{AGKRRV1}:

\medskip

This theorem implies that the action of $\Rep(\cG)_\Ran$ on 
$\Shv_\Nilp(\Bun_G)$ factors uniquely
through an action of $\QCoh(\LocSys^{\on{restr}}_\cG(X))$ 
via the functor $\Loc$. 

\medskip

Now the result follows from the fact that the isomorphism
$$\Loc(\sR)\simeq \CO_{\LocSys^{\on{restr}}_\cG(X)}$$
of \corref{c:diagonal LocSys} respects the (commutative) algebra structures. 

\qed[\propref{p:nu kr}]

\ssec{Proof of \thmref{t:serre sht}} \label{ss:proof of serre sht}

\sssec{Proof of \thmref{t:serre sht}, Step 0}

We begin by introducing some notation.

\medskip 

Throughout the argument, we will replace $\Sht^{\Tr}$ with 
$\Sht(\sR \star -)$. In particular, we consider
our natural transformations as mapping
$$\widetilde{\on{LT}}{}^{\on{Serre}},\widetilde{\on{LT}}{}^{\on{Sht}}:
\Sht(\sR \star -) \to \Sht.$$
For a fixed object $\CV \in \Rep(\cG)_\Ran$, we denote the corresponding
maps by 
$$\widetilde{\on{LT}}{}_{\CV}^{\on{Serre}},\widetilde{\on{LT}}{}_{\CV}^{\on{Sht}}:
\Sht(\sR \star \CV) \to \Sht(\CV).$$

\medskip 

To unburden the notation,
we will use $-\otimes-$ to denote 
$-\underset{\CO_{\LocSys^{\on{restr}}_\cG(X)}}{\otimes}-$ and $\CO$ to denote 
$\CO_{\LocSys^{\on{restr}}_\cG(X)}$.

\sssec{Proof of \thmref{t:serre sht}, Step 1}

Observe that we have a commutative diagram
$$\CD 
\bMaps_{\DGCat}(\QCoh(\LocSys^{\on{restr}}_\cG(X),\Vect) @>{F \mapsto 
F(\CO\otimes-)}>> 
\bMaps_{\DGCat}(\QCoh(\LocSys^{\on{restr}}_\cG(X),\Vect) \\
@V{\Loc^{\vee}}VV @VV{\Loc^{\vee}}V \\
\bMaps_{\DGCat}(\Rep(\cG)_\Ran,\Vect) @>{F \mapsto F(\sR\star-)}>> \bMaps_{\DGCat}(\Rep(\cG)_\Ran,\Vect).
\endCD$$
Here we have used the isomorphism $\Loc(\sR) \simeq \CO$.
The vertical arrows are fully faithful by \corref{c:Loc dual}(a), and the
top horizontal arrow is tautologically isomorphic to the identity. Therefore, the bottom arrow
is fully faithful when restricted to the essential image of $\Loc^{\vee}$.

\medskip 

The (isomorphic!) functors $\Sht$ and $\Sht(\sR \star -)$ lie 
in the essential image of $\Loc^{\vee}$ by \thmref{t:Cong cor}.
Therefore, it suffices to identify the natural transformations
$\widetilde{\on{LT}}{}^{\on{Serre}}$ and $\widetilde{\on{LT}}{}^{\on{Sht}}$
after precomposing with $\sR \star -$. 

\medskip

That is, it suffices to identify 
the two induced natural transformations 
$$\Sht(\sR \star (\sR \star -)) \to \Sht(\sR \star -),$$
i.e., the two maps
$$\widetilde{\on{LT}}{}_{\sR\star \CV}^{\on{Serre}},\widetilde{\on{LT}}{}_{\sR\star \CV}^{\on{Sht}}:
\Sht(\sR \star (\sR \star \CV)) \to \Sht(\sR \star \CV),\quad \CV\in \Rep(\cG)_\Ran.$$

\sssec{Proof of \thmref{t:serre sht}, Step 2}

Let $m:\sR \star \sR \to \sR$ denote the multiplication for the algebra
structure on $\sR$.

\medskip 

We obtain a map 
$$m_\CV:\Sht(\sR \star \sR \star \CV) \to 
\Sht(\sR \star \CV).$$

\medskip 

We claim that there is a natural identification
\begin{equation}\label{eq:lambda 1}
m_\CV \simeq \widetilde{\on{LT}}{}^{\Sht}_{\sR \star \CV}
\end{equation}
of morphisms $\Sht(\sR \star \sR \star \CV) \to 
\Sht(\sR \star \CV)$.

\medskip 

Indeed, the compatibility of the isomorphism $\Loc(\sR)\simeq \CO$ with algebra structures 
implies that we have a commutative diagram
$$\CD
\Loc(\sR) \otimes \Loc(\sR \star \CV) 
@>{\sim}>>
\Loc(\sR \star \sR \star \CV) \\
@V{\sim}VV @VV{\Loc(m \star \on{Id}_\CV)}V \\
\Loc(\sR \star \CV) @>{\on{Id}}>> \Loc(\sR \star \CV)
\endCD
$$
where the left vertical arrow is the canonical isomorphism obtained by 
identifying $\Loc(\CV) \simeq \CO$.
Applying $\Sht_{\Loc}$ and the definition yields the claim.

\sssec{Proof of \thmref{t:serre sht}, Step 3}

By the above, it suffices to show that there are natural identifications 
\begin{equation}\label{eq:lambda 2}
m_\CV \simeq \widetilde{\on{LT}}{}^{\on{Serre}}_{\sR \star \CV}
\end{equation}
of morphisms $\Sht(\sR \star \sR \star \CV) \to 
\Sht(\sR \star \CV)$.

\medskip 

Now,  \eqref{eq:lambda 2} is obtained by applying $\on{C}^\cdot_c(\Bun_G,(\on{Graph}_{\Frob_{\Bun_G}})^*(-))$ to
the commutative diagram of \corref{c:nu kr}.

\section{Comparison of $\on{LT}^{\on{true}}$ and $\on{LT}^{\on{Serre}}$} \label{s:serre-true}

\ssec{Statement of the result} \label{ss:setting for Serre}

\sssec{} \label{sss:setting for Serre}

Let $\CY$ be a quasi-compact algebraic stack, and let $\CN$ be a conical Zariski-closed
subset of $T^*(\CY)$. We will assume that the subcategory 
\begin{equation} \label{e:emb N}
\Shv_\CN(\CY)\hookrightarrow \Shv(\CY)
\end{equation} 
is generated by objects that are compact in $\Shv(\CY)$
(in \cite[Sect. A.5.2]{AGKRRV2} this property of $(\CY,\CN)$ was termed ``constraccessible").

\medskip

Assume that $\CN$ is Frobenius-invariant, so the endofunctor
$(\Frob_\CY)_*$ of $\Shv(\CY)$ preserves the subcategory $\Shv_\CN(\CY)\subset \Shv(\CY)$,
see \cite[Sect. 22.3.1 and Lemma 22.3.2]{AGKRRV1}. 

\medskip

In this case, the embedding \eqref{e:emb N} induces a map
\begin{equation} \label{e:emb N Tr}
\Tr((\Frob_\CY)_*,\Shv_\CN(\CY))\to \Tr((\Frob_\CY)_*,\Shv(\CY)).
\end{equation} 

\medskip

Let us denote by
$$\on{LT}^{\on{true}}_{\CY,\CN}:\Tr((\Frob_\CY)_*,\Shv_\CN(\CY))\to \on{Funct}(\CY(\BF_q))$$
the composition of \eqref{e:emb N Tr} with the map
$$\on{LT}^{\on{true}}_{\CY}:\Tr((\Frob_\CY)_*,\Shv(\CY))\to \on{Funct}(\CY(\BF_q)),$$
defined in \secref{sss:true}. 

\sssec{}

Assume now that the pair $(\CY,\CN)$ is Serre (see \secref{sss:Serre} for what this means).

\medskip

Recall that in this case, we also have the map 
$$\on{LT}^{\on{Serre}}_{\CY,\CN}: \Tr((\Frob_\CY)_*,\Shv_\CN(\CY))\to \on{Funct}(\CY(\BF_q)).$$

\sssec{}

Recall now the miraculous endofunctor $\Mir_\CY$, see \cite[Sect. 5.6.2]{AGKRRV2}. Following
\cite[Definition 5.6.6]{AGKRRV2}, we will say that the pair $\CN$ is \emph{miraculous-compatible} if the
endofunctor $\Mir_\CU$ preserves the subcategory \eqref{e:emb N}. 

\medskip

The main result of this section reads:

\begin{thm} \label{t:true and Serre}
Assume that $\CN$ is miraculous-compatible. Then the maps $\on{LT}^{\on{true}}_{\CY,\CN}$ and $\on{LT}^{\on{Serre}}_{\CY,\CN}$
$$\Tr((\Frob_\CY)_*,\Shv_\CN(\CY))\rightrightarrows \on{Funct}(\CY(\BF_q))$$
are equal. 
\end{thm}

\begin{rem}
The assertion of \thmref{t:true and Serre} is far from tautological: it says that two ways to map $\Tr((\Frob_\CY)_*,\Shv_\CN(\CY))$
to $\on{Funct}(\CY(\BF_q))$, corresponding to two different self-dualities on $\Shv_\CN(\CY)$ coincide. 

\medskip

A somewhat analogous problem arises when we calculate the trace of the identity endofunctor on the category
$\QCoh(Z)$, where $Z$ is a smooth proper scheme. There are two ways to calculate the trace that correspond
to two choices of self-duality data on $\QCoh(Z)$: the naive self-duality and Serre self-duality. Each calculation
yields Hodge cohomology of $Z$, i.e.,
$$\underset{i}\oplus\, \Gamma(Z,\Omega^i(Z))[i].$$

However, the resulting two identifications are different, and the difference is given by the Todd class of $Z$. This
observation lies at the core of a proof of the Grothendieck-Riemann-Roch theorem via categorical traces, see
\cite{KP}.

\end{rem}

\sssec{} \label{sss:truncatable}

Let us explain how \thmref{t:true and Serre} implies the equality $\on{LT}^{\on{true}}=\on{LT}^{\on{Serre}}$
(the issue here is the fact that $\Bun_G$ is not quasi-compact). 

\medskip

Let $\CY$ be a not necessarily quasi-compact algebraic stack, and let $\CN\subset T^*(\CY)$ be a conical 
Zariski-closed subset. We will recall some definitions from \cite[Sect. C.1]{AGKRRV2}. 

\medskip

An open substack $\CU\overset{j}\hookrightarrow \CY$ is said to be \emph{cotruncative} if for every quasi-compact 
open $\CU'\subset \CY$, the open embedding 
$$\CU\cap \CU' \overset{j}\hookrightarrow \CU'$$
is such that the functor 
$$j_*:\Shv(\CU\cap \CU')\to \Shv(\CU')$$
admits a right adjoint \emph{as a functor defined by a kernel}. 

\medskip

An open substack $\CU\overset{j}\hookrightarrow \CY$ is said to be universally $\CN$-\emph{cotruncative} if it is cotruncative, and for every 
stack $\CZ$ and $\CN_\CZ\subset T^*(\CZ)$, the functor
$$(\on{id}\times j)_!:\Shv(\CZ\times \CU) \to \Shv(\CZ\times \CY)$$
sends $\Shv_{\CN_\CZ\times \CN}(\CZ\times \CU) \subset \Shv(\CZ\times \CU)$ to 
$\Shv_{\CN_\CZ\times \CN}(\CZ\times \CY) \subset \Shv(\CZ\times \CY)$. 

\medskip 

Recall that $\CY$ is said to be \emph{truncatable} if we can write $\CY$ as a union of quasi-compact 
cotruncative open substacks. Finally, recall that $\CY$ is is said to be universally $\CN$-\emph{truncatable} if we can write $\CY$ as a filtered union of 
quasi-compact universally $\CN$-cotruncative open substacks.

\sssec{} \label{sss:cond non qc}

Assume that $\CN$ is Frobenius-invariant. By \secref{sss:non-qc} we have a well-defined map
$$\on{LT}^{\on{true}}_{\CY}:\Tr((\Frob_\CY)_*,\Shv(\CY))\to \on{Funct}_c(\CY(\BF_q)).$$

Let us make the following assumptions:

\begin{itemize}

\item $\CY$ is universally $\CN$-truncatable;

\item For every quasi-compact universally $\CN$-cotruncative open substack $\CU\subset \CY$, we have:

\begin{itemize}

\item The pair $(\CU,\CN|_{\CU})$ is constraccessible;

\item The pair $(\CU,\CN|_{\CU})$ is Serre;

\item $\CN|_{\CU}$ is miraculous-compatible.

\end{itemize}

\end{itemize} 

It follows that in this case, the subcategory $\Shv_\CN(\CY)$ is generated by objects that are compact
in $\Shv(\CY)$, and that the pair $(\CY,\CN)$ is Serre. So, the map
$$\on{LT}^{\on{Serre}}_{\CY,\CN}:\Tr((\Frob_\CY)_*,\Shv_\CN(\CY))\to \on{Funct}_c(\CY(\BF_q))$$
is also well-defined by \secref{sss:LT Serre}. Moreover, we have a commutative diagram
$$
\CD
\underset{\CU}{\on{colim}}\, \Tr((\Frob_\CU)_*,\Shv_\CN(\CU)) @>{\underset{\CU}{\on{colim}}\, \on{LT}^{\on{Serre}}_{\CU,\CN}}>> 
\underset{\CU}{\on{colim}}\, \on{Funct}(\CU(\BF_q)) \\
@VVV @VV{\sim}V \\
\Tr((\Frob_\CY)_*,\Shv_\CN(\CY)) @>{\on{LT}^{\on{Serre}}_{\CY,\CN}}>> \on{Funct}_c(\CY(\BF_q)).
\endCD
$$

\medskip

It now follows formally from Theorems \ref{t:true and Serre} and \ref{t:gv}(b) that the maps $\on{LT}^{\on{true}}_{\CY,\CN}$ 
and $\on{LT}^{\on{Serre}}_{\CY,\CN}$ are canonically
homotopic. 

\sssec{}

We apply the above discussion to $\CY=\Bun_G$ and $\Nilp=\CN$. The conditions in \secref{sss:cond non qc}
are satisfied by \cite[Theorem 1.7.3]{AGKRRV2}, \cite[Conjecture 14.1.8 and Lemma F.8.10]{AGKRRV1},  
\cite[Corollary 5.8.4]{AGKRRV2} and \cite[Proposition 2.8.10]{AGKRRV2}, respectively. 
This implies the desired equality
$$\on{LT}^{\on{true}}=\on{LT}^{\on{Serre}}.$$

\ssec{A geometric local term theorem} \label{ss:geom lt}

\sssec{}

Let $\CY$ be a quasi-compact algebraic stack. We start by constructing a natural transformation
\begin{equation} \label{e:fund trans}
\on{C}_c^\cdot(\CY, \Delta_\CY^* \circ (\on{Id}\otimes \Mir_\CY)(-)) \to \on{C}^\cdot_\blacktriangle(\CY,\Delta^!_\CY(-)),
\end{equation} 
as functors 
$$\Shv(\CY\times \CY)\rightrightarrows \Vect.$$

\sssec{} \label{sss:construct psi map}

It suffices to specify the value of the natural transformation \eqref{e:fund trans} on compact objects. Thus, we fix
a compact object $\CQ\in \Shv(\CY\times \CY)$. 

\medskip

We note that the map
$$\on{C}^\cdot_\blacktriangle(\CY,\Delta^!_\CY(\CQ))\to \on{C}^\cdot(\CY,\Delta^!_\CY(\CQ))$$
is an isomorphism for $\CQ$. So we need to construct a map 
\begin{equation} \label{e:fund trans 0}
\on{C}_c^\cdot(\CY, \Delta_\CY^* \circ (\on{Id}\otimes \Mir_\CY)(\CQ)) \to \on{C}^\cdot(\CY,\Delta^!_\CY(\CQ)),
\end{equation} 
functorial in $\CQ\in \Shv(\CY\times \CY)^c$.

\medskip

In what follows we will use the notation
$$\on{u}_\CY:=(\Delta_\CY)_*(\omega_\CY)\in \Shv(\CY\times \CY).$$

\medskip

We start with the map
$$\CQ\boxtimes \BD^{\on{Verdier}}(\CQ) \to \on{u}_{\CY\times \CY},$$
given by Verdier duality. Applying the transposition $\sigma_{2,3}$, we interpret it as a map
\begin{equation} \label{e:key serre map pre}
(\CQ\boxtimes \BD^{\on{Verdier}}(\CQ))^{\sigma_{2,3}}\to \on{u}_\CY\boxtimes \on{u}_\CY.
\end{equation}

By definition
\begin{equation} \label{e:Mir to id}
(\on{Id}\otimes \Mir_\CY)(\on{u}_\CY)\simeq \on{ps-u}_\CY
\end{equation}

Applying the functor $\on{Id}\otimes \on{Id}\otimes \Mir_\CY \otimes \on{Id}$ to the map \eqref{e:key serre map pre},
we obtain a map
\begin{equation} \label{e:key serre map}
((\on{Id}\otimes \Mir_\CY)(\CQ)\boxtimes \BD^{\on{Verdier}}(\CQ))^{\sigma_{2,3}}\to \on{u}_\CY\boxtimes \on{ps-u}_\CY.
\end{equation} 

Applying the functor
$$(p_{1,2})_!\circ (\Delta_\CY \times \on{id} \times \on{id})^* \circ \sigma_{2,3}:
\Shv(\CY\times \CY\times \CY\times \CY)\to \Shv(\CY\times \CY),$$
from \eqref{e:key serre map} we obtain a map
$$\on{C}_c^\cdot(\CY, \Delta_\CY^* \circ (\on{Id}\otimes \Mir_\CY)(\CQ)) 
\otimes \BD^{\on{Verdier}}(\CQ) \to \on{u}_\CY.$$

Applying Verdier duality again, we obtain the desired map \eqref{e:fund trans 0}. 

\sssec{}

Let us take $\CQ:=(\on{Graph}_{\Frob_\CY})_*(\omega_\CY)\in \Shv(\CY\times \CY)$. 

\medskip

Then the right-hand side in \eqref{e:fund trans} identifies with 
$$\on{C}^\cdot_\blacktriangle(\CY^\Frob,\omega_{\CY^\Frob})\simeq \on{C}^\cdot(\CY^\Frob,\omega_{\CY^\Frob}).$$

\medskip

Since $\Frob_\CY$ is a proper map,
$$(\on{Id}\otimes \Mir_\CY)((\on{Graph}_{\Frob_\CY})_*(\omega_\CY))\simeq (\on{Graph}_{\Frob_\CY})_!(\ul\sfe_\CY).$$
Hence, the left-hand side in \eqref{e:fund trans} identifies with 
$$\on{C}_c^\cdot(\CY^\Frob,\ul\sfe_{\CY^\Frob}).$$

\sssec{}

We will prove:

\begin{thm} \label{t:true and Serre geom}
The diagram
\begin{equation} \label{e:true and Serre 1}
\CD
\on{C}_c^\cdot(\CY, \Delta_\CY^* \circ (\on{Id}\otimes \Mir_\CY)\circ (\on{Graph}_{\Frob_\CY})_*(\omega_\CY)) 
@>{\text{\eqref{e:fund trans}}}>> 
\on{C}^\cdot_\blacktriangle(\CY,\Delta^!_\CY\circ (\on{Graph}_{\Frob_\CY})_*(\omega_\CY)) \\
@V{\simeq}VV @VV{\simeq}V \\
\on{C}_c^\cdot(\CY^\Frob,\ul\sfe_{\CY^\Frob}) & & \on{C}^\cdot(\CY^\Frob,\omega_{\CY^\Frob}) \\
@V{\simeq}VV @VV{\simeq}V  \\
\on{Funct}(\CY(\BF_q)) @>{\on{id}}>> \on{Funct}(\CY(\BF_q)) 
\endCD
\end{equation} 
commutes. 
\end{thm}

The proof will be given in \secref{ss:proof geom}. We will presently show how \thmref{t:true and Serre geom} implies \thmref{t:true and Serre}. 

\ssec{Proof of \thmref{t:true and Serre}} \label{ss:proof true and Serre}

\sssec{} \label{sss:usual duality N}

Let $\ev_\CY$ denote the pairing
$$\on{C}^\cdot_\blacktriangle(\CY,\Delta^!_\CY(-\boxtimes -)):\Shv_\CN(\CY)\otimes \Shv_\CN(\CY)\to \Vect.$$

The assumption that $\Shv_\CN(\CY)$ is generated by objects compact in the ambient category $\Shv(\CY)$ implies
that the restriction of $\ev_\CY$ to 
$$\Shv_\CN(\CY)\otimes \Shv_\CN(\CY)\subset \Shv(\CY)\otimes \Shv(\CY)$$
defines an identification
$$\Shv_\CN(\CY)\simeq \Shv_\CN(\CY)^\vee,$$
see \cite[Sect. A.5.4]{AGKRRV2}. 

\medskip

The unit of this duality, to be denoted $\on{u}_{\CY,\CN}$ is obtained by applying to
$\on{u}_\CY$ the right adjoint to the fully faithful embedding
\begin{equation} \label{e:N Y emb again}
\Shv_\CN(\CY)\otimes \Shv_\CN(\CY)\hookrightarrow \Shv(\CY\times \CY),
\end{equation} 
see \cite[Corollary 5.4.5]{AGKRRV2}. 

\medskip

In what follows, we will not distinguish notationally between $\on{u}_{\CY,\CN}$ and its image along \eqref{e:N Y emb again}. 
The counit of the adjunction gives rise to a map
\begin{equation} \label{e:u N Y}
\on{u}_{\CY,\CN} \to \on{u}_\CY.
\end{equation} 

\sssec{}

The assumption that $\CN$ is miraculous-compatible implies that the endofunctor $\on{Id}_\CY \boxtimes \Mir_\CY$ of $\Shv(\CY\times \CY)$ 
preserves the subcategory 
\begin{equation} \label{e:N embed}
\Shv_\CN(\CY)\otimes \Shv_\CN(\CY)\hookrightarrow \Shv(\CY\times \CY).
\end{equation} 

The resulting endofunctor of $\Shv_\CN(\CY)\otimes \Shv_\CN(\CY)$ identifies with $\on{Id} \otimes \Mir_\CY$.

\medskip

By \cite[Corollary 5.6.10(c)]{AGKRRV2}, the assumption that the pair $(\CY,\CN)$ is Serre implies that the above 
endofunctor $\on{Id} \otimes \Mir_\CY$ intertwines the duality data given by the pair
$(\on{u}_{\CY,\CN},\on{ev}_\CY)$ with one given by $(\on{ps-u}_{\CY,\CN},\on{ev}^l_\CY)$.

\medskip

In particular, we have a canonical isomorphism 
\begin{equation} \label{e:Mir diag N}
\on{ps-u}_{\CY,\CN} \simeq (\on{Id} \otimes \Mir_\CY)(\on{u}_{\CY,\CN})
\end{equation} 
and a datum of commutativity for the diagram
$$
\CD
\Shv_\CN(\CY)\otimes \Shv_\CN(\CY) @>{\on{Id}\otimes \Mir_\CY}>> \Shv_\CN(\CY)\otimes \Shv_\CN(\CY) \\
@V{\on{ev}_\CY}VV @VV{\on{ev}^l_\CY}V \\
\Vect @>{\on{Id}}>> \Vect,
\endCD
$$
i.e., an isomorphism of functors 
\begin{equation} \label{e:fund trans N}
\on{ev}^l_\CY \circ (\on{Id}\otimes \Mir_\CY) \simeq \on{ev}_\CY, \quad 
\Shv_\CN(\CY)\otimes \Shv_\CN(\CY) \rightrightarrows \Vect.
\end{equation}


\sssec{}

In particular, for an endofunctor $F$ of $\Shv_\CN(\CY)$, we have a commutative diagram
\begin{equation} \label{e:true and Serre 2}
\CD
\Tr(F,\Shv_\CN(\CY))  & @>{\on{id}}>>  & \Tr(F,\Shv_\CN(\CY))  \\
@V{\sim}VV  & & @VV{\sim}V \\
\on{ev}^l_\CY\circ (F\otimes \on{id})(\on{ps-u}_{\CY,\CN})
@>{\sim}>{\text{\eqref{e:Mir diag N}}}> \on{ev}^l_\CY\circ (F\otimes \Mir_\CY)(\on{u}_{\CY,\CN}) @>{\sim}>{\text{\eqref{e:fund trans N}}}>
\on{ev}_\CY\circ (F\otimes \on{id})(\on{u}_{\CY,\CN}) 
\endCD
\end{equation} 

\sssec{}

We now use the following two observations: 

\medskip

\noindent(i) The diagram
$$
\CD
\on{ps-u}_{\CY,\CN}  @>{\text{\eqref{e:Mir diag N}}}>{\sim}> (\on{Id} \otimes \Mir_\CY)(\on{u}_{\CY,\CN})   \\
@V{\text{\eqref{eq:varepsilon y}}}VV @VV{\text{\eqref{e:u N Y}}}V \\
\on{ps-u}_{\CY} @>>{\text{\eqref{e:Mir to id}}}> (\on{Id} \otimes \Mir_\CY)(\on{u}_{\CY}) 
\endCD
$$
commutes. This follows tautologically from the constructions. 

\medskip

\noindent(ii) The isomorphism \eqref{e:fund trans N} is canonically homotopic to
the restriction of the natural transformation \eqref{e:fund trans} along the embedding \eqref{e:N embed}.
This follows from \cite[Proposition 5.7.2]{AGKRRV2}. 

\medskip

Concatenating, we obtain a commutative diagram
$$
\CD
\on{ev}^l_\CY\circ ((\Frob_\CY)_*\otimes \on{id})(\on{ps-u}_{\CY,\CN})) @>{\text{\eqref{e:fund trans N}}\circ \text{\eqref{e:Mir diag N}}}>> 
\on{ev}_\CY\circ ((\Frob_\CY)_*\otimes \on{id})(\on{u}_{\CY,\CN})) \\
@VV{\text{\eqref{eq:varepsilon y}}}V @VV{\text{\eqref{e:u N Y}}}V \\
\on{C}_c^\cdot(\CY,-)\circ (\Delta_\CY)^*\circ (\Frob_\CY\times \on{id})_* (\on{ps-u}_{\CY})
@>{\text{\eqref{e:fund trans}} \circ \text{\eqref{e:Mir to id}}}>>
\on{C}_\blacktriangle^\cdot(\CY,-)\circ (\Delta_\CY)^!\circ (\Frob_\CY\times \on{id})_* (\on{u}_{\CY}). 
\endCD
$$

\sssec{} \label{sss:Serre diag 3}

Hence, concatenating with \eqref{e:true and Serre 2} for $F=(\Frob_\CY)_*$, we obtain a commutative diagram
$$
\CD
\Tr((\Frob_\CY)_*,\Shv_\CN(\CY))   @>{\on{id}}>>   \Tr((\Frob_\CY)_*,\Shv_\CN(\CY))  \\
@V{\sim}VV   @VV{\sim}V \\
\on{ev}^l_\CY\circ ((\Frob_\CY)_*\otimes \on{id})(\on{ps-u}_{\CY,\CN})) @>{\text{\eqref{e:fund trans N}}\circ \text{\eqref{e:Mir diag N}}}>> 
\on{ev}_\CY\circ ((\Frob_\CY)_*\otimes \on{id})(\on{u}_{\CY,\CN})) \\
@VV{\text{\eqref{eq:varepsilon y}}}V @VV{\text{\eqref{e:u N Y}}}V \\
\on{C}_c^\cdot(\CY,-)\circ (\Delta_\CY)^*\circ (\Frob_\CY\times \on{id})_* (\on{ps-u}_{\CY})
@>{\text{\eqref{e:fund trans}} \circ \text{\eqref{e:Mir to id}}}>>
\on{C}_\blacktriangle^\cdot(\CY,-)\circ (\Delta_\CY)^!\circ (\Frob_\CY\times \on{id})_* (\on{u}_{\CY}) 
\endCD
$$

\sssec{}

Finally, we note that the commutative diagram \eqref{e:true and Serre 1} can be rephrased as 
$$
\CD
\on{C}_c^\cdot(\CY,-)\circ (\Delta_\CY)^*\circ (\Frob_\CY\times \on{id})_* (\on{ps-u}_{\CY})
@>{\text{\eqref{e:fund trans}} \circ \text{\eqref{e:Mir to id}}}>>
\on{C}_\blacktriangle^\cdot(\CY,-)\circ (\Delta_\CY)^!\circ (\Frob_\CY\times \on{id})_* (\on{u}_{\CY})  \\
@V{\sim}VV  @VV{\sim}V \\
\on{C}_c^\cdot(\CY^\Frob,\ul\sfe_{\CY^\Frob}) & &  \on{C}^\cdot(\CY^\Frob,\omega_{\CY^\Frob}) \\
@V{\simeq}VV  @VV{\simeq}V  \\
\on{Funct}(\CY(\BF_q))  @>{\on{id}}>>  \on{Funct}(\CY(\BF_q)),
\endCD
$$

Thus, concatenating with the commutative diagram in \secref{sss:Serre diag 3} above, we obtain a commutative diagram
$$
\CD
\Tr((\Frob_\CY)_*,\Shv_\CN(\CY))   @>{\on{id}}>>   \Tr((\Frob_\CY)_*,\Shv_\CN(\CY))  \\
@V{\sim}VV   @VV{\sim}V \\
\on{ev}^l_\CY\circ ((\Frob_\CY)_*\otimes \on{id})(\on{ps-u}_{\CY,\CN})) @>{\text{\eqref{e:fund trans N}}\circ \text{\eqref{e:Mir diag N}}}>> 
\on{ev}_\CY\circ ((\Frob_\CY)_*\otimes \on{id})(\on{u}_{\CY,\CN})) \\
@VV{\text{\eqref{eq:varepsilon y}}}V @VV{\text{\eqref{e:u N Y}}}V \\
\on{C}_c^\cdot(\CY,-)\circ (\Delta_\CY)^*\circ (\Frob_\CY\times \on{id})_* (\on{ps-u}_{\CY})
@>{\text{\eqref{e:fund trans}} \circ \text{\eqref{e:Mir to id}}}>>
\on{C}_\blacktriangle^\cdot(\CY,-)\circ (\Delta_\CY)^!\circ (\Frob_\CY\times \on{id})_* (\on{u}_{\CY}) \\
@V{\sim}VV  @VV{\sim}V \\
\on{C}_c^\cdot(\CY^\Frob,\ul\sfe_{\CY^\Frob}) & &  \on{C}^\cdot(\CY^\Frob,\omega_{\CY^\Frob}) \\
@V{\simeq}VV  @VV{\simeq}V  \\
\on{Funct}(\CY(\BF_q))  @>{\on{id}}>>  \on{Funct}(\CY(\BF_q)),
\endCD
$$
in which the left composite vertical arrow is $\on{LT}^{\on{Serre}}_{\CY,\CN}$, and 
the right composite vertical arrow is $\on{LT}^{\on{true}}_{\CY,\CN}$.

\medskip

This provides the sought-for identification between $\on{LT}^{\on{Serre}}_{\CY,\CN}$ and $\on{LT}^{\on{true}}_{\CY,\CN}$.

\ssec{Proof of \thmref{t:true and Serre geom}} \label{ss:proof geom}

\sssec{}

First, we recall the local version of the Grothendiek-Lefschetz trace formula, following \cite{GV}. 

\medskip

Let $\CY$
be a quasi-compact algebraic stack, and let $\CF\in \Shv(\CY)$ be a \emph{constructible} sheaf, equipped
with a weak Weil structure, i.e., a morphism
$$\alpha:\CF\to (\Frob_\CY)_*(\CF),$$
or equivalently, a morphism
$$\alpha^L:\Frob_\CY^*(\CF)\to \CF.$$

On the one hand, we attach to the pair $(\CF,\alpha^L)$ a function $\on{funct}(\CF)^{\on{naive}}\in \on{Funct}(\CY(\BF_q))$
by the standard procedure of taking the trace of Frobenius on *-fibers of $\CF$ at $\BF_q$-points of $\CF$. I.e., for a Frobenius-invariant point
$$\on{pt} \overset{i_y}\to \CY,$$
we consider the endomorphism
$$i_y^*(\CF)\simeq (\Frob_\CY \circ i_y)^*(\CF)\simeq i_y^*\circ \Frob_\CY^*(\CF)\overset{\alpha^L}\to  i_y^*(\CF),$$
and we set the value of $\on{funct}(\CF)^{\on{naive}}$ at $y\in \CY(\BF_q)$ to be the trace of the above endomorphism. 

\medskip

On the other hand, we can attach to $(\CF,\alpha)$ a function $\on{funct}(\CF)^{\on{true}}\in \on{Funct}(\CY(\BF_q))$, defined
as follows. 

\medskip

Consider the canonical maps
$$\CF\boxtimes \BD^{\on{Verdier}}(\CF)\to (\Delta_\CY)_*(\omega_\CY) \text{ and }
\ul\sfe_\CY\to \CF\sotimes \BD^{\on{Verdier}}(\CF).$$

From the first of these maps we produce the map
\begin{multline} \label{e:Verdier Frob}
\CF\boxtimes \BD^{\on{Verdier}}(\CF)\overset{\alpha\boxtimes \on{id}}\longrightarrow 
(\Frob_\CY\times \on{id})_*(\CF\boxtimes \BD^{\on{Verdier}}(\CF)) \to \\
\to (\Frob_\CY\times \on{id})_*((\Delta_\CY)_*(\omega_\CY))
\simeq (\on{Graph}_{\Frob_\CY})_*(\omega_\CY).
\end{multline} 

The function $\on{funct}(\CF)^{\on{true}}$, viewed as an element of 
$$\on{C}^\cdot(\CY^\Frob,\omega_{\CY^\Frob})\simeq \on{C}^\cdot(\CY,\Delta_\CY^!\circ (\on{Graph}_{\Frob_\CY})_*(\omega_\CY)),$$
corresponds to the map
$$\ul\sfe_\CY \to \CF\sotimes \BD^{\on{Verdier}}(\CF) =\Delta_\CY^!(\CF\boxtimes \BD^{\on{Verdier}}(\CF)) 
\overset{\text{\eqref{e:Verdier Frob}}}\longrightarrow \Delta_\CY^!\circ (\on{Graph}_{\Frob_\CY})_*(\omega_\CY).$$

\medskip

The local Grothendiek-Lefschetz trace formula says:

\begin{thm} \label{t:GR}
The functions $\on{funct}(\CF)^{\on{naive}}$ and $\on{funct}(\CF)^{\on{true}}$ are equal.
\end{thm} 

\begin{rem}
When $\CF$ is \emph{compact}, the assertion of \thmref{t:GR} is a particular case of
that of \thmref{t:gv}. Namely, the functions $\on{funct}(\CF)^{\on{naive}}$ and $\on{funct}(\CF)^{\on{true}}$
are the values of the maps $\on{LT}_\CY^{\on{naive}}$ and $\on{LT}_\CY^{\on{true}}$ on the element
$$\on{cl}(\CF,\alpha)\in \Tr((\Frob_\CY)_*,\Shv(\CY)),$$
respectively.

\medskip

For $\CF$ which is constructible but not compact, the assertion of \thmref{t:GR} can be obtained 
by proving a version of \thmref{t:gv} for the \emph{renormalized} version of the category $\Shv(\CY)$,
namely, one obtained as the ind-completion of the constructible subcategory of $\Shv(\CY)$.

\end{rem} 

\sssec{} \label{sss:no triangle}

We precede the proof of \thmref{t:true and Serre geom} by the following observation. 

\medskip

Let
$\CQ$ be a \emph{constructible} object of $\Shv(\CY\times \CY)$. Note that the procedure 
in \secref{sss:construct psi map} defines a map
\begin{equation} \label{e:fund trans again}
\psi: \on{C}_c^\cdot(\CY, \Delta_\CY^* \circ (\on{Id}\otimes \Mir_\CY)(\CQ)) \to \on{C}^\cdot(\CY,\Delta^!_\CY(\CQ)).
\end{equation} 

\medskip

It is easy to see that the map \eqref{e:fund trans again} equals the composition of the value of the natural
transformation \eqref{e:fund trans}, followed by the canonical map
\begin{equation} \label{e:no triangle}
\on{C}_\blacktriangle^\cdot(\CY,\Delta^!_\CY(\CQ))\to \on{C}^\cdot(\CY,\Delta^!_\CY(\CQ)).
\end{equation} 

\sssec{}

We are ready to launch the proof of \thmref{t:true and Serre geom}. We apply the observation in 
\secref{sss:no triangle} to
$$\CQ:=(\on{Graph}_{\Frob_\CY})_*(\omega_\CY).$$

Note that in this case, the map \eqref{e:no triangle} is an isomorphism, as the corresponding map identifies with
$$\on{C}^\cdot_\blacktriangle(\CY^\Frob,\omega_{\CY^\Frob})\to \on{C}^\cdot(\CY^\Frob,\omega_{\CY^\Frob}).$$

\medskip

We need to show that a certain map 
$$\on{Funct}(\CY(\BF_q))=\on{C}^\cdot_c(\CY^\Frob,\ul\sfe_{\CY^\Frob})\to \on{C}^\cdot(\CY^\Frob,\omega_{\CY^\Frob})\simeq \on{Funct}(\CY(\BF_q))$$
equals the identity, where the middle arrow is the result of the construction in \secref{sss:construct psi map} applied to the above choice of $\CQ$.

\medskip

We interpret the above map as a functional
\begin{equation} \label{e:functional}
\on{Funct}(\CY(\BF_q))\otimes \on{Funct}(\CY(\BF_q)) \simeq \on{C}^\cdot_c(\CY^\Frob,\ul\sfe_{\CY^\Frob})\otimes \on{C}^\cdot_c(\CY^\Frob,\ul\sfe_{\CY^\Frob})\to \sfe,
\end{equation} 
and we wish to show that this functional is given by
$$f_1,f_2\mapsto \underset{y\in \CY(\BF_q)}\Sigma\, \frac{1}{|\Aut(y)(\BF_q)|}\cdot 
f_1(y)\cdot f_2(y).$$

\sssec{}

Let us unwind the construction of the functional \eqref{e:functional}. We start with  the map \eqref{e:key serre map} for the above choice of $\CQ$. 
This is a map
\begin{equation} \label{e:key serre map Frob 0}
\left((\on{Graph}_{\Frob_\CY})_!(\ul\sfe_\CY)\boxtimes (\on{Graph}_{\Frob_\CY})_!(\ul\sfe_\CY)\right)^{\sigma_{2,3}}\to 
(\Delta_\CY)_*(\omega_\CY)\boxtimes (\Delta_\CY)_!(\ul\sfe_\CY).
\end{equation} 

We apply to this map the functor 
$$(p_{2,3})_!\circ (\Delta_\CY\times \on{id}\times \on{id})^* \circ \sigma_{2,3}:
\Shv(\CY\times \CY\times \CY\times \CY)\to \Shv(\CY\times \CY),$$
and we obtain a map
$$\on{C}^\cdot_c(\CY^\Frob,\ul\sfe_{\CY^\Frob}) \otimes (\on{Graph}_{\Frob_\CY})_!(\ul\sfe_\CY) \to (\Delta_\CY)_*(\omega_\CY).$$
We apply to the latter map the adjunction
$$\on{C}_c^\cdot(\CY,-) \circ \Delta_\CY^*:\Shv(\CY\times \CY) \rightleftarrows \Vect:(\Delta_\CY)_*(\omega_\CY),$$
and we obtain the desired pairing
$$\on{C}^\cdot_c(\CY^\Frob,\ul\sfe_{\CY^\Frob})\otimes \on{C}^\cdot_c(\CY^\Frob,\ul\sfe_{\CY^\Frob})\to \sfe.$$

\sssec{}

Let us apply Verdier duality to the above morphisms. The dual of \eqref{e:key serre map Frob 0} is a morphism
\begin{equation} \label{e:key serre map Frob 1}
(\Delta_\CY)_!(\ul\sfe_\CY)\boxtimes (\Delta_\CY)_*(\omega_\CY)\to 
\left((\on{Graph}_{\Frob_\CY})_*(\omega_\CY)\boxtimes (\on{Graph}_{\Frob_\CY})_*(\omega_\CY)\right)^{\sigma_{2,3}}.
\end{equation} 

From this morphism we obtain an element of
$$\on{C}^\cdot(\CY^\Frob,\omega_{\CY^\Frob})\otimes \on{C}^\cdot(\CY^\Frob,\omega_{\CY^\Frob})$$
by the following procedure. 

\medskip

We apply to \eqref{e:key serre map Frob 1} the functor 
$$(p_{2,3})_*\circ (\Delta_\CY\times \on{id}\times \on{id})^! \circ \sigma_{2,3}:
\Shv(\CY\times \CY\times \CY\times \CY)\to \Shv(\CY\times \CY),$$
and we obtain a map
$$(\Delta_\CY)_!(\ul\sfe_\CY) \to \on{C}^\cdot(\CY^\Frob,\omega_{\CY^\Frob})\otimes (\on{Graph}_{\Frob_\CY})_*(\omega_\CY).$$

Applying the adjunction
$$(\Delta_\CY)_!(\ul\sfe_\CY):\Vect\rightleftarrows \Shv(\CY\times \CY):\on{C}^\cdot(\CY,-) \circ \Delta_\CY^!,$$
we obtain the desired element of
\begin{equation} \label{e:elt in prod}
\on{C}^\cdot(\CY^\Frob,\omega_{\CY^\Frob})\otimes \on{C}^\cdot(\CY^\Frob,\omega_{\CY^\Frob})\simeq
\on{Funct}(\CY(\BF_q))\otimes \on{Funct}(\CY(\BF_q)).
\end{equation}

We wish to show that the resulting function is the characteristic function of the diagonal, i.e., its value on a
$(y_1,y_2)\in \CY(\BF_q)\times \CY(\BF_q)$ equals the cardinality of the set of isomorphisms between
the corresponding two points of the groupoid $\CY(\BF_q)$. 

\sssec{}

However, unwinding the definitions, we obtain that the map \eqref{e:key serre map Frob 1} identifies with the map
\eqref{e:Verdier Frob} for $\CF=(\Delta_\CY)_!(\ul\sfe_\CY)$ and $\alpha$ being the tautological map $\alpha_{\on{taut}}$ 
$$(\Delta_\CY)_!(\ul\sfe_\CY) \simeq (\Frob_{\CY\times \CY})_!\circ (\Delta_\CY)_!(\ul\sfe_\CY) \simeq 
(\Frob_{\CY\times \CY})_*\circ (\Delta_\CY)_!(\ul\sfe_\CY).$$

\medskip

From here, we obtain that the element in \eqref{e:elt in prod} constructed above equals
$$\on{funct}^{\on{true}}((\Delta_\CY)_!(\ul\sfe_\CY),\alpha_{\on{taut}}).$$

Applying \thmref{t:GR}, we obtain that the above element equals 
$$\on{funct}^{\on{naive}}((\Delta_\CY)_!(\ul\sfe_\CY),\alpha_{\on{taut}}).$$

\medskip

Now, the classical Grothendiek-Lefschetz trace formula about the compatibility of the assignment
$$(\CF,\alpha) \rightsquigarrow \on{funct}^{\on{naive}}(\CF,\alpha)$$
with the !-pushforward functor implies that the above function equals the direct image (=sum along the fibers) of the constant function
along the map
$$\CY(\BF_q)\to (\CY\times \CY)(\BF_q)\simeq  \CY(\BF_q)\times \CY(\BF_q),$$
as required. 

\qed[\thmref{t:true and Serre geom}]

\end{document}